\documentclass[12pt,a4paper,reqno]{amsart}

\usepackage[usenames,dvipsnames]{color}

\usepackage{tikz,graphicx}
\usepackage[all,arc,curve,color,frame,graph,matrix,cmtip,poly]{xy}

\usetikzlibrary{decorations.markings}
\tikzstyle directed=[postaction={decorate,decoration={markings,
    mark=at position .65 with {\arrow{latex}}}}]

\usepackage{amsfonts,amsmath,amssymb,color,amscd,amsthm}
\usepackage[T1]{fontenc}

\usepackage{mathrsfs}

\usepackage{fancyhdr}
\usepackage{enumerate}

\makeatletter
\renewcommand\p@enumii{}
\makeatother

\usepackage{ulem}

\usepackage{cases}

\usepackage{units}

\usepackage[]{hyperref}

\usepackage{hyperref}
\hypersetup{
    colorlinks=true,
    allcolors=violet,
    linktocpage=true
}

\theoremstyle{plain}
 \newtheorem{theorem}{Theorem}[section]
 \newtheorem{proposition}[theorem]{Proposition}
 \newtheorem{lemma}[theorem]{Lemma}
 \newtheorem{corollary}[theorem]{Corollary}
 \newtheorem{definition}[theorem]{Definition}
 \newtheorem{example}[theorem]{Example}
 \newtheorem{conjecture}[theorem]{Conjecture}
 \newtheorem{remark}[theorem]{Remark}
 \newtheorem{question}[theorem]{Question}

\usepackage{stackengine}

\newcommand\A{{\mathbb A}}
\newcommand\C{{\mathbb C}}

\newcommand\F{{\mathbb F}}

\newcommand\K{{\mathbb K}}
\renewcommand\L{{\mathbb L}}

\newcommand\p{{\mathbb P}}

\newcommand\T{{\mathbb T}}
\newcommand\U{{\mathbb U}}
\newcommand\Z{{\mathbb Z}}

\newcommand{\CC}{{\mathfrak C}}
\newcommand{\FF}{{\mathfrak F}}

\newcommand{\LL}{{\mathfrak L}}

\DeclareMathAlphabet{\mathpzc}{OT1}{pzc}{m}{it}
\newcommand{\g}{\mathpzc{g}}

\DeclareMathOperator{\BBB}{\mathcal{B}}

\DeclareMathOperator{\DDD}{\mathcal D}

\DeclareMathOperator{\So}{S_{\rm odd}}

\renewcommand\k{\mathrm{k}}

\newcommand{\MM}{{\mathfrak M}}

\DeclareMathOperator{\Vect}{Vect}

\DeclareMathOperator{\I}{I}

\newcommand{\smat}[4]{\left( \begin{smallmatrix} #1\,&#2\\ #3\,&#4 \end{smallmatrix} \right)}

\newcommand\dis{{\displaystyle}}

\newcommand\restr[2]{{
  \left.\kern-\nulldelimiterspace 
  #1 
  \right|_{#2} 
  }}

\newcommand\Bir{{\mathrm{Bir}}}
\newcommand\id{{\mathrm{id}}}

\newcommand\Aff{{\mathrm{Aff}}}

\DeclareMathOperator{\Mat}{M}
\newcommand\GL{{\mathrm{GL}}}

\newcommand\PGL{{\mathrm{PGL}}}

\DeclareMathOperator{\Fix}{Fix}

\DeclareMathOperator{\Norm}{N}

\DeclareMathOperator{\degmax}{d}

\newcommand{\Jonq}{{\rm Jonq}}

\newcommand\Aut{{\mathrm{Aut}}}

\renewcommand{\phi}{{\varphi}}

\newcommand{\Com}{{\mathcal C}{\mathfrak o}{\mathfrak m}}

\DeclareMathOperator{\Exc}{Exc}
\DeclareMathOperator{\Stab}{Stab}

\DeclareMathOperator{\vv}{v}

\DeclareMathOperator{\rk}{rk}
\DeclareMathOperator{\tr}{tr}
\DeclareMathOperator{\BaumBott}{BB}
\DeclareMathOperator{\Div}{Div}
\DeclareMathOperator{\divisor}{div}
\DeclareMathOperator{\NFC}{NFC}

\DeclareMathOperator{\Image}{Im}
\DeclareMathOperator{\pr}{pr}
\DeclareMathOperator{\Cent}{Cent}
\DeclareMathOperator{\Nor}{N}
\DeclareMathOperator{\length}{\ell}

\newcommand{\bp}{{\rm Bir}({\mathbb P}^2)}
\newcommand{\bpn}{{\rm Bir}({\mathbb P}^n)}
\newcommand{\ban}{{\rm Bir}({\mathbb A}^n)}
\newcommand{\bpGothic}{{\mathfrak B}{\mathfrak i}{\mathfrak r}({\mathbb P}^2)}
\newcommand{\bpnGothic}{{\mathfrak B}{\mathfrak i}{\mathfrak r}({\mathbb P}^n)}

\newcommand{\rn}{{\rm Rat}({\mathbb P}^n)}

\newcommand{\rrn}{{\mathfrak R}{\mathfrak a}{\mathfrak t}({\mathbb P}^n)}
\newcommand{\JJonq}{{\mathfrak J}{\mathfrak o}{\mathfrak n}{\mathfrak q}}

\usepackage{mathtools}

\title[Borel subgroups of the plane Cremona group]{Borel subgroups of the plane Cremona group}

\author{Jean-Philippe Furter}
\address{J.-P.\ Furter, Univ. Bordeaux, IMB, CNRS, UMR 5251, F-33400 Talence, France}
\email{jean-philippe.furter@math.u-bordeaux.fr}

\author{Isac Hed\'en}
\address{Dept.\ of Mathematics, Royal Institute of Technology (KTH), SE-100\,44 Stockholm, Sweden and Mathematics Institute, University of Warwick, Coventry, CV4 7AL, England}\email{isach@kth.se}

\thanks{The second author gratefully acknowledges support from the Knut and Alice Wallenberg Foundation, grant number KAW2016.0438.}

\begin{document}

\maketitle
\begin{abstract}
It is well known that all Borel subgroups of a linear algebraic group are conjugate. Berest, Eshmatov, and Eshmatov have shown that this result also holds for the automorphism group $\Aut (\A^2)$ of the affine plane. In this paper, we describe all Borel subgroups of the complex Cremona group $\bp$ up to conjugation, proving in particular that they are not necessarily conjugate. In principle, this fact answers a question of Popov. More precisely, we prove that $\bp$ admits Borel subgroups of any rank $r \in \{ 0,1,2 \}$ and that all Borel subgroups of rank $r \in \{ 1,2 \}$ are conjugate. In rank $0$, there is a $1-1$ correspondence between conjugacy classes of Borel subgroups of rank $0$ and hyperelliptic curves of genus $\g \geq 1$. Hence, the conjugacy class of a rank $0$ Borel subgroup admits two invariants: a discrete one, the genus $\g$, and a continuous one, corresponding to the coarse moduli space of hyperelliptic curves of genus $\g$. This moduli space is of dimension $2 \g-1$.
\end{abstract}

\tableofcontents

\newpage

\section{Introduction}

Let $\L$ be an algebraically closed field. The famous Lie--Kolchin theorem asserts that any closed connected solvable subgroup of $\GL_n (\L)$ is triangularisable, i.e.\ conjugate to a subgroup all of whose elements are upper-triangular. More generally, let $G$ be a linear algebraic group defined over $\L$. Define a Borel subgroup of $G$  as a maximal subgroup among the closed connected solvable subgroups. Then, Borel has shown that  any closed connected solvable subgroup of $G$ is contained in a Borel subgroup and that all Borel subgroups of $G$ are conjugate. By \cite{BerestEshmatovEshmatov2016} (see also \cite{FurterPoloni2018}), the same result still holds for the automorphism group $\Aut (\A^2)$ of the affine plane.

The Cremona group ${\rm Bir}(\p^n)$ is the group of all birational transformations of the $n$-dimensional complex projective space $\p^n$. From an algebraic point of view, it corresponds to the group of $\C$-automorphisms of the field  $\C (x_1, \ldots, x_n)$. This group is naturally endowed with the Zariski topology introduced by Demazure \cite{Demazure1970} and Serre \cite{Serre2010}. We describe this topology in Section~\ref{section: Zariski Topology on Bir(pn)}; see in particular Definition~\ref{definition: Zariski topology}. For more details on this subject we refer to \cite{BlancFurter2013}.  Following Popov we define the Borel subgroups of $\bpn$ as the maximal closed connected solvable subgroups with respect to this topology \cite{Popov2017}. Borel subgroups of a closed subgroup $G$ of $\bpn$ are defined analogously.

An element of $\bpn$ is of the form
\[ f \colon [x_1: \dots : x_{n+1} ] \dasharrow [ F_1 (x_1,\ldots,x_{n+1}) : \dots : F_{n+1} (x_1 ,\ldots,x_{n+1} ) ], \]
where the $F_i \in \C [x_1, \ldots, x_{n+1} ]$ are homogeneous polynomials of the same degree. We then often write $f= [F_ 1 : \dots : F_{n+1} ]$. Using the open embedding
\[ \A^n \hookrightarrow \p^n, \quad (x_1, \ldots, x_n ) \mapsto [ x_1 : \dots : x_n : 1],\]
we have a natural isomorphism $\ban \simeq \bpn$. This allows us to write $f$ in affine coordinates as
\[ f \colon (x_1, \dots, x_n) \dasharrow ( f_1 (x_1, \dots,x_n), \dots, f_n (x_1, \dots,x_n) ) , \]
where the $f_i \in \C ( x_1, \ldots, x_n )$ are rational functions. We then often write $f= (f_1, \ldots, f_n)$.

Let $\mathcal B_n \subseteq \bpn$ be the subgroup of birational transformations of $\p^n$ of the form $f= (f_1, \ldots, f_n)$, $f_ i = a_i x_i + b_i$, with $a_i ,b_i \in \C (x_{i+1}, \ldots, x_n )$ and $a_i \neq 0$. The following result is proven by Popov \cite[Theorem 1]{Popov2017}.

\begin{theorem}
The group $\mathcal B_n$ is a Borel subgroup of $\Bir (\p^n)$.
\end{theorem}

In the same paper he also makes the following conjecture.

\begin{conjecture}
For $n \ge 5$, the Cremona group ${\rm Bir}({\mathbb P}^n)$ contains nonconjugate Borel subgroups.
\end{conjecture}

The main result of our paper is the description of all Borel subgroups of $\bp$ up to conjugation. Before giving the precise statement a few definitions are needed.

The Jonqui\`eres group $\Jonq$ is defined as the group of birational transformations of $\p^2$ preserving the pencil of lines passing through the point $[1:0:0] \in \p^2$. In affine coordinates, it is the group of birational transformations of the form
\[ (x,y) \dasharrow \left( \frac{\alpha (y) x + \beta (y)}{\gamma (y) x + \delta (y)},  \frac{ay+b}{cy+d} \right),\]
where $\smat{\alpha }{\beta}{\gamma}{\delta}  \in \PGL_2(\C(y) )$,  $\smat{a }{b}{c}{d}  \in \PGL_2$. 
Hence we have
\[ \Jonq = \PGL_2 ( \C (y) ) \rtimes \PGL_2 \; \supseteq \;  \Aff_1 ( \C (y) ) \rtimes \Aff_1 = \BBB_2 .\]

Here and in the rest of this paper we write $\PGL_2$ and $\Aff_1$ instead of $\PGL_2 (\C)$ and $\Aff_1 (\C)$. Also, all the semidirect products considered will be inner semidirect products. This means we will write $G = N \rtimes H$ when $N,H$ are subgroups of the (abstract) group $G$ which satisfy the three following assertions:
\begin{enumerate}
\item
$N$ is a normal subgroup $N \triangleleft G$;
\item
$G = NH$;
\item
$N \cap H = \{ 1 \}$ (where $1$ denotes the identity element of $G$).
\end{enumerate}

For any nonsquare element $f$ of $\C (y)$ we define the subgroup
\[ \T_f :=  \left\{ \smat{a }{bf}{b}{a}, \; a,b \in \C (y), \;  (a,b) \neq (0,0) \right\} \]
of $\PGL_2 ( \C (y) ) \subseteq \Jonq$ and for any coprime integers $p,q \in \Z$, we define the $1$-dimensional torus
\[ \T_{p,q} := \{ (t^p x, t^q y), \; t \in \C^* \} \;  \subseteq  \; \Jonq .\]
We make the following two conventions. When talking about the genus of a complex curve, we will  always mean the genus of the associated smooth projective curve. If $f$ is a nonsquare element of $\C  (y)$, when talking about the hyperelliptic curve associated with $x^2 = f(y)$, we will always mean the smooth projective curve whose function field is equal to $\C (y) [ \sqrt{f} ]$.

Note that we allow hyperelliptic curves of genus $0$ and $1$.

We can now state the principal result of our paper. 

\begin{theorem}[Main Theorem] \label{theorem: main theorem}
Up to conjugation, any Borel subgroup of $\bp$ is one of the following groups:
\begin{enumerate}
\item \label{Borel: B2-introduction}
$\BBB_2$;
\item \label{Borel: Ty semidirect T-introduction}
$\T_y \rtimes \T_{1,2}$;
\item \label{Borel: Tf-introduction}
$\T_f$, where $f$ is a nonsquare element of $\C (y)$ such that the genus $\mathpzc{g}$ of the hyperelliptic curve associated with $x^2 =f(y)$ satisfies $\mathpzc{g} \geq 1$.
\end{enumerate}

Moreover these three cases are mutually disjoint and in case \eqref{Borel: Tf-introduction} the Borel subgroups $\T_f$ and $\T_g$ are conjugate if and only if the hyperelliptic curves associated with  $x^2=f(y)$ and $x^2=g(y)$ are isomorphic.
\end{theorem}

If $G$ is an abstract group, we denote by $D(G)$ its derived subgroup. It is the subgroup generated by all commutators $[g, h] := ghg^{-1}h^{-1}$, $g,h \in G$. The $n$-th derived subgroup of $G$ is then defined inductively by $D^0(G) := G$ and $D^n (G) := D ( D^{n-1} (G) )$ for $n \ge 1$. A group $G$ is called solvable if $D^n(G) = \{1 \}$ for some integer $n \ge 0$. The smallest such integer $n$ is called the derived length of $G$ and is denoted $\length (G)$. 
Recall that the subgroup of upper triangular matrices in $\GL_n (\C)$ is solvable and has
derived length $\lceil \log_2 (n) \rceil +1$, where $\lceil x \rceil $ denotes the smallest integer greater than or equal to the real number $x$ (see e.g.\ \cite[page 16]{Wehrfritz1973}). Also, the subgroup of upper triangular automorphisms in $\Aut ( \A^n_{\C} )$ has derived length $n+1$ (see \cite[Lemma 3.2]{FurterPoloni2018}). In contrast, we will prove the following result in Appendix \ref{section: computation of the derived length of Bn}.

\begin{proposition} \label{proposition: the derived length of Bn}
The derived length of the Borel subgroup $\mathcal B_n$ of $\bpn$ is equal to~$2n$.
\end{proposition}

As usual we let the rank $\rk (G)$ of a complex linear algebraic group $G$ be the maximal dimension $d$ of an algebraic torus $(\C^*)^d$ in $G$.  Analogously, the rank $\rk (G)$ of a closed subgroup $G$ of $\Bir (\p^n)$ is defined as the maximal dimension of an algebraic  torus in $G$. The following result is proven by Bia{\l}ynicki-Birula \cite[Corollary 2, page 180]{Bialynicki-Birula1966} (see also \cite[Theorem 1 (i)]{Popov2013}).

\begin{theorem} \label{theorem: tori in Bir(Pn)}
All algebraic tori in $\Bir (\p^n)$ are of rank $\leq n$. Moreover, all algebraic tori of a given rank $\geq n-2$ are conjugate in $\Bir (\p^n)$.
\end{theorem}

Hence, we have $\rk (\Bir (\p^n) ) = \rk ( \BBB_n) =n$, and any closed subgroup $G$ of $\Bir (\p^n)$ satisfies $\rk (G) \leq n$.

The derived lengths and ranks of the three kinds of Borel subgroups of $\bp$ are given in the following table (see Remark~\ref{remark: lengths and ranks of Borel subgroups} for the computations):

\vspace{4mm}

\centerline{
\begin{tabular}{|c|c|c|}
\hline
Type of Borel subgroup & Derived length & Rank   \\
\hline
$\BBB_2$ & $4$  & $2$ \\
\hline
$\T_y \rtimes \T_{1,2}$ & $2$ & $1$ \\
\hline
$\T_f$ & $1$ & $0$  \\
\hline
\end{tabular}
}

\vspace{4mm}

Theorem~\ref{theorem: main theorem} directly gives the following result:

\begin{corollary}
All Borel subgroups of maximal rank 2 of $\bp$ are conjugate.
\end{corollary}

More generally the following question seems natural.

\begin{question}
Are all Borel subgroups of maximal rank $n$ of $\bpn$ conjugate?
\end{question}

In view of Theorem~\ref{theorem: main theorem} we believe that the following slightly strengthened version of Popov's conjecture should hold.

\begin{conjecture}
For $n \ge 2$, the Cremona group ${\rm Bir}({\mathbb P}^n)$ contains nonconjugate Borel subgroups (note that Theorem~\ref{theorem: main theorem} establishes the case $n=2$).
\end{conjecture}

Our article is organised as follows: In Section~\ref{section: Zariski Topology on Bir(pn)} we outline the construction of the Zariski topology on $\bpn$, following \cite[$\S 5.2$]{BlancFurter2018}.
We also establish various results to be used later on.
In Section~\ref{section: Any Borel subgroup of Bir(P2) is conjugate to a Borel subgroup of Jonq}, we prove that any Borel subgroup of $\bp$ is conjugate to a subgroup of the Jonqui\`eres group $\Jonq$ (Theorem~\ref{theorem: a closed connected solvable subgroup of Bir(P2) is conjugate to a subgroup of Jonq}). This key result heavily relies on Urech's nice paper \cite{Urech2018}. Then, in Section~\ref{section: Any Borel subgroup of Jonq is conjugate to a Borel subgroup of PGL_2(C(y)) rtimes Aff_1}, we prove that any Borel subgroup of $\Jonq$ is conjugate to a subgroup of $\PGL_2 ( \C(y) ) \rtimes \Aff_1$ (Theorem~\ref{theorem: any closed connected solvable of Jonq is conjugate to a subgroup of PGL_2(C(y)) rtimes Aff_1}). These two results directly imply that any Borel subgroup of $\bp$ is conjugate to a (Borel) subgroup of $\PGL_2 ( \C(y) ) \rtimes \Aff_1$ (Theorem~\ref{theorem: any closed connected solvable subgroup of Bir(P2) is conjugate to a subgroup of PGL_2(C(y)) rtimes Aff_1}). Surprisingly enough, we will see that a Borel subgroup of $\PGL_2 ( \C(y) ) \rtimes \Aff_1$ is not necessarily a Borel subgroup of $\bp$ (Example~\ref{example: a Borel subgroup of PGL_2(C(y)) rtimes PGL_2 which is not a Borel subgroup of Bir(P2)} of Appendix~\ref{section: Borel subgroups of PGL_2(C(y)) rtimes PGL_2}). In Section~\ref{section: K-Borel subgroups of PGL(2,K)}, we define the $\K$-Borel subgroups of $\PGL_2 ( \K )$ for any field $\K$ of characteristic zero in Definition~\ref{definition: K-Borel subgroup of G(K)} and we describe them in Theorem~\ref{theorem: K-Borel subgroups of PGL(2,K)}. This result is then used in Section~\ref{section: Borel subgroups of PGL(2,C(y))} where all Borel subgroups of the closed subgroup $\PGL_2 ( \C (y) )$ of $\bp$ are  described in Theorem~\ref{theorem : the Borel subgroups of PGL(2,C(y))}. It turns out that these Borel subgroups coincide with the $\C (y)$-Borel subgroups of $\PGL_2 ( \C (y) )$! At the end of Section~\ref{section: Borel subgroups of PGL(2,C(y))} we then prove that the maximal derived length of a closed connected solvable subgroup of $\bp$ is  $4$ (Lemma~\ref{lemma: maximal derived length of a closed connected solvable subgroup of Bir(P2)}), and deduce from this result that any closed connected solvable subgroup of $\bp$ is contained in a Borel subgroup of $\bp$ (Proposition~\ref{proposition: pre-Borel subgroups of closed subgroups of Bir(P2) are contained in Borel subgroups}). In Section~\ref{section: The groups TT_f} we study the groups $\T_f \subseteq \bp$. In Proposition~\ref{proposition: equivalent conditions for T_f and T_g to be conjugate in Bir(P2)} we give different equivalent conditions characterising the fact that $\T_f$ and $\T_g$ are conjugate in $\bp$ and in Proposition~\ref{proposition: connected component of the normaliser of Tf in Jonq} we compute the neutral connected component $ \Nor_{\Jonq}( \T_f )^{\circ}$ of the normaliser of $\T_f$ in $\Jonq$. In Section~\ref{section: Subgroups of PGL_2( C(y)) rtimes Aff_1 conjugate to T_{0,1}} we show that an algebraic subgroup $G$ of $\PGL_2 ( \C (y) ) \rtimes \Aff_1$ isomorphic to $\C^*$ is conjugate to $\T_{0,1}$ if and only if the second projection $\pr_2 \colon \PGL_2 ( \C (y) ) \rtimes \Aff_1 \to \Aff_1$ induces an isomorphism $G \to \pr_2 (G)$ (see Lemma~\ref{lemma: characterisation of the groups conjugate to T_{0,1}}). In Section~\ref{section: Embeddings of C+ into JJ} we show that up to conjugation the additive group $(\C , +)$ admits exactly two embeddings in $\Jonq$ (Proposition~\ref{proposition: embeddings of (C,+) in Jonq}). The main result of Section~\ref{section: Borel subgroups of PGL2C(y) and of PGL2C(y) rtimesAff1} is Theorem~\ref{theorem: any Borel subgroup of PGL_2 (C (y) ) rtimes Aff_1 contains a Borel subgroup of PGL_2 (C (y) )}, asserting that any Borel subgroup of $\PGL_2(\C(y))\rtimes\Aff_1$ contains at least one Borel subgroup of $\PGL_2(\C(y))$. Then in Theorem~\ref{theorem: the natural bijection between the Borel subgroups of PGL_2 ( C(y) ) and the Borel subgroups of PGL_2 ( C(y) ) rtimes Aff_1} we show that any Borel subgroup $B'$ of $\PGL_2(\C(y))\rtimes\Aff_1$ actually contains a unique Borel subgroup $B$ of $\PGL_2(\C(y))$ and that the corresponding map $B'\mapsto B$ defines a bijection from the set of Borel subgroups of $\PGL_2(\C(y))\rtimes\Aff_1$ to the set of Borel subgroups of $\PGL_2(\C(y))$. Finally in Section~\ref{section: some Borel subgroups of Bir(P2)} we show that all the subgroups listed in Theorem~\ref{theorem: main theorem} are actually Borel subgroups of $\bp$, cf.\ Theorem~\ref{theorem: some Borel subgroups of Bir(P2)}, and in Section~\ref{section: all Borel subgroups of Bir(P2)} we show that up to conjugation there are no others, cf.\ Theorem~\ref{theorem: all Borel subgroups of Bir(P2)}. These two sections contain the two following additional results: Any Borel subgroup of $\Jonq$ is a Borel subgroup of $\bp$  (Proposition~\ref{proposition: a Borel subgroup of Jonq is also a Borel subgroup of Bir(P2)}); if $B$ is a Borel subgroup of $\bp$, then we have $B= \Nor_{\bp} (B) ^{\circ}$ (Proposition~\ref{proposition: a variant of Borel normaliser theorem for Bir(P2)}; this statement is an analog of the usual Borel normaliser theorem which asserts that $B= \Nor_{G} (B)$ when $B$ is a Borel subgroup of a linear algebraic group $G$).

Our paper also contains two appendices. In Appendix~\ref{section: computation of the derived length of Bn} we show that the derived length of $\BBB_n$ is equal to $2n$. In Appendix~\ref{section: Borel subgroups of PGL_2(C(y)) rtimes PGL_2} we give an example of a Borel subgroup of $\PGL_2(\C(y))\rtimes \Aff_1$ which is not a Borel subgroup of $\bp$ -- even if we have shown that any Borel subgroup of $\bp$ is conjugate to a Borel subgroup of $\PGL_2(\C(y))\rtimes\Aff_1$!

\section{The Zariski topology on $\Bir(\p^n)$} \label{section: Zariski Topology on Bir(pn)}

Following \cite{Demazure1970,Serre2010}, the notion of families of birational maps is defined, and used in Definition \ref{definition: Zariski topology} for describing the natural Zariski topology on $\Bir( W )$ where $W$ is an irreducible complex algebraic variety.

\begin{definition} \label{definition: morphism to Bir(W)}
Let $A,W$ be irreducible complex algebraic varieties, and let $f$ be an $A$-birational map of the $A$-variety $A \times W$, inducing an isomorphism $U \to V$, where $U,V$ are open subsets of $A \times W$, whose projections on $A$ are surjective.

The birational map $f$ is given by $(a,w) \dasharrow (a,p_2(f(a,w)))$, where $p_2$ is the second projection, and for each $\C$-point $a \in A$, the birational map $ w \dasharrow p_2( f(a, w) )$ corresponds to an element  $f_a \in \Bir( W )$. The map $a \mapsto f_a$ represents a map from $A$ $($more precisely from the $\C$-points of $A)$ to $\Bir( W )$, and will be called a morphism from $A$ to $\Bir( W )$.
\end{definition}

\begin{definition}  \label{definition: Zariski topology}
A subset $F \subseteq \Bir(W)$ is closed in the Zariski topology if for any algebraic variety $A$ and any morphism $A \to \Bir(W)$ the preimage of $F$ is closed.
\end{definition}

Recall that a birational transformation $f$ of $\p^n$ is given by
\[ f \colon [x_1: \dots :x_{n+1} ] \dasharrow [ f_1 (x_1,\ldots,x_{n+1}) : \dots : f_{n+1} (x_1,\ldots,x_{n+1}) ], \]
where the $f_i$ are homogeneous polynomials of the same degree. Choosing the $f_i$ without common component, the degree of $f$ is the degree of the $f_i$. If $d$ is a positive integer, we set $\bpn_d:= \{  f \in \bpn, \; \deg (f) \leq d \}$. We will use the following result, which is \cite[Proposition 2.10]{BlancFurter2013}:

\begin{lemma}  \label{lemma:description-of-the-topology-of-Bir(Pn)-as-an-inductive-limit}
A subset $F \subseteq \bpn$ is closed if and only if $F \cap \bpn_d$ is closed in $\bpn_d$ for any positive integer $d$.
\end{lemma}

\begin{remark} \label{remark: strengthening of a criterion to check that a subset of Bir(Pn) is closed}
Since $\bpn_d$ is closed in $\bpn$ (see \cite[Corollary 2.8]{BlancFurter2013}), a subset $F$ of $\bpn$ is closed if and only if there exists a positive integer $D$ such that $F \cap \bpn_d$ is closed in $\bpn_d$ for any $d \ge D$.
\end{remark}

We will now describe the topology on $\bpn_d$. A convenient way to handle this topology is through the map $\pi_d \colon \bpnGothic_d \to \bpn_d$ that we introduce in the next definition and whose properties are given in Lemma~\ref{lemma:bridge-to-an-algebraic-variety} below.

Let's now fix the integer $d \geq 1$. We will henceforth use the following notation:

\begin{definition} \label{definition: Gothic Rat(d) and Bir(d)}
Denote by  $\rrn_d$ the projective space associated with the complex vector space of $(n+1)$-tuples $(f_1,\dots,f_{n+1})$ where all $f_i \in \C [x_1,\dots,x_{n+1}]$ are homogeneous polynomials of degree $d$. The equivalence class of $(f_1,\dots,f_{n+1})$ will be denoted by $[f_1,\dots,f_{n+1}]$.

For each $f=[f_1,\dots,f_{n+1}]\in \rrn_d$, we denote by $\psi_f$ the rational map $\p^n \dasharrow \p^n$ defined by
\[  [x_1:\dots:x_{n+1}] \dasharrow [f_1 (x_1,\dots,x_{n+1}): \dots : f_{n+1} (x_1,\dots,x_{n+1}) ].\]
Writing $\rn$ for the set of rational maps from $\p^n$ to $\p^n$ and setting
\[ \rn_d:=\{ h \in \rn, \; \deg (h) \leq d \},\]
we obtain a surjective map
\[ \Psi_d \colon \rrn_d \to \rn_d, \quad f \mapsto \psi_f.\]
This map induces a surjective map $\pi_d \colon \bpnGothic_d \to \bpn_d$, where $\bpnGothic_d$ is defined to be $\Psi_d^{-1} ( \bpn_d)$.
\end{definition}

The following result is \cite[Lemma 2.4(2)]{BlancFurter2013}.

\begin{proposition} \label{proposition: B'_d-is-locally-closed-and-control-of-its-closure}
The set $\bpnGothic_d$ is locally closed in the projective space $\rrn_d$ and thus inherits from $\rrn_d$ the structure of an algebraic variety.
\end{proposition}

The following result, which is \cite[Corollary 2.9]{BlancFurter2013}, will be crucial for us since it provides a bridge from the ``weird'' topological space $\bpn_d$ to the ``nice'' topological space $\bpnGothic_d$ which is an algebraic variety.

\begin{lemma} \label{lemma:bridge-to-an-algebraic-variety}
The map $\pi_d \colon \bpnGothic_d \to \bpn_d$ is continuous and closed. In particular, it is a quotient topological map: A subset $F \subseteq \bpn_d$ is closed if and only if its preimage $\pi_d^{-1}(F)$ is closed.
\end{lemma}

\begin{remark} \label{remark: connected components of a closed subset of Bir(Pn)d}
Let $F$ be a closed subset of $\bpn_d$ and let $\FF:= \pi^{-1}_d(F) \subseteq \bpnGothic_d$ be its preimage via $\pi_d$. By what has been said above, $\FF$ is naturally a (finite dimensional, but not necessarily irreducible) variety (being closed in the variety $\bpnGothic_d$), and the closed continuous map $\pi_d \colon \bpnGothic_d \to \bpn_d$ induces a closed continuous map $\pi_{d,F} \colon \FF \to F$ whose fibres are connected and nonempty. It follows that $\pi_{d,F}$ induces a $1-1$ correspondence between the connected components of $\FF$ and the connected components of $F$. More precisely, if $\CC$ is a connected component of $\FF$, then $\pi_{d,F} ( \CC)$ is a connected component of $F$ and conversely if $C$ is a connected component of $F$, then $(\pi_{d,F})^{-1} (C)$ is a connected component of $\FF$. In particular, $F$ admits finitely many connected components and these connected components are closed and open in $F$.
\end{remark}

Lemma~\ref{lemma:description-of-the-topology-of-Bir(Pn)-as-an-inductive-limit}, Remark~\ref{remark: strengthening of a criterion to check that a subset of Bir(Pn) is closed}, and Lemma~\ref{lemma:bridge-to-an-algebraic-variety} give the following useful characterisation  of closed subsets of $\bpn$ (this criterion is a slight generalisation of \cite[Corollary~2.7]{BlancFurter2013}).

\begin{lemma} \label{lemma:useful-characterisation-of-closed-subsets-of-Bir(Pn)}
A subset $F \subseteq \bpn$ is closed if and only if there exists a positive integer $D$ such that $\pi_d^{-1}(F \cap \bpn_d) \subseteq \bpnGothic_d$ is closed for any $d \ge D$.
\end{lemma}

Our main use of the following lemma will be the characterisation of the connectedness of a closed subset of $\bpn$ given in \eqref{any two points of F can be joined by a curve}. The intermediate characterisation \eqref{any pair of elements is contained in a connected component of the filtration} will only be used in the proof of the equivalence between \eqref{F is connected} and \eqref{any two points of F can be joined by a curve}.

\begin{lemma} \label{lemma: a useful criterion of connectedness for closed subsets of Bir(Pn)}
Let $F$ be a closed subset of $\bpn$. Then, the three following assertions are equivalent:
\begin{enumerate}
\item \label{F is connected}
$F$ is connected;
\item \label{any pair of elements is contained in a connected component of the filtration}
For each  $\varphi, \psi \in F$, there exists a positive integer $d$ such that $\varphi, \psi$ belong to the same connected component of  $F \cap \bpn_d$;
\item \label{any two points of F can be joined by a curve}
For each $\varphi, \psi \in F$, there exists a connected (not necessarily irreducible) curve $C$ and a morphism $\lambda \colon C \to \bpn$  (see Definition~\ref{definition: morphism to Bir(W)}) whose image satisfies:
\[ \varphi, \psi \in \Image (\lambda ) \subseteq F .\]
\end{enumerate}
\end{lemma}

\begin{proof}
\eqref{F is connected} $\Longrightarrow$ \eqref{any pair of elements is contained in a connected component of the filtration}. For each positive integer $d$, set $F_d:= F \cap \bpn_d$. Assume that $\varphi$ is an element of some $F_{d_0}$. For each $d \geq d_0$, the connected component of $F_d$ which contains $\varphi$ will be denoted by $F_{d,\varphi}$. We want to prove that $F_{\varphi} := \bigcup_{d \geq d_0} F_{d,\varphi}$ is equal to $F$. For this, it is sufficient to check that $F_{\varphi}$ is open and closed in $F$, i.e.\ that $F_{\varphi} \cap F_d$ is closed and open in $F_d$ for each positive integer $d$. However, since the set $F_d$ has only finitely many connected components (see Remark~\ref{remark: connected components of a closed subset of Bir(Pn)d}) and since $F_{\varphi} \cap F_d$ is a union of such connected components, it follows that $F_{\varphi} \cap F_d$ is actually closed and open in $F_d$.

\eqref{any pair of elements is contained in a connected component of the filtration} $\Longrightarrow$ \eqref{any two points of F can be joined by a curve}. Let $\varphi, \psi$ be elements of $F$ and let $d$ be a positive integer such that $\varphi, \psi$ belong to the same connected component of  $F \cap \bpn_d$. Since each connected component of $F \cap \bpn_d$ is of the form $\pi_d (A)$ for some connected variety $A$ (see Remark~\ref{remark: connected components of a closed subset of Bir(Pn)d}), we have
\[ \varphi, \psi \in \pi_d (A) \subseteq F .\]
Let $a,b \in A$ be such that $\varphi = \pi_d (a)$, $\psi = \pi_d (b)$. It's enough to show that there exists a connected curve $C$ on the variety $A$ containing $a$ and $b$. Let $A'$ be the set of points $c \in A$ for which there exists a connected curve $C \subseteq A$ containing $a$ and $c$. By \cite[Lemma on page 56]{Mumford2008}, for any irreducible variety $V$ and any points $v,w \in V$, there is an irreducible curve on $V$ containing $v$ and $w$. It follows that if $c$ belongs to $A'$, then all the irreducible components of $A$ which contain $c$ are contained in $A'$. Since $A$ is connected, this shows that $A' =A$ and this concludes the proof of the implication \eqref{any pair of elements is contained in a connected component of the filtration}~$\Longrightarrow$~\eqref{any two points of F can be joined by a curve}.

\eqref{any two points of F can be joined by a curve}  $\Longrightarrow$ \eqref{F is connected}. This is obvious.
\end{proof}

\begin{definition} \label{definition: morphism from Bir(Pn) to Bir(Pn)}
A map $\varphi \colon \bpn \to \bpn $ will be called a morphism if for each irreducible complex algebraic variety $A$ and each morphism $\rho \colon A \to \bpn$ (in the sense of Definition~\ref{definition: morphism to Bir(W)}), the composition $\varphi \circ \rho$ is also a morphism (still in the sense of Definition~\ref{definition: morphism to Bir(W)}).
\end{definition}

The following result directly follows from the Definitions~\ref{definition: Zariski topology} and \ref{definition: morphism from Bir(Pn) to Bir(Pn)}. 

\begin{lemma} \label{lemma:  four properties of morphisms from Bir(Pn) to Bir(Pn) and products of connected subsets}
The four following assertions are satisfied.
\begin{enumerate}
\item \label{The inverse map is a morphism}
The inverse map $\iota \colon \bpn \to \bpn $, $g \mapsto g^{-1}$,  is a morphism.
\item  \label{The multiplication of two morphisms is a morphism}
Let $\pi \colon \bpn \times \bpn \to \bpn$ be the map that sends $(g,g')$ onto $g \circ g'$. If $\varphi, \varphi' \colon \bpn \to \bpn $ are morphisms, then the map $\pi ( \varphi, \varphi') \colon \bpn \to \bpn $, that sends $g$ onto $\varphi (g)  \circ \varphi ' (g)$, is a morphism.
\item \label{A morphism is continuous}
Any morphism $\varphi \colon \bpn \to \bpn $ is continuous.
\item \label{The product in Bir(Pn) of two connected subsets of Bir(Pn) is connected}
If $V,W$ are two connected subsets of $\bpn$, then their product
\[ V.W = \{ v \circ w, \; v \in V, \; w \in W \} \subseteq \bpn \]
is connected.
\end{enumerate}
\end{lemma}

\begin{proof}
\eqref{The inverse map is a morphism}
Let $A$ be an algebraic variety and $\rho \colon A \to \Bir( \p^n )$ be a morphism. We want to show that $\iota \circ \rho \colon A \to \bpn$ is a morphism. Since $\rho$ is a morphism, there exists an $A$-birational map $f$ of the $A$-variety $A \times  \p^n$, inducing an isomorphism $U \to V$, where $U,V$ are open subsets of $A \times \p^n$, whose projections on $A$ are surjective, and such that $\rho$ is the family associated to $f$. This last point means that for each $\C$-point $a \in A$, the birational map $\rho (a) \colon \p^n \dasharrow \p^n$ is the transformation $ w \dasharrow p_2( f(a, w) )$. For showing that $\iota \circ \rho$ is a morphism, it's enough to note that $\iota \circ \rho$ is the family associated to the $A$-birational map $f^{-1}$.

\eqref{The multiplication of two morphisms is a morphism}
Set $\psi := \pi ( \varphi, \varphi')$. Let $A$ be an algebraic variety and $\rho \colon A \to \Bir( \p^n )$ be a morphism. We want to show that $\psi \circ \rho \colon A \to \bpn$ is a morphism. 

Since $\varphi \circ \rho$ is a morphism, there exists an $A$-birational map $f$ of the $A$-variety $A \times  \p^n$, inducing an isomorphism $ \theta \colon U \to V$, where $U,V$ are open subsets of $A \times \p^n$, whose projections on $A$ are surjective, and such that $\varphi \circ \rho$ is the family associated to $f$. 

Analogously, since $\varphi' \circ \rho$ is a morphism, there exists an $A$-birational map $f'$ of the $A$-variety $A \times  \p^n$, inducing an isomorphism $ \theta' \colon U' \to V'$, where $U',V'$ are open subsets of $A \times \p^n$, whose projections on $A$ are surjective, and such that $\varphi' \circ \rho$ is the family associated to $f'$. 

For showing that $\psi \circ \rho$ is a morphism, it's enough to note that $\psi \circ \rho$ is the family associated to the $A$-birational map $f \circ f'$. Let us just check that there exists open subsets $U'',V''$ of $A \times \p^n$, whose projections on $A$ are surjective, and such that  $f \circ f'$ induces an isomorphism $U'' \to V''$. Since $U,V'$ are open subsets of $A \times \p^n$, whose projections on $A$ are surjective, one would easily check that $U \cap V'$ is also an open subset of $A \times \p^n$, whose projection on $A$ is surjective. It's now enough to set $U'':= (\theta')^{-1} (U \cap V')$ and $V'' := \theta (U \cap V')$.

\eqref{A morphism is continuous}
Let $F$ be a closed subset of $ \bpn $. We want to show that $\varphi ^{-1} (F)$ is closed in $\bpn$, i.e. that for each algebraic variety $A$ and each morphism $\rho \colon A \to \Bir( \p^n )$, the pre\-image $\rho^{-1} ( \varphi ^{-1} (F) )$ is closed. Since $\varphi \circ \rho \colon A \to \bpn$ is a morphism  (by Definition~\ref{definition: morphism from Bir(Pn) to Bir(Pn)}), this follows from  Definition~\ref{definition: Zariski topology}.

\eqref{The product in Bir(Pn) of two connected subsets of Bir(Pn) is connected}
We may assume that both $V$ and $W$ are nonempty. Take $w_0 \in W$. Since
\[ V.W = \bigcup_{v \in V} v.W\]
where each $v.W$ is connected and intersects the fixed connected subset $V.w_0$ of $V.W$, this shows that $V.W$ is connected.
\end{proof}

If $G$ is a linear algebraic group, it is well-known that its derived group $D(G)$ is closed (see e.g.\  \cite[Proposition 17.2, page 110]{Humphreys1975}). It is not clear whether this result remains true for closed subgroups $G$ of $\bpn$, but thanks to the following definition and to the next two lemmas, this will not be a concern for us.

\begin{definition} \label{definition: DDD(G)}
Let $G$ be any subgroup of $\bpn$.
\begin{enumerate}
\item
Set $\DDD (G):=\overline{ D(G) }$ (the closure of the derived group of $G$). 
\item
We then define $\DDD^k (G)$ inductively by
\[\DDD^0 (G) :=\overline{G}, \quad \DDD^1 (G) :=\DDD (G),  \hspace{2mm} \text{and} \hspace{3mm} \DDD^k (G) := \DDD^1 ( \DDD^{k-1} (G) ) \hspace{2mm} \text{for} \hspace{2mm} k \geq 2.\] 
\end{enumerate}
\end{definition}

\begin{lemma} \label{lemma: DG and (mathcal-D)G}
Let $G$ be any subgroup of $\bpn$.
\begin{enumerate}
\item \label{the derivative of the closure is contained in the closure of the derivative}
We have $D(G) \subseteq D ( \overline{G} ) \subseteq \DDD (G)$.
\item \label{(mathcal-D)G=(mathcal-D)(closure of G)}
We have $\DDD (G) = \DDD ( \overline{G} )$.
\item \label{G-closed-implies-(mathcal-D)1G-normal}
If $G$ is closed, then  $\DDD (G) $ is normal in $G$ and the quotient $G / \DDD (G) $ is abelian.
\item \label{G-connected-implies-DG-connected}
If $G$ is connected, then $D (G)$ is also connected.

\item  \label{G-closed-connected-implies-(mathcal-D)G-connected}
If $G$ is closed and connected, then $\DDD (G)$ is also closed and connected.
\end{enumerate}
\end{lemma}

\begin{proof}
\eqref{the derivative of the closure is contained in the closure of the derivative}
The inclusion $D(G) \subseteq D ( \overline{G} ) $ being obvious, let's prove $D ( \overline{G} ) \subseteq \overline{ D(G) }$.
Fix an element $h$ of $G$. By Lemma~\ref{lemma:  four properties of morphisms from Bir(Pn) to Bir(Pn) and products of connected subsets}\eqref{The inverse map is a morphism}-\eqref{The multiplication of two morphisms is a morphism},  the map
\[ \varphi_h \colon \bpn \to \bpn, \quad g \mapsto [g,h] = ghg^{-1}h^{-1}\]
is a morphism.  By Lemma~\ref{lemma:  four properties of morphisms from Bir(Pn) to Bir(Pn) and products of connected subsets}\eqref{A morphism is continuous}, it is in particular continuous. Since $G$ is obviously contained in $(\varphi_h)^{-1} ( \overline{D(G)} )$, we get $\overline{G} \subseteq (\varphi_h)^{-1} ( \overline{D(G)} )$. Consequently,   we have proven that
\[ \forall g \in \overline{G}, \quad \forall h \in G, \quad [g,h] \in \overline{D(G)}.\]
Similarly, for each fixed element $g$ of $\overline{G}$, the map $\psi \colon \bpn \to \bpn$, $h \mapsto [g,h]$ is continuous. Since $G$ is included in  $\psi^{-1} ( \overline{D(G)} )$, we get $\overline{G} \subseteq \psi^{-1} ( \overline{D(G)} )$ and thus
\[ \forall g,h \in \overline{G},  \quad [g,h] \in \overline{D(G)}.\]
This  implies the desired inclusion.

\eqref{(mathcal-D)G=(mathcal-D)(closure of G)} By taking the closure of \eqref{the derivative of the closure is contained in the closure of the derivative}, we obtain $\DDD (G) \subseteq \DDD ( \overline{G} ) \subseteq \DDD (G)$ and the equality follows.

\eqref{G-closed-implies-(mathcal-D)1G-normal}
Any group $H$ such that $D(G) \subseteq H \subseteq G$ is normal in $G$ with $G/H$ abelian. Hence, the result is a consequence of \eqref{the derivative of the closure is contained in the closure of the derivative}.

\eqref{G-connected-implies-DG-connected}
Let's begin by checking that the set $\Com (G)$ of commutators of $G$ is connected. Let $h$ be an element of $G$. We have seen above (in the proof of \eqref{the derivative of the closure is contained in the closure of the derivative}) that the map
\[ \varphi_h \colon \bpn \to \bpn, \quad g \mapsto [g,h] = ghg^{-1}h^{-1}\]
is continuous. Hence $\varphi_h (G)$ is connected. Since all $\varphi_h (G)$, $h \in G$, are connected and contain $\id \in G$, it follows from the equality $\Com (G) = \bigcup_{h \in G} \varphi_h (G)$ that $\Com (G)$ is connected. It now follows from Lemma~\ref{lemma:  four properties of morphisms from Bir(Pn) to Bir(Pn) and products of connected subsets}\eqref{The product in Bir(Pn) of two connected subsets of Bir(Pn) is connected} that for each positive integer $j$, the set
\[ \Com^j (G):= \{ c_1 \ldots c_j, \; c_1, \ldots, c_j \in \Com (G) \} \]
is also connected. Therefore the increasing union $D^1 (G) = \bigcup_j \Com^j(G)$ is connected.

\eqref{G-closed-connected-implies-(mathcal-D)G-connected} This is a direct consequence of the previous point.
\end{proof}

An induction based on Lemma~\ref{lemma: DG and (mathcal-D)G}\eqref{(mathcal-D)G=(mathcal-D)(closure of G)} yields the following result:

\begin{lemma} \label{lemma: D^kG and (mathcal-D)^kG}
Let $G$ be a subgroup of $\bpn$. Then, for each nonnegative integer $k$, we have $\DDD^k (G) = \overline{D^k (G) }$. In particular, we have $D^k(G) = \{1 \}$ if and only if $\DDD^k (G )  =\{ 1 \}$. This means that if $G$ is solvable its derived length is also equal to the least nonnegative integer $k$ such that $\DDD^k(G) = \{ 1 \}$.
\end{lemma}

\begin{definition} \label{definition: bounded subset of Bir(Pn)}
A subset $A$ of $\bpn$ is bounded if there exists a constant $K$ such that $\deg(g)\leq K$ for all $g \in A$.
\end{definition}

\begin{remark} \label{remark: an algebraic subgroup of Bir(Pn) is a bounded closed subgroup}
An algebraic subgroup of $\bpn$  is  nothing else than a bounded closed subgroup $($see  \cite[Remark 2.20]{BlancFurter2013}$)$.
\end{remark}

\begin{lemma} \label{lemma: Jonq is closed and pr_2 is continuous}
\begin{enumerate}
\item \label{Jonq is closed in Bir(P2)}
The group $\Jonq$ is closed in $\bp$.
\item \label{pr_2 is continuous}
The projection $\pr_2 \colon \Jonq \to \PGL_2$, $\left( \frac{\alpha (y) x + \beta (y)}{\gamma (y) x + \delta (y)}, \frac{ay+b}{cy+d} \right) \mapsto \smat  {a} {b} {\rule{0mm}{3mm} c} {d} $, is continuous.
\item \label{The image of a bounded closed subset of Jonq by pr_2 is a constructible subset of PGL_2}
If $A$ is a bounded closed subset of $\Jonq$, then $\pr_2 (A)$ is a constructible subset of~$\PGL_2$.
\item \label{PGL_2 ( C (y) ) rtimes Aff_1 and PGL_2 ( C (y) ) are closed in Bir(P2)}
If $H$ is a closed subgroup of $\PGL_2$, then the group $\PGL_2 ( \C (y) ) \rtimes H$ is a closed subgroup of $\Jonq$, and hence of $\bp$. In particular, the groups $\PGL_2 ( \C (y) ) \rtimes \Aff_1$ and $\PGL_2 ( \C (y) ) \subseteq \Jonq$ are closed in $\Jonq$.
\end{enumerate}
\end{lemma}

\begin{proof}
\eqref{Jonq is closed in Bir(P2)} Even if the proof is already given in \cite[Remark 5.22]{BlancFurter2018}, we recall it here in preparation for the proof of \eqref{pr_2 is continuous}. By Lemma~\ref{lemma:useful-characterisation-of-closed-subsets-of-Bir(Pn)}, it is enough to prove that $\JJonq_{d} =$ $\pi_d^{-1}(\Jonq \cap \bp_d )$ is closed in $\bpGothic_d$ for each $d$. Denote by  $\LL$ the projective space (of dimension 3) associated with the complex vector space of pairs $(g_1,g_2)$ where $g_1,g_2 \in \C [y,z]$ are homogeneous polynomials of degree $1$. The equivalence class of $(g_1,g_2)$ will be denoted by $[g_1:g_2]$. Denote by $Y\subseteq \bpGothic_d \times \LL$ the closed subvariety given by elements $([f_1:f_2:f_3],[g_1:g_2])$ satisfying $f_2g_2 =f_3 g_1$. Since $\LL$ is a complete variety, the first projection $p_1 \colon \bpGothic_d \times \LL \to \bpGothic_d$ is a closed morphism. Hence, assertion \eqref{Jonq is closed in Bir(P2)} follows from the equality $\JJonq_d = p_1 (Y)$.\\

\noindent \eqref{pr_2 is continuous} It is enough to prove that the restriction of $\pr_2$ to the set $\Jonq_d := \Jonq \cap \bp_d$ is continuous. This can be seen using the following commutative diagram:
\[\xymatrix{
 &  Y \ar[ld]_{p_1}   \ar[rdd]^{p_2} \\
\JJonq_d \ar[d]_{\pi_d} & \\
\Jonq_d \ar[rr]^{\pr_2}   & & \LL  }\]

Note that $\PGL_2 $ is naturally identified with the open subset of $\LL$ whose elements \linebreak  $[ay+bz : cy+dz]$ satisfy $\det \smat{a}{b}{c}{d} \neq 0$. Hence, it is enough to note that the horizontal map $\pr_2 \colon \Jonq_d \to \LL$ is continuous. We will use the fact that $p_2$ is continuous (being a morphism of algebraic varieties) and that $p_1 \colon Y \to \JJonq_d$ and $\pi_d \colon \JJonq_d \to \Jonq_d$ are surjective and closed (the surjectivity is obvious, the closedness comes from the fact that these two maps are restrictions to closed subsets of the closed maps $p_1 \colon \bpGothic_d \times \LL \to \bpGothic_d$ and $\pi_d \colon \bpGothic_d \to \bp_d$, see Lemma~\ref{lemma:bridge-to-an-algebraic-variety}).

Take any closed subset $F$ of $\LL$. We want to prove that $(\pr_2)^{-1} (F)$ is closed in $\Jonq_d$. This comes from the previous remarks and the equality $(\pr_2)^{-1} (F) = (\pi_d \circ p_1) ( (p_2)^{-1} (F) )$.\\

\noindent \eqref{The image of a bounded closed subset of Jonq by pr_2 is a constructible subset of PGL_2} Choose $d$ so that we have $A \subseteq \Jonq_d$. Since $p_1 \colon Y \to \JJonq_d$ and $\pi_d \colon \JJonq_d \to \Jonq_d$ are surjective, we have $\pr_2 (A) = p_2 ( ( \pi_d \circ p_1)^{-1} (A) )$, and since $p_2 \colon Y \to \LL$ is a morphism of algebraic varieties, the closed subset $( \pi_d \circ p_1)^{-1} (A)$ of $Y$ is sent by $p_2$ onto the constructible subset $\pr_2 (A)$ of $\LL$. This also shows that  $\pr_2 (A)$ is a constructible subset of $\PGL_2$.\\

\noindent \eqref{PGL_2 ( C (y) ) rtimes Aff_1 and PGL_2 ( C (y) ) are closed in Bir(P2)}
The group $\PGL_2 (\C (y) ) \rtimes H$ is the preimage of $H$ by $\pr_2 \colon \Jonq \to \PGL_2 $. Therefore, the result follows from \eqref{pr_2 is continuous}.
\end{proof}

\section{Any Borel subgroup of $\bp$ is conjugate to a subgroup of $\Jonq$} \label{section: Any Borel subgroup of Bir(P2) is conjugate to a Borel subgroup of Jonq}
The aim of this section is to prove the following result:

\begin{theorem} \label{theorem: a closed connected solvable subgroup of Bir(P2) is conjugate to a subgroup of Jonq}
Any closed connected solvable subgroup of $\bp$ is conjugate to a subgroup of $\Jonq$. In particular, any Borel subgroup of $\bp$ is conjugate to a Borel subgroup of $\Jonq$.
\end{theorem}

Recall that by definition $\Jonq$ is the group of birational transformations preserving the pencil of lines through $ [1:0:0] \in \p^2$. An element, resp.\ a subgroup, of $\bp$, is conjugate to an element, resp.\ a subgroup, of $\Jonq$ if and only if it preserves a {\it rational fibration}. In our text a rational fibration denotes what is often called  a rational fibration with rational fibres. For the sake of clarity we include the following complete definition.

\begin{definition} \label{definition: rational fibration}
\begin{enumerate}
\item
A rational fibration of $\p^2$ is a rational map $\pi \colon \p^2 \dasharrow  \p^1$ whose generic fibre is birational to $\p^1$. By Tsen's lemma, this is equivalent to saying that the element $\pi \in \C (x,y)$ is the coordinate of a Cremona transformation, i.e.\ there exists a rational map $\pi' \colon \p^2 \dasharrow  \p^1$ such that the rational map $(\pi,\pi') \colon \p^2 \dasharrow \p^1 \times \p^1$ is birational.
\item \label{preserved rational fibration}
The rational fibration $\pi \colon \p^2 \dasharrow  \p^1$ is preserved by the Cremona transformation $\alpha \in \bp$ if there exists an automorphism $\beta \in \Aut (\p^1)$ such that the following diagram is commutative:
\[\xymatrix{
\p^2 \ar@{-->}[d]_{\pi} \ar@{-->}[r]^{\alpha} & \p^2  \ar@{-->}[d]^{\pi} \\
\p^1  \ar[r]^{\beta} &  \p^1}\]
Equivalently, there exists a Cremona transformation $\varphi =  [ \varphi_1 : \varphi_2 : \varphi_3 ] \in \bp $ such that:
\begin{enumerate}
\item $\pi =  [ \varphi_2 : \varphi_3 ] $;
\item $\varphi  \alpha \varphi^{-1} \in \Jonq$.
\end{enumerate}
\item
Two rational fibrations $\pi, \pi'  \colon \p^2 \dasharrow  \p^1$ are called equivalent if there exists an automorphism $\beta \in \Aut (\p^1)$ such that $\pi' = \beta \pi$.
\item
We say that the rational fibration  $\pi \colon \p^2 \dasharrow  \p^1$ is the only rational fibration preserved by $\alpha \in \bp$ if it is preserved by $\alpha$ and if all rational fibrations preserved are equivalent to $\pi$.
\end{enumerate}
\end{definition}

The following lemma should not come as a surprise:

\begin{lemma}
Any countable closed subset of $\bp$ is discrete.
\end{lemma}

\begin{proof}
Let $C$ be such a  countable closed subset. We want to prove that any subset  $C' \subseteq C$ is closed in $C$. By Lemma~\ref{lemma:description-of-the-topology-of-Bir(Pn)-as-an-inductive-limit}, this is equivalent to proving that $C' \cap \bp_d$ is closed in $ \bp_d$ for any $d \ge 1$. Therefore, it is sufficient to prove that $C_d:= C \cap \bp_d$ is finite. Writing $C_d = \bigcup_{n \ge 1} F_n$ as an increasing union of finite subsets, we get $(\pi_d)^{-1}(C_d)=$ $\bigcup_{n \ge 1} (\pi_d)^{-1}(F_n)$. Since $(\pi_d)^{-1}(C_d)$ and $(\pi_d)^{-1}(F_n)$, $n \ge 1$, are closed subvarieties of $\bpGothic_d$ and since the ground field $\C$ is uncountable, this proves that the increasing union is stationary, i.e.\ $(\pi_d)^{-1}(C_d) =   (\pi_d)^{-1}(F_n)$ for some $n$, proving that $C_d = F_n$ is finite.
\end{proof}

\begin{corollary} \label{corollary:closed-connected-countable}
Any closed connected and countable subgroup of $\bp$ is trivial.
\end{corollary}

The following result is for example proven in \cite[Proposition 9.3.1]{FurterKraft2018}:

\begin{lemma} \label{lemma:torsion-countable}
Any torsion subgroup of a linear algebraic group is countable.
\end{lemma}

Our proof of the next result relies on Urech's paper \cite{Urech2018}:

\begin{lemma}  \label{lemma:jonq-or-countable-or-bounded}
Each solvable subgroup $G$ of $\bp$ satisfies one of the following assertions:
\begin{enumerate}
\item $G$ is conjugate to a subgroup of $\Jonq$;
\item $G$ is countable;
\item $G$ is bounded.
\end{enumerate}
\end{lemma}

\begin{proof}
By \cite[Theorem 8.1, page 25]{Urech2018}, $G$ satisfies one of the following assertions:
\begin{enumerate}[$($a$)$]
\item \label{G is conjugate to a subgroup of Jonq}
$G$ is conjugate to a subgroup of $\Jonq$;
\item \label{G is countable}
$G$ is countable;
\item \label{G is conjugate to a subgroup of the automorphism group of a Halphen surface}
$G$ is conjugate to a subgroup of the automorphism group of a Halphen surface;
\item \label{G is a subgroup of elliptic elements}
$G$ is a subgroup of elliptic elements.
\end{enumerate}
In case \eqref{G is conjugate to a subgroup of the automorphism group of a Halphen surface}, it is well known that $G$ is countable (see e.g.\ \cite[Theorem 2.4, page 9]{Urech2018}). In case \eqref{G is a subgroup of elliptic elements}, it follows from \cite[Theorem 1.3, page 3]{Urech2018} that $G$ satisfies one of the following assertions (here we use that G preserves a rational fibration if and only if it is conjugate to a subgroup of $\Jonq$):
\begin{enumerate}[$($i$)$]
\item \label{G is conjugate to a subgroup of Jonq -- 2}
$G$ is conjugate to a subgroup of $\Jonq$;
\item \label{G is a bounded subgroup}
$G$ is a bounded subgroup;
\item \label{G is a subgroup of torsion elements}
$G$ is a subgroup of torsion elements.
\end{enumerate}
In case \eqref{G is a subgroup of torsion elements}, it follows from \cite[Theorem 1.5, page 3]{Urech2018} that $G$ is isomorphic to a subgroup of $\GL_{48}(\C)$. Then, $G$ is countable by Lemma~\ref{lemma:torsion-countable}.
\end{proof}

We are now able to prove Theorem~\ref{theorem: a closed connected solvable subgroup of Bir(P2) is conjugate to a subgroup of Jonq}:

\begin{proof}[Proof of Theorem~\ref{theorem: a closed connected solvable subgroup of Bir(P2) is conjugate to a subgroup of Jonq}]
Let $G$ be a closed connected solvable subgroup of $\bp$. By Lemma~\ref{lemma:jonq-or-countable-or-bounded} and Corollary~\ref{corollary:closed-connected-countable}, we may assume that $G$ is a bounded subgroup. Therefore, $G$ is an algebraic group (see Remark~\ref{remark: an algebraic subgroup of Bir(Pn) is a bounded closed subgroup}). It follows from Enriques theorem (see \cite[Theorem (2.25), page 238]{Umemura1982b} and also \cite[Proposition (2.18), page 233]{Umemura1982b}) that any connected algebraic subgroup of $\bp$ is contained in a maximal connected algebraic subgroup, and that, up to conjugation, any maximal connected algebraic subgroup of $\bp$ is one of the following subgroups:
\begin{enumerate}
\item \label{The group Aut (p^2)}
The group $\Aut (\p^2) \simeq \PGL_3$;
\item \label{The group Aut^{0}(F_n)}
The connected component $\Aut^{\circ} (\F_n)$ of the automorphism group of the $n$-th Hirzebruch surface $\F_n$ where $n$ is a nonnegative integer different from $1$ (because $\F_n$ needs to be a minimal surface).
\end{enumerate}
In case \eqref{The group Aut (p^2)}, recall that any  closed connected  solvable subgroup of $\PGL_3$ is contained in a Borel subgroup of $\PGL_3$, that such a Borel subgroup is conjugate to the subgroup of upper triangular matrices, and that this latter group is contained in $\Jonq$. In case~\eqref{The group Aut^{0}(F_n)}, it is enough to note that all groups $\Aut^{\circ} (\F_n)$ are already conjugate to subgroups of $\Jonq$.
\end{proof}

\section{Any Borel subgroup of $\Jonq$ is conjugate to a subgroup of $\PGL_2 ( \C (y) )  \rtimes \Aff_1 $} \label{section: Any Borel subgroup of Jonq is conjugate to a Borel subgroup of PGL_2(C(y)) rtimes Aff_1}

The aim of this short section is to prove the following easy result:

\begin{theorem} \label{theorem: any closed connected solvable of Jonq is conjugate to a subgroup of PGL_2(C(y)) rtimes Aff_1}
Any closed connected solvable subgroup of $\Jonq$ is conjugate to a subgroup of $\PGL_2 ( \C(y) ) \rtimes \Aff_1$ by an element of $ \{ 1 \} \rtimes \PGL_2  \subseteq \Jonq$. In particular, any Borel subgroup of $\Jonq$ is conjugate to a Borel subgroup of $\PGL_2 ( \C(y) ) \rtimes \Aff_1$ by an element of $ \{1\}\rtimes\PGL_2 \subseteq \Jonq$.
\end{theorem}

\begin{proof}
Let $G$ be a closed connected solvable subgroup of $\Jonq$. Let 
\[ \pr_2 \colon \Jonq = \PGL_2 ( \C (y) ) \rtimes \PGL_2   \to \PGL_2\]
be the second projection. Since $\pr_2$ is a morphism of groups, the image $\pr_2(G)$ is a solvable subgroup of $\PGL_2 $ and since $\pr_2$ is continuous (Lemma~\ref{lemma: Jonq is closed and pr_2 is continuous}\eqref{pr_2 is continuous}) $\pr_2(G)$ is moreover connected. It follows that $\overline{ \pr_2(G) }$ is a closed connected solvable subgroup of $ \PGL_2$. Up to conjugation we may assume that it is contained in the  subgroup of upper triangular matrices of $\PGL_2$.
\end{proof}

Theorems~\ref{theorem: a closed connected solvable subgroup of Bir(P2) is conjugate to a subgroup of Jonq} and \ref{theorem: any closed connected solvable of Jonq is conjugate to a subgroup of PGL_2(C(y)) rtimes Aff_1} directly give the following result:

\begin{theorem} \label{theorem: any closed connected solvable subgroup of Bir(P2) is conjugate to a subgroup of PGL_2(C(y)) rtimes Aff_1}
Any closed connected solvable subgroup of $\bp$ is conjugate to a subgroup of  $\PGL_2 (\C(y) ) \rtimes \Aff_1$. In particular, any Borel subgroup of $\bp$ is conjugate to a Borel subgroup of $\PGL_2 (\C(y) ) \rtimes \Aff_1$.
\end{theorem}

\section{$\K$-Borel subgroups of $\PGL_2 (\K)$} \label{section: K-Borel subgroups of PGL(2,K)}

In this section, $\K$ denotes a field of characteristic zero. For the sake of clarity, we begin with a classical definition and a well-known lemma.

\begin{definition} \label{definition: the K-topology}
Let $V$ be an affine variety defined over $\K$. The Zariski $\K$-topology on $V ( \K)$ (or for short, the $\K$-topology) is the topology for which a subset is closed if it is the zero set of some collection of elements of the affine algebra $\K [ V ]$.
\end{definition}

\begin{lemma} \label{lemma: field restriction on varieties}
For any field extension $\K \subseteq \K'$ and any affine variety $V$ defined over $\K$, the $\K$-topology on $V (\K)$ coincides with the topology induced by the $\K'$-topology on $V ( \K ')$.
\end{lemma}

\begin{proof}
It is clear that any $\K$-closed subset of $V (\K)$ is the trace of a $\K'$-closed subset of $V ( \K' )$. Conversely, choose any basis $(e_i)_{i \in I}$ of $\K'$ over $\K$. We have $\K' = \bigoplus_{i \in I} \K e_i$ and $\K' [V ] =  \bigoplus_{i \in I} \K [V ] e_i$. Hence, if $Z \subseteq V (\K)$ is the zero set of some collection $(f_j)_{j \in J}$ of elements $f_j \in \K' [V]$, then it is also the zero set of the collection $( f_{ij} ) _{(i,j) \in I \times J}$ where each $f_j$ is decomposed as $f_j = \sum_{i \in I} f_{ij} e_i$, $f_{ij} \in \K [V]$.
\end{proof}

\begin{definition} \label{definition: K-Borel subgroup of G(K)}
Let $G$ be a linear algebraic group defined over $\K$. We will say that a subgroup $B$ of  $G (\K)$ is a $\K$-Borel subgroup if it is a maximal closed connected solvable subgroup of $G (\K)$ for the $\K$-topology.
\end{definition}

This notion is not to be confused with the classical notion of a Borel subgroup of $G$ defined over $\K$  (or equivalently an algebraic $\K$-Borel subgroup of $G$) which we now recall:  This is  a $\K$-closed subgroup $B$ of $G$ such that $B (\overline{\K})$ is a maximal $\overline{\K}$-closed connected solvable subgroup of $G(\overline{\K})$. We will not use this classical notion at all in our text. One reason for  studying $\K$-Borel subgroups rather than algebraic $\K$-Borel subgroups will become apparent in Theorem~\ref{theorem : the Borel subgroups of PGL(2,C(y))}.

We will prove in Theorem~\ref{theorem: K-Borel subgroups of PGL(2,K)} that if $f$ is a nonsquare element of $\K$, then the group 
\[ \T_f :=  \left\{ \smat{a }{bf}{b}{a}, \; a,b \in \K, \;  (a,b) \neq (0,0) \right\} \]
is a $\K$-Borel subgroup of $\PGL_2 ( \K )$.

\begin{proposition} \label{proposition: T_f is isomorphic to K [ sqrt{f} ]^* / K^*}
If $f$ is a nonsquare element of $\K$, then the group 
\[ \T_f =  \left\{ \smat{a }{bf}{b}{a}, \; a,b \in \K, \;  (a,b) \neq (0,0) \right\} \]
is abstractly isomorphic to the group $\K [ \sqrt{f} ]^* / \K^*$.
\end{proposition}

\begin{proof}
Let $\K [ C ]$ be the $\K$-subalgebra of $ \Mat_2 (\K )$ spanned by $C$. The minimal polynomial of the matrix $C:= \smat{0 }{f}{1}{0} \in \Mat_2 (\K )$ being equal to $\mu_C (T) = T^2-f$, we have $\K [ C ] = \{ a I + b C, \; a,b \in \K \}$ and the map
\[ \K [C] \to \K [ \sqrt{f} ], \quad a I + b C \mapsto a + b \sqrt{f}\]
is a $\K$-isomorphism of fields. This map induces the isomorphism of groups
\[  \T_f  \stackrel{\sim}{\longrightarrow} \K [ \sqrt{f} ]^* / \K^* . \qedhere \]
\end{proof}

In the sequel we will often use the distinguished element $\iota_f$ of $\T_f$ that we now define.

\begin{definition} \label{definition: iota_f}
For a nonsquare element $f$ of $\K$, we set $\iota_f := \smat{0 }{f}{1}{0} \in \T_f$.
\end{definition}

A straightforward computation establishes the following result.

\begin{lemma} \label{lemma: iota_f is the unique involution of T_f}
Let $f$ be a nonsquare element of $\K$. Then $\iota_f$ is the unique involution of~$\T_f$.
\end{lemma}

\begin{lemma} \label{lemma: each nontrivial element of T_f is nontriangularisable}
Let $f$ be a nonsquare element of $\K$. Then each nontrivial element of $\T_f$ is nontriangularisable in  $\PGL_2 (\K)$.
\end{lemma}

\begin{proof}
Set $A:= \smat{a}{bf}{b}{a} \in \GL_2 (\K)$ where $a,b \in \K$ and $b \neq 0$. If the class $\overline{A}$ of $A$ in $\PGL_2 (\K)$ was triangularisable, then $A$ should be triangularisable in $\GL_2 (\K)$. However, its characteristic polynomial is $\chi_A (T)  = T^2 -2a T + (a^2 -b^2f) \in \K [T]$ whose discriminant $\Delta = 4 b^2 f$ is not a square. Hence $\chi_A$ does not split over $\K$. A contradiction.
\end{proof}

The following result is the key lemma of this section.

\begin{lemma} \label{lemma: closed connected subgroups of PGL_2(K)}
Let $H$ be a closed connected subgroup of $\PGL_2 (\K)$. Then, up to conjugation, one of the following cases occurs:
\begin{enumerate}
\item \label{H={id}}
$H = \{ \id \}$;
\item \label{H isomorphic to K^*}
$H = \{ \smat{a }{0}{0}{1}, \; a \in \K^*  \} \simeq (\K^*, \times)$;
\item \label{H isomorphic to K}
$H = \{ \smat{1}{a}{0}{1}, \; a \in \K  \} \simeq (\K, +)$;
\item \label{H=T_f}
$H = \T_f$ for some nonsquare element $f \in \K$;
\item \label{H=Aff_1}
$H = \Aff_1 (\K) =  \{ \smat{a}{b}{0}{1}, \; a \in \K^*, \: b \in \K  \}$;
\item \label{H=PGL_2(K)}
$H= \PGL_2 (\K)$.
\end{enumerate}
\end{lemma}

\begin{proof}
Let $p \colon \GL_2 (\K) \to \PGL_2 (\K)$ be the natural surjection. The map $H \mapsto p^{-1} (H)$ induces a bijection between the closed connected subgroups of $\PGL_2 (\K)$ and the closed connected subgroups of $\GL_2 (\K)$ containing the group $\K^* \,  \id$. If $G$ is a linear algebraic group defined over $\K$ whose Lie algebra is denoted ${\mathfrak g}$, recall that each closed connected subgroup $H$ of $G$ is uniquely determined by its Lie algebra ${\mathfrak h} \subseteq {\mathfrak g}$ (but that not every Lie subalgebra of ${\mathfrak g}$ corresponds to an algebraic subgroup of $G$; the Lie algebras corresponding to algebraic subgroups are called algebraic)(see e.g.\ \cite[\S 7, page 105]{Borel1991}). In our case, we want to describe, up to conjugation in $\GL_2 (\K)$, all algebraic Lie subalgebras of $\mathfrak{gl}_2$ containing $\K \,  \id$. Actually, we will show that each Lie subalgebra of $\mathfrak{gl}_2$ containing $\K \,  \id$ is  algebraic, i.e.\ is the Lie algebra of some closed connected subgroup of $\GL_2 (\K)$ containing $\K^* \, \id$. Since $\mathfrak{gl}_2 = \mathfrak{sl}_2 \oplus \K \, \id$ as Lie algebras, our problem amounts to describing up to conjugation all Lie subalgebras ${\mathfrak h}$ of $\mathfrak{sl}_2$.

If $\dim {\mathfrak h} =2$, let's prove that ${\mathfrak h}$ is conjugate by an element of $\GL_2 ( \K)$ to the Lie algebra ${\mathfrak u}$ of upper triangular matrices. Write $\mathfrak{sl}_2 = \Vect (E,F,H)$ where
\[E=\smat{0}{1}{0}{0}, \quad   F=\smat{0}{0}{1}{0}, \quad H= \smat{1}{  \hspace{2mm} 0}{0}{-1},\]
and recall that $[H,E] = 2 E$, $[H,F] = -2 F$, and $[E,F] = H$. It is clear that ${\mathfrak h}$ admits a basis $A,B$ where $A$ is upper triangular, i.e.\ of the form $A= \smat{a}{  \hspace{2mm}b}{0}{-a}$. Up to conjugation and multiplication by an element of $\K^*$, we may even assume that $A$ is either $E$ (if $a=0$) or $H$ (if $a \ne 0$).

If $A= E$, then we may assume that $B =\alpha F + \beta H$ where $\alpha, \beta  \in \K$. We have $[A,B] = \alpha H - 2 \beta E$. This yields $\alpha = 0$ (because otherwise ${\mathfrak h}$ would contain $E,H, F$ and would not be $2$-dimensional) and hence ${\mathfrak h} = \Vect (E,H) = {\mathfrak u}$.

If $A= H$, then we may assume that $B =\alpha E + \beta F$ where $\alpha, \beta  \in \K$. We have $[A,B] = 2 \alpha E - 2 \beta F $. Hence, the matrices $B$ and $[A,B]$ are linearly dependent, i.e.\ $\det \smat{\alpha}{ \hspace{2mm} \beta}{2 \alpha}{- 2 \beta}$  $= 0$, i.e.\ $4 \, \alpha \beta =0$. If $\alpha =0$, we get  ${\mathfrak h} = \Vect (F,H)$ and finally $\smat{0}{1}{1}{0} {\mathfrak h} \smat{0}{1}{1}{0}^{-1} = {\mathfrak u}$. If  $\beta =0$, we get  ${\mathfrak h} = \Vect (E,H) = {\mathfrak u}$. We have actually proven that up to conjugation ${\mathfrak h}$ is always equal to the Lie subalgebra of upper triangular matrices. We are in case \eqref{H=Aff_1}.

If $\dim {\mathfrak h} =1$, then ${\mathfrak h} = \Vect (A)$ for some nonzero matrix $A$ of $\mathfrak{sl}_2$. Setting  ${f:= - \det (A)}$, the characteristic polynomial of $A$ is equal to $\chi_A = T^2 - f$. If $f$ is a nonsquare element, then $\chi_A$ is irreducible, and up to conjugation by an element of $\GL_2 (\K)$, we may assume that $A$ is the companion matrix $\smat{0}{f}{1}{0}$. We are in case~\eqref{H=T_f}. Assume now that $f$ is a square. If $f=0$, then, up to conjugation, we may assume that $A = \smat{0}{1}{0}{0}$. We are in case \eqref{H isomorphic to K}. If $f \neq 0$, then, up to conjugation and multiplication by an element of $\K^*$, we may assume that $A = \smat{1}{ \hspace{2mm} 0}{ 0}{-1}$. We are in case \eqref{H isomorphic to K^*}.
\end{proof}

Our main result is a direct consequence of 
Lemma~\ref{lemma: closed connected subgroups of PGL_2(K)}.

\begin{theorem} \label{theorem: K-Borel subgroups of PGL(2,K)}
Up to conjugation, any $\K$-Borel subgroup of $\PGL_2( \K )$ is equal to either $\Aff_1 (\K  )$ or some $\T_f$, where $f$ is a nonsquare element of $\K$. Moreover, any closed connected solvable subgroup of $\PGL_2 (\K)$ for the $\K$-topology is contained in a $\K$-Borel subgroup.
\end{theorem}

\begin{proposition} \label{proposition: when are Tf and Tg conjugate in PGL(2,K)}
If $f,g \in \K$ are nonsquares, then the following assertions are equivalent:
\begin{enumerate}
\item \label{Tf and Tg are conjugate}
The groups $\T_f$ and $\T_g$ are conjugate;
\item \label{f/g is a square}
The ratio $f / g$ is a perfect square.
\end{enumerate}
\end{proposition}

\begin{proof}
\eqref{Tf and Tg are conjugate} $\Longrightarrow$ \eqref{f/g is a square}. Assume that $\T_f$ and $\T_g$ are conjugate. Then Lemma~\ref{lemma: iota_f is the unique involution of T_f} shows that $\iota_f = \smat{0 }{f}{1}{0}$ and $\iota_g = \smat{0 }{g}{1}{0}$ are conjugate. Denote by $\overline{\det} \colon \PGL_2 (\K) \to \K^* / ( \K^*)^2$ the morphism of groups induced by the determinant morphism $\det \colon \GL_2 (\K) \to \K^*$. The equality  $\overline{\det} ( \iota_f) = \overline{\det} ( \iota_g) $ exactly means that $f/g$ is a square.

\eqref{f/g is a square}  $\Longrightarrow$ \eqref{Tf and Tg are conjugate}.
Assume that $g = \lambda^2 f$ for some nonzero element $\lambda \in \C (y)$. If $a,b \in \C (y) ^*$ the equality $ \smat{ \lambda }{0}{0}{1} \smat{a }{bf}{b}{a} \smat{ \lambda }{0}{0}{1}^{-1} = \smat{ \lambda a }{ bg}{b}{\lambda a}$ shows that $ \smat{ \lambda }{0}{0}{1}  \T_f  \smat{ \lambda }{0}{0}{1}^{-1} = \T_g$.
\end{proof}

The following result shows that a non-triangula\-risable $\K$-Borel subgroup of $\PGL_2 (\K)$ is uniquely determined by any of its nontrivial elements.

\begin{lemma} \label{lemma: a nontrivial Borel subgroup of PGL(2,K) is uniquely determined by any of its nontrivial elements}
If $A$ is a nontrivial element of a non-triangularisable $\K$-Borel subgroup $B$ of $\PGL_2 (\K)$, then $B$ is the unique $\K$-Borel subgroup of $\PGL_2 (\K)$ containing $A$.
\end{lemma}

\begin{proof}
Let $p \colon \GL_2 (\K) \to \PGL_2 (\K)$ be the canonical surjection and let $\tilde A$ be an element of $\GL_2 (\K)$ satisfying $p( \tilde A) = A$. Set also $I:= \smat{1 }{0}{0}{1} \in \Mat_2 (\K)$. Let's begin by checking that the following equality holds 
\begin{equation}
B = \{ p ( \alpha I + \beta \tilde A),  \; \alpha, \beta \in \K, \;  ( \alpha, \beta ) \neq (0,0)  \,  \}. \label{equation: B in terms of A tilde} \end{equation}
Up to conjugation, we may assume that $B =\T_f$ for some nonsquare element $f$ of $\K$. Setting $J_f :=   \smat{0 }{f}{1}{0} \in \Mat_2 (\K)$, $\tilde A$ is necessarily of the form  $\tilde A = \smat{a}{bf}{b}{a}$, i.e. $\tilde A = a I + b J_f $,  for some $a,b \in \K$ with $b \neq 0$. This shows that in the $\K$-vector space $\Mat_2 (\K)$ we have
\[ \Vect ( I, J_f) =  \Vect ( I, \tilde A ) . \]
It follows that 

$\begin{array}{ll}
\T_f  & = \{ p ( \alpha I + \beta J_f),  \; \alpha, \beta \in \K, \;  ( \alpha, \beta ) \neq (0,0)  \,  \} \\ &  = \{ p ( \alpha I + \beta \;  \tilde A ),  \; \alpha, \beta \in \K, \;  ( \alpha, \beta ) \neq (0,0)  \,  \}, \end{array} $

\noindent and this proves \eqref{equation: B in terms of A tilde}.

Assume now that $B'$ is a $\K$-Borel subgroup of $\PGL_2 (\K)$ containing $A = \smat{a}{bf}{b}{a}$. By Lemma~\ref{lemma: each nontrivial element of T_f is nontriangularisable}, $B'$ is non-triangularisable. Hence the equality \eqref{equation: B in terms of A tilde} also applies to $B'$ and this shows that $B=B'$.
\end{proof}

Lemma~\ref{lemma: a nontrivial Borel subgroup of PGL(2,K) is uniquely determined by any of its nontrivial elements} provides the following useful result.

\begin{proposition} \label{proposition: a remarkable characterisation of equality or conjugation}
Let $B,B'$ be two $\K$-Borel subgroups of $\PGL_2 (\K)$. Assume moreover that $B$ or $B'$ is not triangularisable. Then, the following assertions hold.
\begin{enumerate}
\item
We have $B= B'$ if and only if $B \cap B' \ne \{ 1 \}$.
\item
If $\varphi$ is an element of $\PGL_2 (\K)$ and $A$ a nontrivial element of $B$, we have $ \varphi B \varphi^{-1} = B'$ if and only if $\varphi A \varphi^{-1} \in B'$.
\end{enumerate}
\end{proposition}

For later use we now compute the normalisers of $\T_f$ and $\Aff_1 (\K)$.

\begin{lemma} \label{lemma: the normaliser of T_f in PGL_2(K)}
Let $f$ be a nonsquare element of $\K$. Recall that we have set $\iota_f = \smat{0 }{f}{1}{0}$. Then we have
\[ \Nor_{\PGL_2(\K)}(\T_f) = \Cent_{\PGL_2(\K)}(\iota_f) = \T_f \rtimes \langle \smat{1 }{  \hspace{2mm} 0}{0}{-1} \rangle . \]
\end{lemma}

\begin{proof}
We will prove that
\[ \Nor_{\PGL_2(\K)}(\T_f)  \subseteq \Cent_{\PGL_2(\K)}(\iota_f)  \subseteq  \T_f \rtimes \langle \smat{1 }{  \hspace{2mm} 0}{0}{-1} \rangle \subseteq \Nor_{\PGL_2(\K)}(\T_f).\]
The inclusion $ \Nor_{\PGL_2(\K)}(\T_f)  \subseteq \Cent_{\PGL_2(\K)}(\iota_f)$ directly follows from Lemma~\ref{lemma: iota_f is the unique involution of T_f}. Assume now that the element $M= \smat{a }{b}{c}{d}$ of $\PGL_2 (\K)$ centralises $\iota_f$.
We have
\[ \smat{a }{b}{c}{d} \smat{0 }{f}{1}{0} = \smat{0 }{ f}{1}{0} \smat{a }{b}{c}{d}, \quad \text{i.e.} \quad \smat{b }{af}{d}{cf} = \smat{c f }{d f}{a}{b}.\]
This is equivalent to the existence of $\varepsilon \in \K^*$ such that
\[ cf = \varepsilon b, \quad d = \varepsilon a, \quad a = \varepsilon d, \quad b= \varepsilon cf.\]
Since $a= \varepsilon^2 a$ and $b= \varepsilon^2 b$ where $(a,b) \neq (0,0)$, we have $\varepsilon^2 = 1$, i.e.\ $\varepsilon = \pm 1$, and these equations are equivalent to
\[ d = \varepsilon a, \quad b= \varepsilon cf.\]
This means that we have $M=  \smat{a}{\varepsilon cf}{c}{\varepsilon a}$, i.e.\ $M=  \smat{a}{cf}{c}{a} \smat{\varepsilon}{0}{0}{1} \in \T_f \rtimes \langle \smat{1 }{0}{0}{-1} \rangle$. We have proven $ \Cent_{\PGL_2(\K)} \subseteq  \T_f \rtimes \langle \smat{1 }{  \hspace{2mm} 0}{0}{-1} \rangle$.

Finally, the equality $ \smat{1 }{ \hspace{2mm} 0 }{0}{ -1 } \smat{a }{bf }{b}{ a } \smat{1 }{  \hspace{2mm}0 }{0}{ -1 }= \smat{a }{-bf }{-b}{ a }$ shows that $ \smat{1 }{ \hspace{2mm} 0 }{0}{ -1 }$ normalises $\T_f$. This proves that $\T_f \rtimes \langle \smat{1 }{ \hspace{2mm} 0}{0}{-1} \rangle \subseteq \Nor_{\PGL_2(\K)}(\T_f)$.
\end{proof}

\begin{remark}
Recall that a Borel subgroup $B$ of a linear algebraic group $G$ defined over an algebraically closed field is always maximal among the solvable subgroups of $G$ by \cite[Corollary 23.1A, page 143]{Humphreys1975}. In contrast,  the $\K$-Borel subgroup $\T_f$ is contained in the larger solvable subgroup $\T_f  \rtimes \langle  \smat{1}{\hspace{2mm} 0}{0}{-1} \rangle $. 
\end{remark}

\begin{lemma} \label{lemma: the normaliser of Aff_1(K) in PGL_2(K)}
We have
\[ \Nor_{\PGL_2(\K)}(\Aff_1 (\K) ) = \Aff_1 (\K). \]
\end{lemma}

\begin{proof}
Let $A=\smat{a}{b}{c}{d}\in\PGL_2(\K)$. Then $A\smat{1}{1}{0}{1}A^{-1}\in\Aff_1(\K)$ if and only if $-c^2=0$, i.e.\ if and only if $A\in\Aff_1(\K)$.
\end{proof}

From now on we will only consider the groups $\PGL_2 ( \K)$ and $\T_f$ when the base field $\K$ is equal to $\C (y)$.

\section{Borel subgroups of  $\PGL_2 ( \C (y) ) \subseteq \bp$} \label{section: Borel subgroups of PGL(2,C(y))}

The aim of this section is to prove Theorem~\ref{theorem : the Borel subgroups of PGL(2,C(y))} which describes all Borel subgroups of the closed subgroup $\PGL_2 ( \C (y) )$ of $\bp$. It turns out that these Borel subgroups coincide with the $\C (y)$-Borel subgroups of $\PGL_2 ( \C (y) )$ which have been defined in Definition~\ref{definition: K-Borel subgroup of G(K)} and described in Theorem~\ref{theorem: K-Borel subgroups of PGL(2,K)}.

The group $\PGL_2 ( \C (y) )$ admits two natural topologies. The first one (which is the one of most interest to us) is the topology induced by the inclusion $\PGL_2 ( \C (y) ) \subseteq \bp$. The second one is the $\C (y)$-topology which has been defined in Definition~\ref{definition: the K-topology}. We will refer to the first one as the $\bp$-topology and to the second as the $\C(y)$-topology.
When not specified, the topology is always understood to be the $\bp$-topology. Finally note that the $\C (y)$-topology will only be used in this section and in the proof of Theorem~\ref{theorem: any Borel subgroup of PGL_2 (C (y) ) rtimes Aff_1 contains a Borel subgroup of PGL_2 (C (y) )}.

We begin with the following result:

\begin{lemma} \label{lemma: the Bir(P2) topology is finer than the C(y)-topology}
The $\bp$-topology on $\PGL_2(\C(y))$ is finer than the $\C(y)$-topology.
\end{lemma}

\begin{proof}
Any $\C(y)$-closed set in $\PGL_2(\C(y))$ is an intersection of sets of the form
\[F_P:=\left\{ \left(\begin{smallmatrix}\alpha&\beta\\ \gamma&\delta\end{smallmatrix}\right)  \in \PGL_2(\C(y)) \,|\,P(\alpha,\beta,\gamma,\delta)=0\right\}\subset\PGL_2(\C(y)),\]
where $P\in\C[y][X_1,X_2,X_3,X_4]$ is a polynomial with coefficients in $\C[y]$ which is homogeneous with respect to the variables $X_1,\dots,X_4$. Therefore it is enough to show that such a set $F_P$ is $\bp$-closed. By Lemma~\ref{lemma:useful-characterisation-of-closed-subsets-of-Bir(Pn)} we need to show that $\pi_d^{-1}(F_P \cap \bp_d) \subseteq \bpGothic_d$ is closed for any positive $d \ge 2$. Denote by  $\MM$ the projective space associated with the complex vector space of $4$-tuples $(\alpha, \beta, \gamma, \delta)$ where $\alpha, \beta, \gamma, \delta  \in \C [y,z]$ are homogeneous polynomials of respective degrees $d-1, d, d-2, d-1$. The equivalence class of $(\alpha, \beta, \gamma, \delta)$ will be denoted by $[ \alpha : \beta : \gamma : \delta ]$. Denote by $ Z \subseteq \bpGothic_d \times \MM$ the closed subset given by elements $([f_1:f_2:f_3], [ \alpha : \beta : \gamma : \delta ])$ satisfying the two following conditions:

\begin{equation} \forall \, i,j \in \{1,2,3 \}, \; f_i g_j =f_j g_i, \label{equation:  f_i g_j =f_j g_i} \end{equation}
where $(g_1, g_2, g_3) : =   \Big( \alpha (y,z) x + \beta(y,z),  y ( \gamma (y,z) x + \delta(y,z) ),  z ( \gamma (y,z) x + \delta(y,z) ) \Big)$;
\begin{equation} P (  \alpha (y,1) , \beta (y,1),  \gamma (y,1),  \delta (y,1) ) =0. \label{equation: P(alpha, beta, gamma, delta)=0} \end{equation}
The condition \eqref{equation:  f_i g_j =f_j g_i}  means that the elements $[f_1:f_2:f_3]$ and $ [g_1 : g_2 : g_3]$ define the same rational map $\p^2 \dasharrow \p^2$. Note that we then have $\gamma (y,1) x + \delta (y,1) \neq 0$ and that in affine coordinates this rational map is the following birational map
\[ \A^2 \dasharrow \A^2, \quad (x,y) \dasharrow \bigg( \frac{\alpha (y,1) x + \beta (y,1) }{\gamma (y,1) x + \delta (y,1) }, y \bigg).\]
The condition \eqref{equation: P(alpha, beta, gamma, delta)=0} means that the element $\left(\begin{smallmatrix} \alpha (y,1) & \beta (y,1) \\ \gamma (y,1) & \delta (y,1) \end{smallmatrix}\right)$ of  $\PGL_2(\C(y))$ belongs to $F_P$.

We leave it as an easy exercise for the reader to check that the condition
\eqref{equation: P(alpha, beta, gamma, delta)=0} is actually a closed condition on the coefficients of the polynomials $\alpha, \beta, \gamma, \delta$. Since the first projection $p_1 \colon \bpGothic_d \times \MM \to \bpGothic_d$ is a closed morphism, the equality $\pi_d^{-1}(F_P \cap \bp_d) = p_1 (Z)$ shows that $\pi_d^{-1}(F_P \cap \bp_d)$ is closed in $\bpGothic_d$.
\end{proof}

Here is an example showing that the $\bp$-topology on $\PGL_2(\C(y))$ is strictly finer than the $\C(y)$-topology.

\begin{example}
The set $\left\{ \smat{1}{t}{0}{1}, \; t \in \C \right\} \subseteq \PGL_2 ( \C (y) ) $ is closed for the $\bp$-topology, but its $\C(y)$-closure is $\left\{ \smat{1}{t}{0}{1}, \; t \in \C(y) \right\}$.
\end{example}

\begin{lemma} \label{lemma: closed connected subgroup of PGL(2,C(y)) for the C(y)-topology are Bir(P2)-connected}
Any closed connected subgroup of $\PGL_2 (\C (y) )$ for the $\C (y)$-topology is closed connected for the $\bp$-topology.
\end{lemma}

\begin{proof}
By Lemma~\ref{lemma: the Bir(P2) topology is finer than the C(y)-topology}, each $\C(y)$-closed subset of  $\PGL_2 (\C (y) )$ is $\bp$-closed.
It's enough to prove that when $\K = \C (y)$, the six subgroups $H_1, \ldots, H_6$
listed in the Cases \eqref{H={id}}, \ldots,  \eqref{H=PGL_2(K)} of Lemma~\ref{lemma: closed connected subgroups of PGL_2(K)}, are $\bp$-connected. Since $H_1 = \{ \id \}$, $H_5 = H_3 \rtimes H_2$, and $H_6 = \langle H_5, (H_5)^t \rangle$ (where $(H_5)^t$ denotes the transpose of $H_5$), it's enough to show that $H_2, H_3,H_4$ are $\bp$-connected. Using  Lemma~\ref{lemma: a useful criterion of connectedness for closed subsets of Bir(Pn)}, it's enough to show that these three subgroups satisfy the assertion \eqref{any two points of F can be joined by a curve} of that lemma, i.e. that for any $i \in \{2,3,4 \}$ and any  $\varphi, \psi \in H_i$, there exists a connected (not necessarily irreducible) curve $C$ and a morphism $\lambda \colon C \to \bp$  (see Definition~\ref{definition: morphism to Bir(W)}) whose image satisfies $ \varphi, \psi \in \Image (\lambda ) \subseteq H_i$.

Let $h:= \smat{a }{0}{0}{1}$,  $a \in \C (y) \setminus \{0,1 \}$,  be a nontrivial element of $H_2$. Define the curve $C$ by $C:= \A^1$ if $a \notin \C$ and $C := \A^1 \setminus \{ \frac{1}{1-a} \}$ otherwise. Then, the image of the morphism $C \to \bp$, $t \mapsto \smat{ 1-t + t a}{0}{0}{1}$ is contained in $H_2$, and it connects $h$ and $\id$. Hence $H_2$ is connected.

If $h:=  \smat{1}{a}{0}{1}$, $a \in \C (y)$, is any element of $H_3$, the image of the morphism $\A^1 \to \bp$, $t \mapsto \smat{1}{ta}{0}{1}$ is contained in $H_3$, and it connects $h$ and $\id$. Hence $H_3$ is connected.

Let $h:=  \smat{a }{bf}{b}{a} $,  $a,b  \in \C (y) ^*$, be any element of $H_4 = \T_f$ which is different from $1$ and  $\iota_f $. The two morphisms $\A^1 \to \bp$, $t \mapsto  \smat{a }{tbf}{tb}{a}$ and $\A^1 \to \bp$, $t \mapsto  \smat{ta }{bf}{b}{ta}$ both have their images contained in $H_4$. The first one connects $h$ and $\id$, and the second connects $h$ and $\iota_f = \smat{0 }{f}{1}{0}  $. Hence $H_4$ is connected.
\end{proof}

We can now prove the main result of this section.

\begin{theorem}  \label{theorem : the Borel subgroups of PGL(2,C(y))}
Up to conjugation, any Borel subgroup of $\PGL_2( \C (y)  ) $ is equal to either $\Aff_1 ( \C (y)  )$ or some $\T_f$, where $f$ is a nonsquare element of $\C (y) $. Moreover, any closed connected solvable subgroup of $\PGL_2 (\C (y) )$ is contained in a Borel subgroup.
\end{theorem}

\begin{proof}
By Theorem~\ref{theorem: K-Borel subgroups of PGL(2,K)}, it's enough to show the two following assertions:
\begin{enumerate}
\item \label{Each closed connected solvable subgroup of PGL_2 (C (y) )is is contained in a C (y)-Borel subgroup}
Each closed connected solvable subgroup of $\PGL_2 (\C (y) )$ is contained in a $\C (y)$-Borel subgroup of $\PGL_2 (\C (y) )$;
\item \label{The Borel subgroups of  PGL_2 ( C (y) ) coincide with the C (y)-Borel subgroups}
The Borel subgroups of  $\PGL_2 (\C (y) )$ coincide with the $\C (y)$-Borel subgroups of \linebreak  $\PGL_2 (\C (y) )$.
\end{enumerate}

\eqref{Each closed connected solvable subgroup of PGL_2 (C (y) )is is contained in a C (y)-Borel subgroup} Let $H$ be a closed connected solvable subgroup of $\PGL_2( \C(y))$ and let $\overline{H}$ be its $\C(y)$-closure in $\PGL_2( \C(y) )$. By Lemma~\ref{lemma: the Bir(P2) topology is finer than the C(y)-topology}, $H$ is $\C(y)$-connected, and hence $\overline{H}$ is $\C(y)$-connected as well. The group $\overline{H}$ being a $\C(y)$-closed connected solvable subgroup of $\PGL_2( \C(y) )$, it is contained in a $\C(y)$-Borel subgroup of $\PGL_2 (\C (y) )$ by Theorem~\ref{theorem: K-Borel subgroups of PGL(2,K)}.

\eqref{The Borel subgroups of  PGL_2 ( C (y) ) coincide with the C (y)-Borel subgroups}
We will successively prove that each Borel subgroup of $\PGL_2 (\C (y) )$ is a $\C (y)$-Borel subgroup of $\PGL_2 ( \C (y) )$ and that each $\C(y)$-Borel subgroup of $\PGL_2 (\C (y) )$ is a Borel subgroup of $\PGL_2 (\C (y) )$.

Let $B$ be a Borel subgroup of $\PGL_2 (\C (y) )$. By \eqref{Each closed connected solvable subgroup of PGL_2 (C (y) )is is contained in a C (y)-Borel subgroup}, $B$ is contained in a $\C (y)$-Borel subgroup $B'$ of $\PGL_2 ( \C (y) )$. But $B'$ is a  closed connected solvable subgroup of $\PGL_2 (\C (y) )$ (see Lemma~\ref{lemma: closed connected subgroup of PGL(2,C(y)) for the C(y)-topology are Bir(P2)-connected}). Hence, by maximality of $B$, we get $B=B'$, showing that $B$ is actually a $\C(y)$-Borel subgroup of $\PGL_2 (\C (y) )$.

Let now $B$ be a $\C(y)$-Borel subgroup of $\PGL_2 (\C (y) )$. First of all, $B$ is a closed connected solvable subgroup of $\PGL_2( \C(y))$ (see Lemma~\ref{lemma: closed connected subgroup of PGL(2,C(y)) for the C(y)-topology are Bir(P2)-connected}). Secondly, let's check that $B$ is maximal for this property. Let's assume that we have $B \subseteq H$, where $H$ is a closed connected solvable subgroup of $\PGL_2( \C(y))$.  By \eqref{Each closed connected solvable subgroup of PGL_2 (C (y) )is is contained in a C (y)-Borel subgroup}, $H$ is contained in a $\C(y)$-Borel subgroup $B'$ of $\PGL_2 (\C (y) )$. Hence we have $B \subseteq H \subseteq B'$ where $B,B'$ are two $\C (y)$-Borel subgroups of  $\PGL_2 (\C (y) )$. By maximality of $B$, we get $B= B'$ from which we get $H =B$. We have actually shown that $B$ is a Borel subgroup of $\PGL_2 (\C (y) )$.
\end{proof}

We end this section by showing that any closed connected solvable subgroup of $\bp$ is contained in a Borel subgroup. Actually we need a slightly stronger result, to be used later (see Proposition~\ref{proposition: pre-Borel subgroups of closed subgroups of Bir(P2) are contained in Borel subgroups} below). The proof is based on the following result.

\begin{lemma} \label{lemma: maximal derived length of a closed connected solvable subgroup of Bir(P2)}
The maximal derived length of a closed connected solvable subgroup of $\bp$ is $4$.
\end{lemma}

\begin{proof}
Let's begin by proving that the derived length of a closed connected solvable subgroup $H$ of $\bp$ is at most $4$. By Theorem~\ref{theorem: any closed connected solvable subgroup of Bir(P2) is conjugate to a subgroup of PGL_2(C(y)) rtimes Aff_1}, we may assume that $H$ is contained in $\PGL_2 (\C(y) ) \rtimes \Aff_1$. But then, we have $\DDD^2H \subseteq \PGL_2 ( \C (y) )$ and since $\DDD^2H$ is a closed connected solvable subgroup (Lemma~\ref{lemma: DG and (mathcal-D)G}\eqref{G-closed-connected-implies-(mathcal-D)G-connected}) its derived length is at most $2$ (Theorem~\ref{theorem : the Borel subgroups of PGL(2,C(y))}). We have actually proven that the derived length of $H$ is at most $4$. We conclude the proof by noting that the derived length of $\mathcal B_2$ is $4$  (see Proposition~\ref{proposition: the derived length of Bn}).
\end{proof}

\begin{proposition} \label{proposition: pre-Borel subgroups of closed subgroups of Bir(P2) are contained in Borel subgroups}
Let $G$ be a closed subgroup of $\bp$. Then, any closed connected solvable subgroup of  $G$ is contained in a Borel subgroup of $G$.
\end{proposition}

\begin{proof}
If $H$ is a closed connected solvable subgroup of $G$, denote by $\FF$ the set of connected solvable subgroups of $G$ containing $H$. The bound given in Lemma~\ref{lemma: maximal derived length of a closed connected solvable subgroup of Bir(P2)} implies that $ (\FF, \subseteq )$ is inductive. Indeed, if $(H_i)_{i \in I}$ is a chain in $\FF$, i.e.\ a totally ordered family of $\FF$, then the group $\cup_i H_i$ is connected solvable (for details, see \cite[Proposition 3.10]{FurterPoloni2018}). Therefore, by Zorn's lemma, the poset $\FF$ admits a maximal element $B$. Since $\overline{B}$ is connected and solvable, this shows that $B = \overline{B}$. Hence $B$ is closed. Being maximal among the closed connected solvable subgroups of $G$, it is a Borel subgroup of $G$.
\end{proof}

\section{The groups $\T_f$} \label{section: The groups TT_f}

If $f$ is any nonsquare element of $\C (y)$, it follows from Theorem~\ref{theorem : the Borel subgroups of PGL(2,C(y))} that the group 
\[ \T_f =  \left\{ \smat{a }{bf}{b}{a}, \; a,b \in \C (y), \;  (a,b) \neq (0,0) \right\} \]
is a Borel subgroup of $\PGL_2 (\C (y) )$.

\begin{definition} \label{defintion: the geometric degree of a rational function}
The geometric degree of a rational function $f \in \C (y)$ is denoted  $\degmax (f)$. If $f$ is written $ \alpha / \beta$ where $\alpha, \beta \in \C [ y]$ are coprime, it is given by
\[ \degmax (f) = \max \{ \deg (\alpha), \deg ( \beta) \} \in \Z_{\geq 0} .\]
\end{definition}

The geometric degree $\degmax (f)$ is the number of preimages (counted with multiplicities) of any point of $\p^1$ for the morphism $f \colon \p^1 \to \p^1$. It satisfies the following elementary properties:

\begin{lemma}
For each $f,g \in \C(y)$ we have:
\begin{enumerate}
\item \label{upperbound of the geometric degree of the sum}
$\degmax (f+g) \leq   \degmax (f) +  \degmax (g)$.
\item \label{upperbound of the geometric degree of the product}
$\degmax (fg) \leq   \degmax (f) +  \degmax (g)$.
\item  \label{upperbound of the geometric degree of the composition}
$\degmax (f \circ g) =   \degmax (f)   \degmax (g)$ (when the composition $f \circ g$ exists).
\end{enumerate}
\end{lemma}

\begin{proof}
Write $f= \alpha / \beta$ and $g = \gamma / \delta$ where $\alpha, \beta \in \C [ y]$ are coprime and $\gamma, \delta \in \C [ y]$ are coprime.

\noindent \eqref{upperbound of the geometric degree of the sum}
The equality $ f+ g = {\dis \dfrac{ \alpha \delta + \beta \gamma} {\beta \delta} }$ yields

$
\begin{array}{ll}
\degmax (f+g)  &\leq  \max \{ \deg ( \alpha \delta + \beta \gamma ), \deg ( \beta \delta) \}  \\
& \leq \max \{ \deg ( \alpha \delta), \deg (\beta \gamma ), \deg ( \beta \delta) \} \\
& \leq \degmax (f) + \degmax (g).
\end{array}
$

\noindent \eqref{upperbound of the geometric degree of the product} This is straightforward.

\noindent \eqref{upperbound of the geometric degree of the composition} This is obvious from the geometric interpretation.
\end{proof}

Recall that a subgroup of $\bp$ is called algebraic if it is closed and bounded for the degree. Also an element $f$ of $\bp$ is called algebraic if it belongs to an algebraic subgroup of $\bp$. Equivalently this means that the sequence $n \mapsto \deg f^n$, $n \in \Z_{\ge 0}$, is bounded. For elements in $\PGL_2(\C(y))$, we will use the following characterisation of algebraic elements by Cerveau--D\'eserti \cite[Theorem~A]{CerveauDeserti2012}. 

\begin{lemma} \label{lemma: criterion of algebraicity in PGL(2,C(y))}
Let $A$ be an element of $\GL_2 ( \C(y) )$ and let  $\bar A$ be its class in $\PGL_2 ( \C (y) )$. Then, the element  $\bar A$ of $\PGL_2 ( \C (y) ) \subseteq \bp$ is algebraic if and only if its Baum-Bott index
\[ \BaumBott ( \bar A) := \tr^2(A) / \det (A)  \]
belongs to $\C$. 
\end{lemma}

We now describe the algebraic elements of $\PGL_2 ( \C (y) )$ up to conjugation:

\begin{lemma} \label{lemma: algebraic elements of PGL(2,C(y))}
Any algebraic element of $\PGL_2 ( \C (y) )$ is conjugate to one of the following elements:
\begin{enumerate}
\item \label{diagonal form}
$\smat{a }{0}{0}{1}$, where $a \in \C^*$;
\item \label{upper-triangular form}
$\smat{1}{1}{0}{1}$;
\item \label{anti-diagonal form}
$\smat{0}{f}{1}{0}$, where $f$ is a nonsquare element of $\C(y)$.
\end{enumerate}
\end{lemma}

\begin{proof}
Let $ A \in \GL_2 ( \C (y) )$ be an element whose class $ \bar A$ in $\PGL_2 ( \C (y) )$ is algebraic. We then have $\BaumBott ( \bar A) \in \C$. If $\BaumBott ( \bar A) =0$,  we get $\tr (A) = 0$ and $(\bar A)^2 = \id$. Since $A$ is not a homothety, there exists a vector $u \in (\C (y) )^2$ such that $u$ and $v:=Au$ are linearly independent. If $P= (u, v)  \in \GL_2 ( \C (y) )$ is the element whose first  (resp.\ second) column is $u$ (resp.\ $v$), we have $P^{-1}AP = \smat{0}{f}{1}{0}$ for some nonzero element $f$ of $\C(y)$. If $f$ is a square, then the class of this matrix is conjugate to $\smat{-1 }{0}{\hspace{1.5mm}0}{1}$ in $\PGL_2 ( \C (y) )$ and we are in case \eqref{diagonal form}. If $f$ is not a square, we are in case \eqref{anti-diagonal form}. Assume now that $\BaumBott ( \bar A) \in \C^*$. As in the proof of Lemma~\ref{lemma: criterion of algebraicity in PGL(2,C(y))}, up to dividing the matrix $A \in \GL_2 ( \C (y) )$ by $\tr (A)$, we may assume that $\tr (A)$ and $\det (A) \in \C$. Hence, the eigenvalues of $A$ are complex numbers, and $A$ is conjugate to a Jordan matrix $\smat{a }{0}{0}{b}$ or $\smat{a }{1}{0}{a}$, where $a,b \in \C^*$. In the first case, the class of this element is equal to the class of $\smat{a/b }{0}{0}{1}$ and we are in case \eqref{diagonal form}. In the second case, the matrix $\smat{a }{1}{0}{a}$ is conjugate to the matrix $\smat{a }{a}{0}{a}$ whose class is equal to $\smat{1}{1}{0}{1}$ in $\PGL_2(\C(y) )$. We are in case \eqref{upper-triangular form}.
\end{proof}

\begin{lemma} \label{lemma: algebraic elements of Tf}
The group $\T_f \subseteq \bp$ contains exactly two algebraic elements  which correspond to the elements $\smat{a }{bf}{b}{a}$ with $ab= 0$. If $b=0$, we get the identity element of $\bp$ and if $a =0$, we get the involution
\[ \iota_f \colon \p^2 \dasharrow \p^2, \quad (x,y) \dasharrow ( f(y) /x, y).\]
\end{lemma}

\begin{proof}
We apply the criterion for algebraicity given in Lemma~\ref{lemma: criterion of algebraicity in PGL(2,C(y))}. If $A:= \smat{a }{bf}{b}{a}$ with $ab \neq 0$, then we have
\[ \BaumBott (A) =  \frac{ \tr^2 (A) }{ \det (A) } = \frac{4a^2}{a^2-b^2f}   = \frac{4}{1 - \frac{b^2f}{a^2} } \]
 and this element is nonconstant because otherwise we would have $b^2f /a^2 \in \C^*$, proving that $f$ is a square. A contradiction.
\end{proof}

The previous result directly implies the following one:

\begin{lemma} \label{lemma: T_f does not contain connected closed algebraic subgroups}
The group $\T_f$ contains no nontrivial connected closed algebraic subgroup of $\bp$.
\end{lemma}

Lemma~\ref{lemma: T_f does not contain connected closed algebraic subgroups} gives us:

\begin{lemma} \label{lemma: rk(T_f) = 0}
We have $\rk (\T_f ) =0$.
\end{lemma}

\begin{proposition} \label{proposition: normaliser of T_f in Jonq is the centraliser of iota_f}
Let $f$ be a nonsquare element of $\C(y)$. Then we have
\[ \Nor_{\Jonq}(\T_f) = \Cent_{\Jonq}(\iota_f) .\]
\end{proposition}

\begin{proof}
The inclusion $  \Nor_{\Jonq}(\T_f)  \subseteq \Cent_{\Jonq}(\iota_f)$ directly follows from Lemma~\ref{lemma: iota_f is the unique involution of T_f}.

Let now $\varphi$ be an element of $\Cent_{\Jonq}(\iota_f)$. Note that $\varphi \T_f \varphi^{-1}$ is a Borel subgroup of $\PGL_2 (\C (y) )$. For seeing it, we can write $\varphi = uv$ with $u \in \PGL_2 ( \C (y) )$, $v \in \PGL_2$, and observe that $\varphi \T_f \varphi^{-1}= uv \T_f v^{-1} u^{-1} = u   \T_{f \circ v^{-1}}  u^{-1}$. Since  $\varphi \T_f \varphi^{-1}$ and $\T_f$ are two Borel subgroups of $\PGL_2 ( \C(y) )$ whose intersection is nontrivial (because it contains $\iota_f$), they are equal by Proposition~\ref{proposition: a remarkable characterisation of equality or conjugation}. Hence we have shown that $\varphi$ belongs to $  \Nor_{\Jonq}(\T_f)$ and this proves the inclusion $\Cent_{\Jonq}(\iota_f) \subseteq  \Nor_{\Jonq}(\T_f)  $.
\end{proof}

For later use, we prove the following basic result.

\begin{lemma} \label{lemma: unique rational fibration preserved and a property of some conjugants}
Let $\alpha \in \Jonq$ be a Jonqui\`eres transformation. The following assertions are equivalent:
\begin{enumerate}
\item \label{a unique rational fibration is preserved}
$\alpha$ preserves a unique rational fibration;
\item \label{an implication}
$\forall \, \varphi \in \bp, \quad \varphi \alpha \varphi^{-1} \in \Jonq \; \Longrightarrow  \; \varphi \in \Jonq$.
\end{enumerate}
\end{lemma}

\begin{proof}
Let's note that the rational fibration $\pi \colon \p^2 \dasharrow  \p^1$ is preserved by $\alpha$ if and only if the rational fibration $\pi \varphi ^{-1}$ is preserved by $ \varphi \alpha \varphi^{-1}$.

Set $\Pi \colon \p^2 \dasharrow \p^1$, $[x:y:z] \dasharrow [y:z]$. Note that $\Pi$  corresponds to the fibration $y = {\rm const}$.

\eqref{a unique rational fibration is preserved} $\Longrightarrow$ \eqref{an implication} Let $\varphi \in \bp$ be such that $\varphi \alpha \varphi^{-1} \in \Jonq$. This means that there exists $\beta \in \Aut (\p^1)$ such that $\Pi (\varphi \alpha \varphi^{-1}) = \beta \Pi$. This is equivalent to $(\Pi \varphi) \alpha = \beta (\Pi \varphi) $. Hence $\Pi \varphi$ is preserved by $\alpha$. Therefore, there exists $\beta' \in \Aut (\p^1)$ such that  $\Pi \varphi = \beta \Pi$ and $\varphi \in \Jonq$.

\eqref{an implication} $\Longrightarrow$ \eqref{a unique rational fibration is preserved} Assume that $\pi \colon \p^2 \dasharrow \p^1$ is a rational fibration preserved by $\alpha$. By Definition~\ref{definition: rational fibration}\eqref{preserved rational fibration} there exists $\varphi \in \bp$ satisfying $\pi = \Pi \varphi$ and $\varphi \alpha \varphi^{-1} \in \Jonq$. The assumption we've made yields $\varphi \in \Jonq$, i.e.\ $\Pi \varphi = \beta \Pi$ for some $\beta \in \Aut (\p^1)$. We obtain $\pi = \beta \Pi$ which shows that $\pi$ is equivalent to $\Pi$.
\end{proof}

Following the literature a non algebraic element of $\Jonq$ will be called a  Jonqui\`eres twist.

\begin{lemma} \label{lemma: a sufficient condition for an element to be Jonquieres}
Let $\varphi$ be an element of $\bp$. Then, the following assertions are equivalent:
\begin{enumerate}
\item \label{phi belongs to Jonq}
We have $\varphi \in \Jonq$.
\item \label{the conjugate is contained in PGL(2,C(y))}
We have $\varphi \T_f \varphi^{-1} \subseteq \PGL_2 ( \C (y) )$.
\item \label{the conjugate is contained in Jonq}
We have $\varphi \T_f \varphi^{-1} \subseteq \Jonq$.
\item \label{the conjugate of some Jonquieres twist of Tf belongs to Jonq}
$ \exists \, t \in \T_f \smallsetminus \{ \id,  \iota_f \}$ such that $\varphi t  \varphi^{-1} \in  \Jonq$.
\item \label{the conjugate of some Jonquieres twist of Jonq belongs to Jonq}
There exists a Jonqui\`eres twist $t \in \Jonq$ such that $\varphi t  \varphi^{-1} \in  \Jonq$.
\end{enumerate}
\end{lemma}

\begin{proof}
\eqref{phi belongs to Jonq} $\Longrightarrow$ \eqref{the conjugate is contained in PGL(2,C(y))} $\Longrightarrow$ \eqref{the conjugate is contained in Jonq} $\Longrightarrow$ \eqref{the conjugate of some Jonquieres twist of Tf belongs to Jonq} is obvious.

\noindent \eqref{the conjugate of some Jonquieres twist of Tf belongs to Jonq} $\Longrightarrow$ \eqref{the conjugate of some Jonquieres twist of Jonq belongs to Jonq} is a consequence of Lemma~\ref{lemma: algebraic elements of Tf}.

\noindent \eqref{the conjugate of some Jonquieres twist of Jonq belongs to Jonq} $\Longrightarrow$ \eqref{phi belongs to Jonq}. By \cite[Lemma 4.5]{DillerFavre2001} a Jonqui\`eres twist preserves a unique rational fibration. Hence the result follows from Lemma~\ref{lemma: unique rational fibration preserved and a property of some conjugants}.
\end{proof}

Lemma~\ref{lemma: a sufficient condition for an element to be Jonquieres} directly yields the next result:

\begin{proposition}
The group $\Jonq$ is equal to its own normaliser in $\bp$.
\end{proposition}

Consider the embedding $\C (y) ^* \hookrightarrow \PGL_2 (\C(y) )$, $ \lambda \mapsto d_{\lambda} = \smat{\lambda}{0}{0}{1}= ( \lambda (y) x , y) $.

Any element $\varphi \in  \C (y) ^* \rtimes \PGL_2  \subseteq \PGL_2( \C(y) ) \rtimes \PGL_2 = \Jonq$ can be (uniquely) written $\varphi = \mu \circ d_{\lambda}$, where $\mu  = (x, \mu(y) ) \in \PGL_2 $ and $\lambda \in \C (y)^*$. Equivalently, we have $\varphi =  ( \lambda (y) x , \mu(y) )$.
We then have:
\begin{equation} \varphi \circ \T_f  \circ \varphi^{-1} = \mu \circ d_{\lambda} \circ  \T_f \circ (d_{\lambda})^{-1} \circ \mu^{-1}  = \mu \circ \T_{\lambda^2 f} \circ \mu^{-1} = \T_ { (\lambda^2 f) \circ \mu^{-1} }. \label{equation: action by conjugation of the groups Tf} \end{equation}
We summarise this computation in the following statement

\begin{lemma} \label{lemma: the action of C (y)^* rtimes PGL_2 on C(y)}
Considering the action of  $\C (y) ^* \rtimes \PGL_2 $ on $\C (y)$ given by
\[ \forall \, \varphi = ( \lambda  , \mu ) \in \C (y) ^* \rtimes \PGL_2, \; \forall \, f \in \C (y), \quad \varphi . f := (\lambda^2 f) \circ \mu^{-1},\]
we have  $\varphi  \T_f \varphi^{-1}  = \T _{\varphi.f}$ for any nonsquare element $f$ of $\C (y)$.
\end{lemma}

\begin{definition} \label{definition: odd support of a rational function}
If $f \in \C (y)$ is a nonzero rational function, define its odd support $\So (f)$ as the support of the divisor $\overline{  \divisor (f) }$ where  $\divisor (f) \in  \Div ( \p^1)$ is the usual divisor of $f$, and $\overline{  \divisor (f) }$ denotes its image by the canonical map $\Div ( \p^1) \to \Div ( \p^1) \otimes_{\Z} \Z_2$.

Alternatively, if $ \vv_a (f)$ is the order of vanishing of $f$ at the point $a$ (counted positively if $f$ actually vanishes at $a$ and negatively if $f$ admits a pole at $a$), we have 
\[ \So (f) = \{ a \in \p^1, \; \vv_a (f) \text{ is odd} \} .\]
\end{definition}

If $f(y) = c \prod_i (y-a_i)^{n_i}$, where $c \in \C^*$, the $a_i$ are distinct complex numbers, and the $n_i$ are integers, we have
\[ \So (f) = \{ a_i, \; n_i \text{ is odd} \}
\hspace{10mm}  \text{if} \hspace{2mm} \sum_i n_i \text{ is even, and}  \] \[ \So (f) = \{ a_i, \; n_i \text{ is odd} \} \cup \{ \infty \}
\hspace{10mm}  \text{if} \hspace{2mm} \sum_i n_i \text{ is odd.}  \]
Note that $\So (f) =\emptyset $ if and only if $f$ is a square. When $f$ is not a square, let $\g$ denote the genus of the curve $x^2=f(y)$. The following formula is a well-known consequence of the Riemann-Hurwitz formula:
\[ 2 \g  + 2 =  | \So (f) |  .\]
We will constantly use the following straightforward lemma:

\begin{lemma} \label{lemma: fov/g is a square iff v stabilizes the odd support of f}
Let $f,g \in \C (y)^*$ and let $v \in \PGL_2$ be a homography. Then, the following assertions are equivalent:
\begin{enumerate}
\item \label{fov/g is a square}
We have $\dfrac{f(v(y))}{g(y)} = \lambda^2(y)$  for some $\lambda \in \C (y)$.
\item \label{v sends the odd support of g onto the odd support of f}
\rule{0mm}{5mm}
We have $v ( \So (g) ) = \So (f)$.
\end{enumerate}
\end{lemma}

\begin{proposition} \label{proposition: equivalent conditions for T_f and T_g to be conjugate in Bir(P2)}
Let $f,g$ be nonsquare elements of $\C (y)$. Then, the following assertions are equivalent:
\begin{enumerate}
\item \label{conjugate in bir(P2)} $\T_f$ and $\T_g$ are conjugate in $\bp$;
\item \label{conjugate in Jonq} $\T_f$ and $\T_g$ are conjugate in $\Jonq$;
\item \label{conjugate in C(y)^* rtimes PGL_2} $\T_f$ and $\T_g$ are conjugate by an element of  $\C (y) ^* \rtimes \PGL_2$;
\item \label{in the same orbit for the C(y)^* rtimes PGL_2 action} $f$ and $g$ are in the same orbit for the action of $\C (y) ^* \rtimes \PGL_2$ on $\C (y)$;
\item \label{iota_f and iota_g are conjugate}
The involutions $\iota_f$ and $\iota_g$ are conjugate in $\bp$;
\item \label{The fields C[sqrt{f}] and C[sqrt{g}] are C-isomorphic}
The fields $\C (y) [ \sqrt{f}] $ and $\C (y) [ \sqrt{g}] $ are $\C$-isomorphic;
\item \label{The odd supports belong to the same orbit}
There exists $\mu \in \Aut (\p^1)$ such that $\So (g) = \mu \big(  \So (f) \big)$;
\item \label{The curves are isomorphic}
The hyperelliptic curves associated with $x^2 = f(y)$ and $x^2 = g(y)$ are isomorphic.
\end{enumerate}
\end{proposition}

\begin{proof}
\eqref{conjugate in bir(P2)} $\Longrightarrow$ \eqref{conjugate in Jonq} If $\phi \, \T_f \, \phi^{-1}= \T_g$ for some $\phi \in \bp$, it follows by Lemma~\ref{lemma: a sufficient condition for an element to be Jonquieres} that $\phi \in \Jonq$, since $\T_g \subset \PGL_2(\C(y))$. 

\eqref{conjugate in Jonq} $\Longrightarrow$ \eqref{conjugate in C(y)^* rtimes PGL_2}  If $\mu  \phi \,  \T_f \,  \phi^{-1}  \mu^{-1} = \T_g$ for some $\mu \in \PGL_2 $ and $\phi=\smat{\alpha}{\beta}{\gamma}{\delta}\in\PGL_2(\C(y))$, there exists an element $\smat{a}{bf}{b}{a} \in \T_f$ such that 
\[\smat{\alpha}{\beta}{\gamma}{\delta}\smat{a}{bf}{b}{a}\smat{\alpha}{\beta}{\gamma}{\delta}^{-1}=\mu^{-1}\smat{0}{g}{1}{0}\mu\] Comparing traces, we obtain $a=0$ so that $\smat{a}{bf}{b}{a}=\smat{0}{f}{1}{0}$. It follows that 
\[\smat{\delta (\mu^{-1}.g)}{\gamma f(\mu^{-1}.g)}{\beta}{\alpha f}=\smat{\alpha}{\beta}{\gamma}{\delta}.\] 
Thus $\beta=0$ if and only if $\gamma=0$. If $\beta=\gamma=0$ we have $\mu\circ\phi\in\C(y)^*\rtimes \PGL_2 $ as desired, so we may assume that $\beta$ and $\gamma$ are both nonzero.
Then $\beta/\gamma=\gamma f(\mu^{-1}.g)/\beta$ so that $f(\mu^{-1}.g)$ is a square. This is equivalent to $\mu^{-1}.g/f$ being a square, so $\mu^{-1}.g=fh^2$ for some $h \in \C(y)$. It follows that $d_h \T_f (d_h)^{-1} = \T_{\mu^{-1}.g}$ and hence $\mu   d_h \T_f (d_h)^{-1} \mu^{-1} = \T_g$.

\eqref{conjugate in C(y)^* rtimes PGL_2} $\Longrightarrow$ \eqref{in the same orbit for the C(y)^* rtimes PGL_2 action} By assumption we have $\mu\smat{\lambda}{0}{0}{1}\smat{a}{bf}{b}{a}\smat{1}{0}{0}{\lambda}\mu^{-1}=\smat{0}{g}{1}{0}$ for some $\mu \in \PGL_2 ,\,\lambda \in \C(y)^*$ and $\smat{a}{bf}{b}{a} \in \T_f$. Comparing traces, we obtain $a=0$ and it follows that $ (\lambda^2f) \circ \mu^{-1}  =g$. Hence $\phi.f=g$ with $\phi =  \mu \circ  d_\lambda  \in \C(y)^*\rtimes \PGL_2 $.

\eqref{in the same orbit for the C(y)^* rtimes PGL_2 action} $\Longrightarrow$ \eqref{conjugate in bir(P2)} We have $\varphi  \,  \T_f  \, \varphi^{-1}  = \T _{\varphi.f}$ for $\phi \in \C(y)^*\rtimes \PGL_2 $ and $f \in \C(y)$. It is straightforward to check that the set of squares in $\C(y)$ is invariant for the $\C(y)^*\rtimes \PGL_2 $-action on $\C(y)$, so if $f$ is not a square, then neither is $\phi.f$.

\eqref{conjugate in bir(P2)} $\Longrightarrow$ \eqref{iota_f and iota_g are conjugate}
This follows from Lemma~\ref{lemma: iota_f is the unique involution of T_f}.

\eqref{iota_f and iota_g are conjugate} $\Longrightarrow$ \eqref{The fields C[sqrt{f}] and C[sqrt{g}] are C-isomorphic}  Since the set of fixed points of the involution $\iota_f$ is the curve $x^2 = f(y)$, the conclusion follows from Lemma~\ref{lemma: two involutions are conjugate iff they have the same Normalised Fixed curve} below (see also Definition~\ref{definition: Normalised Fixed Curve}).

\eqref{The fields C[sqrt{f}] and C[sqrt{g}] are C-isomorphic} $\Longrightarrow$ \eqref{The odd supports belong to the same orbit} Denote by $\pi_f \colon C_f \to \p^1$ the $2:1$ morphism corresponding to the inclusion $\C (y) \subseteq \C (y) [ \sqrt{f}] $. Note that $\So (f)$ is equal to the ramification locus of $\pi_f$. We will consider three cases depending on the  genus $\g$ of $C_f$.

If $\g \geq 2$, then it is well-known that $\pi_f$ is the only $2:1$ morphism to $\p^1$ up to left composition by an auto\-morphism of $\p^1$ (see \cite[Theorem III.7.3, page 101]{FarkasKra1992}). If  we take $\varphi \colon C_f \to C_g$ to be any isomorphism, then there exists an automorphism $\mu \in \Aut (\p^1)$ making the following diagram commute:
\[\xymatrix{
C_f \ar[d]_{\pi_f} \ar[r]^{\varphi} & C_g  \ar[d]^{\pi_g} \\
\p^1  \ar[r]^{\mu} &  \p^1}\]
The equality $\pi_g  \varphi = \mu  \pi_f$ shows that $\pi_g$ and $\mu  \pi_f$ have the same ramification, i.e.\ $ \So (g) =$  $\So (f   \mu^{-1} ) $, i.e.\ $\So (g) = \mu \big(  \So (f) \big)$.

If $ \g=1$, it is well-known that if the two elliptic curves $C_f$ and $C_g$ are isomorphic,
then there exists an automorphism of $\p^1$ sending the 4 ramification points of $\pi_f$ onto the 4 ramification points of $\pi_g$ (see \cite[(IV, 4.4), page 318]{Hartshorne77}).

If $\g=0$, then, the two ramification loci of $\pi_f$ and $\pi_g$ have 2 elements. It is clear that there exists an automorphism of $\p^1$ sending the first ramification locus onto the other.

\eqref{in the same orbit for the C(y)^* rtimes PGL_2 action} $\Longleftrightarrow$ \eqref{The odd supports belong to the same orbit} is a direct consequence of Lemma~\ref{lemma: fov/g is a square iff v stabilizes the odd support of f}

\eqref{The fields C[sqrt{f}] and C[sqrt{g}] are C-isomorphic} $\Longleftrightarrow$ \eqref{The curves are isomorphic} It is well known that two projective smooth curves are isomorphic if and only if their function fields are isomorphic.
\end{proof}

The following lemma is an easy consequence of Proposition~\ref{proposition: equivalent conditions for T_f and T_g to be conjugate in Bir(P2)}.

\begin{lemma}
Let $f \in \C(y)$ be a nonsquare rational function. Then there exists a monic squarefree polynomial $g \in \C [y]$ of odd degree and divisible by $y$ such that $\T_f$ is conjugate to $\T_g$ in $\Jonq$.
\end{lemma}

The following definition already made in \cite{deFernex2004} is \cite[Definition 2.1]{Blanc2011}:

\begin{definition}[Normalised fixed curve: $\NFC$] \label{definition: Normalised Fixed Curve}
Let $\varphi \in \bp$ be a nontrivial element of finite order. If no curve of positive genus is fixed (pointwise) by $\varphi$, we say that $\NFC (\varphi) = \emptyset$; otherwise $\varphi$ fixes exactly one curve of positive genus (\cite{BayleBeauville2000}, \cite{deFernex2004}), and $\NFC (\varphi)$ is then the isomorphism class of the normalisation of this curve.
\end{definition}

The following result (proven by Bayle--Beauville in \cite[Proposition 2.7]{BayleBeauville2000}) is mentioned just after Definition 2.1 in \cite{Blanc2011}. It shows in particular that an involution $\varphi$ of $\bp$ is linearisable, i.e. conjugate to an automorphism of $\p^2$, if and only if $\NFC (\varphi) = \emptyset$.

\begin{lemma} \label{lemma: two involutions are conjugate iff they have the same Normalised Fixed curve}
Two involutions $\varphi_1, \varphi_2 \in \bp$ are conjugate if and only if \linebreak $\NFC (\varphi_1)=$  $\NFC (\varphi_2)$.
\end{lemma}

For later use we will now compute the neutral connected component $ \Nor_{\Jonq}( \T_f )^{\circ}$ of the normaliser of $\T_f$ in $\Jonq$ (the final result is given in Proposition~\ref{proposition: connected component of the normaliser of Tf in Jonq} below). This is how the proof goes: we begin by introducing the group ${\mathcal N}_{\C(y) ^*\rtimes \PGL_2}(\T_f)$ in Definition~\ref{definition: N_{C* times PGL_2}(T_f)^{circ} and related stuffs}; we then compute its neutral connected component ${\mathcal N}_{\C(y) ^*\rtimes \PGL_2}(\T_f)^{\circ}$ in Lemma~\ref{lemma: description of N_{C* times PGL_2}(T_f)^{circ}} showing in particular that it is either trivial or isomorphic to $\C^*$;  we then prove the equality $\Nor_{\Jonq}( \T_f ) =$  $\T_f  \rtimes {\mathcal N}_{\C(y) ^*\rtimes \PGL_2}(\T_f)$ in Lemma~\ref{lemma: normaliser of Tf in Jonq}, from which it will straightforwardly follow that ${ \Nor_{\Jonq}( \T_f )^{\circ} = \T_f  \rtimes {\mathcal N}_{\C(y) ^*\rtimes \PGL_2}(\T_f)^{\circ} }$ in Proposition~\ref{proposition: connected component of the normaliser of Tf in Jonq}.

Note that point \eqref{Second short exact sequence} of Lemma~\ref{lemma: description of N_{C* times PGL_2}(T_f)^{circ}} is to be used only later on (in the proof of Proposition~\ref{proposition: description of N_{C* times Aff_1}(T_f)^{circ}}).

\begin{definition} \label{definition: N_{C* times PGL_2}(T_f)^{circ} and related stuffs}
Let $f$ be a nonsquare element of $\C (y)$. 
\begin{enumerate}
\item
Let ${\mathcal N}_{\C(y) ^*\rtimes \PGL_2}(\T_f)$ be the subgroup of elements $\varphi \in \C (y) ^* \rtimes \PGL_2   \subseteq \Jonq$ which normalise $\T_f$, i.e.\ such that $\varphi  \T_f  \varphi^{-1} = \T_f$.
\item
Let $\Stab ( \So (f) ) := \{ v \in \PGL_2, \; v( \So (f) ) = \So (f)  \}$ be the subgroup of elements $v \in \Aut (\p^1) = \PGL_2$ which globally preserve $\So (f) \subseteq \p^1$.
\item
Let $\Fix ( \So (f) ) \subseteq \PGL_2$ be the subgroup of elements $v \in \Aut (\p^1) = \PGL_2$ which preserve pointwise $\So (f) \subseteq \p^1$.
\end{enumerate} 
\end{definition}

\begin{lemma} \label{lemma: description of N_{C* times PGL_2}(T_f)^{circ}}
Let $f$ be a nonsquare element of $\C (y)$ and let $\g$ denote the genus of the curve $x^2=f(y)$. Then, the following assertions hold:
\begin{enumerate}
\item \label{First short exact sequence}
We have the short exact sequence
\begin{equation} 1 \to \{ (\pm x,y)  \} \to {\mathcal N}_{\C (y) ^*\rtimes \PGL_2}(\T_f) \xrightarrow{\pr_2} \Stab ( \So (f) ) \to 1. \label{equation: short exact sequence for N_{C* times PGL_2}(T_f)} \end{equation}
\item \label{The two cases for N_{C* times PGL_2}(T_f)^{circ}}
The group ${\mathcal N}_{\C(y) ^*\rtimes \PGL_2}(\T_f)^{\circ}$ is either trivial or isomorphic to $\C^*$: 
\begin{enumerate}
\item \label{g is at least 1}
${\mathcal N}_{\C(y) ^*\rtimes \PGL_2}(\T_f)^{\circ} = \{ 1 \}$ if $\g \ge 1$, i.e.\ $| \So (f) | \ge 4$;
\item \label{g=0}
${\mathcal N}_{\C(y) ^*\rtimes \PGL_2}(\T_f)^{\circ} \simeq \C^*$ if $\g = 0$, i.e.\ $| \So (f) | = 2$. For $f=y$, we have ${\mathcal N}_{\C(y) ^*\rtimes \PGL_2}(\T_y)^{\circ} = \T_{1,2}$.
\end{enumerate}
\item  \label{Second short exact sequence}
In case (b) above the second projection  $\pr_2 \colon \Jonq \to \PGL_2$ induces the short exact sequence
\begin{equation} 1 \to \{ (\pm x,y)  \} \to {\mathcal N}_{\C (y) ^*\rtimes \PGL_2}(\T_f)^{\circ} \xrightarrow{\pr_2} \Fix ( \So (f) ) \to 1, \label{equation: short exact sequence for N_{C(y)* rtimes PGL_2}(T_f)^{circ}} \end{equation}
where $ \Fix ( \So (f) )  \simeq \C^*$.
\end{enumerate}
\end{lemma}

\begin{proof}
\eqref{First short exact sequence}
By Lemma~\ref{lemma: the action of C (y)^* rtimes PGL_2 on C(y)} we have
\[ {\mathcal N}_{\C(y) ^*\rtimes \PGL_2}(\T_f)= \{ ( \lambda x, v(y) ), \; \lambda \in \C (y) ^*, \, v \in \PGL_2, \; f \circ v (y) = \lambda ^2 f(y)  \} .\]
Hence Lemma~\ref{lemma: fov/g is a square iff v stabilizes the odd support of f} shows that $\pr_2$ induces a surjection ${\mathcal N}_{\C(y) ^*\rtimes \PGL_2}(\T_f)  \to \Stab ( \So (f) )$ and it is therefore clear that the short exact sequence \eqref{equation: short exact sequence for N_{C* times PGL_2}(T_f)} holds.

\eqref{g is at least 1} If $\g \ge 1$, i.e.\ $ | \So (f) | \ge 4$, then $\Stab ( \So (f) ) $ is finite, and the short exact sequence \eqref{equation: short exact sequence for N_{C* times PGL_2}(T_f)} shows that ${\mathcal N}_{\C (y) ^*\rtimes \PGL_2}(\T_f)$ is finite. This proves that we have ${\mathcal N}_{\C (y) ^*\rtimes \PGL_2}(\T_f)^{\circ} =$  $\{ 1 \} $.

(b) If $\g =0 $, i.e.\ $ | \So (f) |  = 2$, then $y$ is in the orbit of $f$ under the action of $\C(y)^* \rtimes \PGL_2$ on $\C(y)$ (see Lemma~\ref{lemma: the action of C (y)^* rtimes PGL_2 on C(y)}) so that we may assume $f =y$. We then get
\[ {\mathcal N}_{\C (y) ^*\rtimes \PGL_2}(\T_y) = \{ (\lambda x, \lambda^2 y), \; \lambda \in \C^* \} \cup \{ (\lambda y^{-1} x, \lambda^2 y^{-1}), \; \lambda \in \C^* \},\]
\[ {\mathcal N}_{\C (y) ^*\rtimes \PGL_2}(\T_y)^{\circ} = \{ (\lambda x, \lambda^2 y), \; \lambda \in \C^* \} = \T_{1,2},\]
\[ \Fix ( \So (y) ) = \Fix (  \{ 0, \infty \} ) =  \{  ( y \mapsto \mu y), \; \mu \in \C^*\}= \T_{0,1}.\]
This shows (b).

\eqref{Second short exact sequence}
Still in case (b) the above computation shows that the short exact sequence~\eqref{equation: short exact sequence for N_{C(y)* rtimes PGL_2}(T_f)^{circ}} holds, and the isomorphism $ \Fix ( \So (f) )  \simeq \C^*$ is clear.
\end{proof}

We have $\T_f \cap {\mathcal N}_{\C(y)^*\rtimes \PGL_2}(\T_f)=\{ \id \}$ and the group ${\mathcal N}_{\C(y) ^*\rtimes \PGL_2}(\T_f)$ normalises $\T_f$ (by definition of  ${\mathcal N}_{\C(y)^* \rtimes \PGL_2}(\T_f)$), hence we have a semidirect product structure \linebreak $\T_f  \rtimes {\mathcal N}_{\C (y) ^* \rtimes \PGL_2}(\T_f)$. Note that this semidirect product structure is not induced by the semidirect product $\Jonq = \PGL_2 ( \C (y) ) \rtimes \PGL_2$ which is usually considered in this paper.

\begin{lemma} \label{lemma: normaliser of Tf in Jonq}
Let $f \in \C ( y) $ be a nonsquare element. Then the normaliser of $\T_f$ in $\Jonq$ is equal to
\[ \Nor_{\Jonq}( \T_f ) = \T_f  \rtimes {\mathcal N}_{\C(y) ^*\rtimes \PGL_2}(\T_f) .\]
\end{lemma}

\begin{proof}
The inclusion $\T_f  \rtimes {\mathcal N}_{\C(y) ^*\rtimes \PGL_2}(\T_f) \subseteq \Nor_{\PGL_2 ( \C(y) ) \rtimes \PGL_2}( \T_f )$ being clear, it is enough to prove $\Nor_{\PGL_2 ( \C(y) ) \rtimes \PGL_2}( \T_f )  \subseteq \T_f  \rtimes {\mathcal N}_{\C(y) ^*\rtimes \PGL_2}(\T_f)$. Take $g \in \Jonq$ such that $g  \T_f g^{-1} = \T_f$. Let's begin by proving that $\pr_2 (g)$ belongs to $\Stab ( \So (f) )$ where $\pr_2 \colon \Jonq \to \PGL_2$ denotes the  second projection. Writing $g=uv$ with $u \in \PGL_2 (\C (y) )$, $v \in \PGL_2$, we have $\pr_2 (g) =v$ and we want to prove that  $v ( \So (f) ) = \So (f)$. Since $ g \T_f g^{-1} = u (v\T_fv^{-1}) u^{-1} = u (  \T_{f \circ v^{-1}} ) u^{-1} =\T_f  $, the tori $ \T_{f \circ v^{-1}}$ and $\T_f$ are conjugate in $\PGL_2 ( \C(y) )$. This proves that $f \circ v^{-1} = \lambda^2 f$ for some $\lambda \in \C (y)$ (by Proposition~\ref{proposition: when are Tf and Tg conjugate in PGL(2,K)}) and now Lemma~\ref{lemma: fov/g is a square iff v stabilizes the odd support of f} shows that $v ( \So (f) ) = \So (f)$.

Since $\pr_2 (g) \in \Stab ( \So (f) )$, the short exact sequence \eqref{equation: short exact sequence for N_{C* times PGL_2}(T_f)} of Lemma~\ref{lemma: description of N_{C* times PGL_2}(T_f)^{circ}} shows that there exists  $ g' \in {\mathcal N}_{\C(y)^*\rtimes \PGL_2}(\T_f)$ such that $\pr_2 (g') = \pr_2 (g)$. By definition of ${\mathcal N}_{\C(y)^*\rtimes \PGL_2}(\T_f)$, $g'$ normalises $\T_f$. Hence $g'' := g (g')^{-1}$ also normalises $\T_f$ and moreover it belongs to $\PGL_2 ( \C(y) )$. This shows that $g''$ belongs to $\Nor_{\PGL_2(\C(y) )}(\T_f) = \T_f \rtimes \langle \smat{1 }{0}{0}{-1} \rangle  \subseteq \T_f  \rtimes {\mathcal N}_{\C(y)^*\rtimes \PGL_2}(\T_f) $ (see Lemma~\ref{lemma: the normaliser of T_f in PGL_2(K)}) and this concludes the proof.
\end{proof}

The following result directly follows from Lemmas~\ref{lemma: description of N_{C* times PGL_2}(T_f)^{circ}} and \ref{lemma: normaliser of Tf in Jonq}.

\begin{proposition} \label{proposition: connected component of the normaliser of Tf in Jonq}
Let $f$ be a nonsquare element of $\C(y)$ and let $\g$ be the genus of the curve $x^2 =f(y)$. Then we have
\[ \Nor_{\Jonq}( \T_f )^{\circ} = \T_f  \rtimes {\mathcal N}_{\C(y) ^*\rtimes \PGL_2}(\T_f)^{\circ}. \]
Moreover the group ${\mathcal N}_{\C(y) ^*\rtimes \PGL_2}(\T_f)^{\circ}$ is trivial if $\g \ge 1$ and isomorphic to $\C^*$ if $\g=0$.

Spelling out the details in the two cases, we have:
\begin{enumerate}
\item \label{Nor_{Jonq}(T_f )^{circ} = T_f if g  is at least 1}
$\Nor_{\Jonq}( \T_f )^{\circ} = \T_f$ if $\g \ge 1$;
\item \label{Nor_{Jonq}(T_y )^{circ} = T_t rtimes T_{1,2}}
$ \Nor_{\Jonq}( \T_y )^{\circ} = \T_y \rtimes \T_{1,2}$.
\end{enumerate}
\end{proposition}

\begin{remark}
Let $f$ be a nonsquare element of $\C(y)$. Note that the equation $x^2 = f(y)$ defines an affine curve in $\A^2$. The group $\T_f$ fixes  this curve pointwise. In the case $f= y$, the group $\T_y \rtimes \T_{1,2}$ stabilises this curve (but not pointwise).
\end{remark}

\section{Subgroups of $\PGL_2(\C(y)) \rtimes \Aff_1$ conjugate to $\T_{0,1}$} \label{section: Subgroups of PGL_2( C(y)) rtimes Aff_1 conjugate to T_{0,1}}

The following technical lemma describes rather explicitly a morphism from a factorial irreducible affine variety $W$ to $\bp$ whose image is contained in $\PGL_2 ( \C (y) )$. We will use it in the proof of Lemma~\ref{lemma: characterisation of the groups conjugate to T_{0,1}}.

\begin{lemma}  \label{lemma: morphism to PGL(2,C(y))}
Consider the monoid
\[ \Mat_2 ( \C [y] ) _{\det \ne 0} := \{ M \in \Mat_2 ( \C [y] ) , \; \det M \ne 0 \} \]
and let
\[ p \colon \Mat_2 ( \C [y] )_{\det \ne 0}  \to \PGL_2 ( \C (y) ), \quad \smat{a}{b}{c}{d} \mapsto \smat{a}{b}{c}{d},\]
be the natural surjective monoid  morphism. Let $\iota \colon \PGL_2 ( \C (y) ) \hookrightarrow \bp$ be the natural injection. If $W$ is a factorial irreducible affine variety and $\varphi \colon W \to \PGL_2 ( \C (y) )$  a map such that $\iota \circ \varphi \colon W \to \bp$ is a morphism in the sense of Definition~\ref{definition: morphism to Bir(W)}, then $ \varphi$ admits a lifting $\widehat{\varphi} \colon $  $ W \to \Mat_2 ( \C [y] ) _{\det \ne 0}$ (i.e.\ a map such that $ \varphi = p \circ \widehat{\varphi}$) which is a morphism of ind-varieties.\footnote{For the definitions of ind-varieties and morphisms of ind-varieties, see e.g. \cite[\S 1.1]{FurterKraft2018}.} This exactly means that  $\widehat{\varphi}$  (or $\varphi$!) is of the form
\[ w \mapsto  \smat{a(w,y) }{b(w,y)}{c(w,y)}{d(w,y)}, \]
for some $a,b,c,d \in \C [W][y]$ and that for any $w \in W$, the determinant $\det \smat{a(w,y) }{b(w,y)}{c(w,y)}{d(w,y)}$ is a nonzero element of $\C [y]$.
\end{lemma}

\begin{proof} 
By Definition~\ref{definition: morphism to Bir(W)}, $\iota \circ \varphi$ corresponds to a birational map $\psi \colon W \times \A^2 \dasharrow W \times \A^2$, $(w,x,y) \dasharrow (w, f(w,x,y), y)$, i.e.\ a $(W \times \A^1_y)$-birational map of the $(W \times \A^1_y)$-variety $(W \times \A^1_y) \times \A^1_x$. Such a birational map corresponds to a $\C (W)(y)$-automorphism of the field $\C (W)(y)(x)$ and hence the rational function $f \in \C(W)(y)(x)$ is necessarily of the form $f \colon(w,x,y) \dasharrow \frac{ax+b}{cx+d}$, for some $a,b,c,d \in \C(W)(y)$ which satisfy $ad-bc \neq 0$. Since $\C(W)(y)$ is the field of fractions of $\C[W][y]$ we may assume that $a,b,c,d$ belong to $\C[W][y]$ and since $\C[W][y]$ is factorial, we may even assume that $\gcd (a,b,c,d) = 1$ in $\C[W][y]$. This determines $a,b,c,d$ uniquely up to a common factor in $\C[W]^*$ (the group of invertible elements of $\C [W]$). We could easily check that we have $\gcd(ax+b,cx+d)=1$ in $\C[W][x,y]$. If $h$ belongs to $\C[W][x,y]$, its zero set is defined by
\[ Z(h) = \{ (w,x,y) \in W\times\A^2, \; h(w,x,y) = 0\} \subseteq  W\times\A^2.\]
Set
\[
\begin{array}{ll}
S_1(\psi)=Z(ad-bc) \cup Z(cx+d) \subset W\times \A^2  &  \text{and}  \\
S_2(\psi)=Z(ad-bc) \cup Z(-cx+a) \subset W\times\A^2.
\end{array}
\]
By the previous remark $S_1(\psi)$ and $S_2 ( \psi)$ only depend on $\psi$ and not on the choice of the functions $a,b,c,d \in \C[W][y]$. It is clear that $\psi$ induces an isomorphism $U \stackrel\sim\to V$ where $U:=(W\times\A^2)\setminus S_1(\psi)$ and $V:=(W\times\A^2)\setminus S_2(\psi)$, the inverse map $\psi^{-1}$ being given by $(w,x,y) \dasharrow (w,g(w,x,y),y)$ where $g: =\frac{\hspace{2mm}dx-b}{ -cx+a}$.

We now show that $\Exc(\psi)=S_1(\psi)$, where the exceptional set $\Exc(\psi)\subset W\times\A^2$, by definition, is the complement of the open set consisting of all points at which $\psi$ induces a local isomorphism. In particular $\Exc(\psi)$ is closed in $W\times\A^2$. We have
\[Z(cx+d)\cap D(ax+b)\subset \Exc(\psi),\]
where $D(h)\subset W\times\A^2$, for a function $h\in\C[W][x,y]$, denotes the principal open subset $(W\times\A^2)\setminus Z(h)$. Since $\gcd(ax+b,cx+d)=1$, the hypersurfaces $Z(ax+b)$ and $Z(cx+d)$ have no common irreducible components and thus the intersection of $D(ax+b)$ with any irreducible component of $Z(cx+d)$ is nonempty. Hence $D(ax+b)\cap Z(cx+d)$ is dense in $Z(cx+d)$ and since $\Exc(\psi)$ is closed, all of $Z(cx+d)$ is contained in $\Exc(\psi)$. In particular, $Z(c)\cap Z(d) \subset Z(cx+d)$ is contained in $\Exc(\psi)$. It only remains to show that $Z(ad-bc)\setminus (Z(c)\cap Z(d))$ is contained in $\Exc(\psi)$. Let $(w_0,x_0,y_0)\in Z(ad-bc)\setminus (Z(c)\cap Z(d))$, so that $c(w_0,y_0)x+d(w_0,y_0)\in\C[x]$ is nonzero at $x_0$. Then $\psi$ is defined in $(w_0,x_0,y_0)$ but, since $(w_0,x_0,y_0)\in Z(ad-bc)$, it is constant along the line $\A^1\simeq \{w_0\}\times\A^1_x\times\{y_0\}$. In particular $\psi$ does not induce a local isomorphism at $(w_0,x_0,y_0)$. We have proven that $\Exc(\psi)=S_1(\psi)$. Similarly $\Exc(\psi^{-1})=S_2(\psi)$.\par
The map $\iota \circ \phi$ being a morphism in the sense of Definition~\ref{definition: morphism to Bir(W)}, the projection of $U$ to $W$ is surjective. In other words, for each $w_0\in W$, there exists $(x_0,y_0)\in\A^2$ such that $(w_0,x_0,y_0)\notin Z(ad-bc)\cup Z(cx+d)$. In particular we have $\smat{a(w_0,y)}{b(w_0,y)}{c(w_0,y)}{d(w_0,y)}\in \PGL_2 (\C(y))$. Hence the correspondence  $w \mapsto  \smat{a(w,y) }{b(w,y)}{c(w,y)}{d(w,y)}$ actually defines a morphism of ind-varieties $\widehat{\varphi}$  $\colon W \to \Mat_2 ( \C [y] ) _{\det \ne 0}$ which satisfies $ \varphi = p \circ \widehat{\varphi}$.
\end{proof}

\begin{lemma} \label{lemma: characterisation of the groups conjugate to T_{0,1}}
Let $G$ be an algebraic subgroup of $\PGL_2 ( \C (y) ) \rtimes \Aff_1$ isomorphic to the multiplicative algebraic group $(\C^*, \times)$. Then, the following assertions are equivalent:
\begin{enumerate}
\item \label{The map pr_2 : G -> pr_2(G) is an isomorphism}
The map $\pr_2 \colon \PGL_2 ( \C (y) ) \rtimes \Aff_1 \to \Aff_1$ induces an isomorphism $G \to \pr_2 ( G)$;
\item \label{G and T_{0,1} are conjugate}
The groups $G$ and $\T_{0,1}$ are conjugate.
\end{enumerate}
\end{lemma}

\begin{proof}
\eqref{G and T_{0,1} are conjugate} $\Longrightarrow$ \eqref{The map pr_2 : G -> pr_2(G) is an isomorphism} is clear. Let's prove  \eqref{The map pr_2 : G -> pr_2(G) is an isomorphism} $\Longrightarrow$ \eqref{G and T_{0,1} are conjugate}. If $\pr_2$ induces an isomorphism $G \to \pr_2 ( G)$, then, up to conjugation by an element of $\Aff_1 \subseteq \PGL_2 ( \C (y) ) \rtimes \Aff_1$ we have $\pr_2 (G ) = \{ \lambda y, \; \lambda \in \C^* \}$. Hence, there is a morphism $s \colon \C^* \to \PGL_2(\C(y))$, $\lambda \mapsto s_{\lambda}$, such that $B = \{ (s_\lambda(y),\lambda y), \; \lambda \in \C^* \}$. Actually, by Lemma~\ref{lemma: morphism to PGL(2,C(y))}, there exist polynomials $a,b,c,d\in\C[\lambda^{\pm1}][y]$ such that $s_\lambda(y) = \smat{a(\lambda,y)}{b(\lambda,y)}{c(\lambda,y)}{d(\lambda,y)}$. Since we have $(s_{\lambda \mu }(y)x ,\lambda \mu y) = (s_\lambda(y)x ,\lambda y) \circ (s_\mu(y)x ,\mu y) $ this yields $s_{\lambda\mu}(y)=s_\lambda(\mu y)s_\mu(y)$. For all $\lambda\in\C^*$ we have $a(\lambda,y)d(\lambda,y)-b(\lambda,y)c(\lambda,y)\neq 0$ in $\C [y ]$. We may therefore choose $y_0 \in \C\setminus\{0\}$ such that $a(\lambda,y_0) d(\lambda,y_0) - b(\lambda,y_0) c(\lambda,y_0) \neq 0$ in  $\C[\lambda^{\pm1}]$. It follows that $s_y (y_0)\in\PGL_2 ( \C( y ))$ and because $y_0 \neq 0$ we can replace $y$ by $y y_0^{-1}$ to obtain an element $t(y) := s_{y y_0^{-1}}(y_0)$ which belongs to $\PGL_2( \C(y) )$ as well. Let $\phi:= (t (y) x,y) \in \PGL_2 ( \C (y) ) \subseteq $  $\PGL_2 ( \C (y) )  \rtimes \Aff_1$. Then we have $s_{\lambda y y_0 ^{-1} } (y_0) = s_{\lambda} (y) s_{yy_0^{-1}} (y_0)$, i.e.\ $t(\lambda y) =  s_{\lambda} (y)  t(y)$, i.e.\ $t(\lambda y) t(y)^{-1} =  s_{\lambda} (y)  $, i.e.\ 
\[ (s_\lambda(y),\lambda y)=\phi\circ(x,\lambda y)\circ \phi^{-1},\]
showing that up to conjugation we have $B=\T_{0,1}$. 
\end{proof}

\section{Embeddings of $(\C,+)$ into $\Jonq$}\label{section: Embeddings of C+ into JJ}

The aim of this section is to prove the proposition below which describes the different embeddings of the additive group  $(\C,+)$ into $\Jonq$ up to conjugation.

\begin{proposition} \label{proposition: embeddings of (C,+) in Jonq}
Any algebraic subgroup of $\Jonq$ isomorphic to the additive algebraic group $(\C,+)$ is either conjugate to $\U_1 := \{ (x+c, y), \; c \in \C \}$ or to $\U_2 := \{ (x,y+c), \; c \in \C \}$.
\end{proposition}

Our proof relies on the following characterisation of the embeddings of $(\C,+)$ into $\Jonq$ which are conjugate to $\U_2$.

\begin{lemma} \label{lemma: characterisation of the groups conjugate to U2}
Let $G$ be an algebraic subgroup of $\Jonq$ isomorphic to the additive algebraic group $(\C,+)$. Then, the following assertions are equivalent:
\begin{enumerate}
\item \label{pr_2(G) is nontrivial}
The group $\pr_2 (G)$ is nontrivial (where $\pr_2 \colon \Jonq \to \PGL_2$ is the second projection).
\item \label{G and U2 are conjugate}
The groups $G$ and $\U_2$ are conjugate
by an element of $\PGL_2 ( \C(y) ) \subseteq \Jonq$.
\end{enumerate}
\end{lemma}

\begin{proof}
\eqref{G and U2 are conjugate} $\Longrightarrow$  \eqref{pr_2(G) is nontrivial} is clear. Let's prove \eqref{pr_2(G) is nontrivial} $\Longrightarrow$ \eqref{G and U2 are conjugate}. If  $\pr_2 (G)$ is nontrivial, $\pr_2$ induces an isomorphism from $G$ onto its image $\U_2 \simeq ( \C, +)$.
Up to conjugation the inverse isomorphism $(\C, +) \to G$ is of the form $\theta \colon c \mapsto (A_c, y+c) \in \Jonq$, where the map $c \mapsto A_c$  is defined as a morphism $\C \to \bp$ (see Definition~\ref{definition: morphism to Bir(W)}) with values in $\PGL_2 (\C (y) )$. By Lemma~\ref{lemma: morphism to PGL(2,C(y))}, there exist polynomials $a(t,y), b(t,y),c(t,y),d(t,y)  \in \C [t , y]$ such that
\[ \forall \, t \in \A^1, \; A_t(y) = \smat{a(t,y)}{b(t,y)}{c(t,y)}{d(t,y) } \in \PGL_2 (\C (y) ) .\]
In particular, the determinant
\[ \Delta (t,y) = \det \smat{a(t,y)}{b(t,y)}{c(t,y)}{d(t,y) } \in \C [t,y]\]
is such that $\Delta (t_0, y)$ is a nonzero element of $\C [y]$ for each $t_0 \in \A^1$.

The equality $\theta (t+u) = \theta (t) \circ \theta (u)$ yields
\begin{equation} A_{t+u} (y) = A_t (y+u) A_u (y). \label{the basic equation} \end{equation}
Setting $t=-u$, we get:
\[ \I_2= A_0(y)  = A_{-u}(y+u)A_u(y), \quad \text{i.e.} \]
\[  \smat{a(0,y)}{b(0,y)}{c(0,y)}{d(0,y) } = \smat{a(-u,y+u)}{b(-u,y+u)}{c(-u,y+u)}{d(-u,y+u) } \smat{a(u,y)}{b(u,y)}{c(u,y)}{d(u,y) } . \]
Choosing $y_0 \in \C$ such that $\Delta (0, y_0) \neq 0$ and replacing $y$ with $y_0 $ in the latter equality, proves that  $A_u(y_0)$ defines an element of $\PGL_2 ( \C (u) )$
(Example~\ref{example:A(t,y)} below shows that the polynomial $\Delta (0,y_0)$ may actually vanish for some values of $y_0$). Setting $y=y_0$ in \eqref{the basic equation}, we get:
\[ A_{t+u} (y_0) = A_t (u+y_0) A_u (y_0),\]
i.e.\ $ A_t (u+y_0) = A_{t+u} (y_0) A_u (y_0) ^{-1}  $. Replacing $u$ with $u-y_0$ gives
\[ A_t (u) = A_{u+t-y_0} (y_0) A_{u-y_0} (y_0) ^{-1}.\]
Writing $B(u) := A_{u-y_0} (y_0) \in \PGL_2 ( \C (u) )$, we get $A_t(u) = {B(u+t)} B(u)^{-1} $, i.e.\ (replacing $u$ with $y$) $A_t(y) = B(y+t) B(y)^{-1} $, or equivalently
\begin{equation} A_c (y) = B(y+c) B(y) ^{-1}. \label{cocycle condition on A} \end{equation}
Setting $\varphi := (B(y) , y) \in \PGL_2 ( \C (y) ) \subseteq \Jonq$, the equation \eqref{cocycle condition on A} shows that
\[ \theta (c) = \varphi \circ (x, y+c) \circ \varphi^{-1}. \qedhere \]
\end{proof}

\begin{example} \label{example:A(t,y)}
Set $\tilde A_t(y) := \smat{(y+t+1) (y-1)}{-t}{0}{(y+t-1)(y+1) } \in \Mat_2 ( \C [t,y] )$ and $A_t(y) :=$  $[ \tilde A_t(y) ] \in \PGL_2 (\C(y) )$. We have
\[ (y+u+1) (y+u-1) \tilde A_{t+u}(y) = \tilde A_t (y+u) \tilde A_u (y), \]
showing that the equation \eqref{the basic equation}  above is satisfied. Moreover, we have
\[ \Delta (t,y) := \det \tilde A_t(y) = (y+t+1) (y+t-1)(y+1) (y-1) \]
showing that $\Delta (0,y_0) \neq 0$ if and only if $y_0 \neq \pm 1$. Writing $y_0=0$ and setting as in the proof of Lemma~\ref{lemma: characterisation of the groups conjugate to U2}
\[ B(u) := A_{u-y_0} (y_0) =A_u(0) =  \smat{-1-u}{-u}{0}{-1+u }=  \smat{1+u}{u}{0}{1-u }   \in \PGL_2 ( \C (u) ), \]
we get $B(y) = \smat{1+y}{y}{0}{1-y}$ and $A_c (y) = { B(y+c)  }B(y)^{-1}$.
\end{example}

\begin{proof}[Proof of Proposition~\ref{proposition: embeddings of (C,+) in Jonq}]
Let $G$ be an algebraic subgroup of $\Jonq$ isomorphic to $(\C,+)$. If $\pr_2 (G)$ is  nontrivial, then the conclusion follows from Lemma~\ref{lemma: characterisation of the groups conjugate to U2}. We may therefore assume that $G$ is contained in $\PGL_2 ( \C(y) )$. By Lemma~\ref{lemma: algebraic elements of PGL(2,C(y))}, we may assume (up to conjugation) that $G$ contains the matrix $\smat{1}{1}{0}{1} \in \PGL_2 ( \C(y) )$ from which it follows that $G= \U_1$.
\end{proof}

\section{Borel subgroups of $\PGL_2 (\C (y) )$ and of $\PGL_2 (\C (y) ) \rtimes \Aff_1$}\label{section: Borel subgroups of PGL2C(y) and of PGL2C(y) rtimesAff1}

In this section, we will prove that any Borel subgroup of $\PGL_2 (\C (y) )$ is contained in a unique Borel subgroup of $\PGL_2 (\C (y) ) \rtimes \Aff_1$ and conversely that any Borel subgroup of $\PGL_2 (\C (y) ) \rtimes \Aff_1$ contains a unique Borel subgroup of $\PGL_2 (\C (y) )$. Hence, there will be a natural bijection from the set of Borel subgroups of $\PGL_2 (\C (y) )$ to the set of Borel subgroups of $\PGL_2 (\C (y) ) \rtimes \Aff_1$ (see Theorem~\ref{theorem: the natural bijection between the Borel subgroups of PGL_2 ( C(y) ) and the Borel subgroups of PGL_2 ( C(y) ) rtimes Aff_1}).

The main difficulty is to prove that any Borel subgroup of $\PGL_2 (\C (y) ) \rtimes \Aff_1$ contains at least one Borel subgroup of $\PGL_2 (\C (y) )$ and we will begin by showing this result which is Theorem~ \ref{theorem: any Borel subgroup of PGL_2 (C (y) ) rtimes Aff_1 contains a Borel subgroup of PGL_2 (C (y) )} below. The proof relies on a few preliminary results.

We will use the next technical result for the proof of both Theorem~\ref{theorem: any Borel subgroup of PGL_2 (C (y) ) rtimes Aff_1 contains a Borel subgroup of PGL_2 (C (y) )} and Proposition~\ref{proposition: precise description of the unique Borel subgroup of PGL_2(C(y)) rtimes PGL_2 containing T_f}.

\begin{proposition}  \label{proposition: description of N_{C* times Aff_1}(T_f)^{circ}}
Let $f$ be a nonsquare element of $\C(y)$. Then we have
\begin{equation}
\Nor_{\PGL_2 ( \C(y) ) \rtimes \Aff_1}(\T_f)^\circ = \T_f \rtimes  {\mathcal N}_{\C(y) ^*\rtimes \Aff_1}(\T_f)^{\circ} \label{equation: neutral component of the normaliser of T_f in PGL_2 ( C(y) ) rtimes Aff_1} \end{equation}
where ${\mathcal N}_{\C(y) ^*\rtimes \Aff_1}(\T_f)$ is the subgroup of elements $\varphi \in \C (y) ^* \rtimes \Aff_1   \subseteq \Jonq$ which normalise $\T_f$, i.e.\ such that $\varphi  \T_f  \varphi^{-1} = \T_f$. Moreover, the group $ {\mathcal N}_{\C(y) ^*\rtimes \Aff_1}(\T_f)^{\circ}$ is either trivial or isomorphic to $\C^*$. More precisely, if $\g$ denotes the genus of the curve $x^2=f(y)$, we have
\begin{enumerate}
\item \label{genus(f) is at least 1 or infinity does not belong to the odd support of f}
${\mathcal N}_{\C(y) ^*\rtimes \Aff_1}(\T_f)^{\circ}  = \{ \id \}$ if $\g \ge 1$ or $\infty \notin \So (f)$;
\item \label{genus(f)=0 and infinity belongs to the odd support of f}
${\mathcal N}_{\C(y) ^*\rtimes \Aff_1}(\T_f)^{\circ}  = {\mathcal N}_{\C(y) ^*\rtimes \PGL_2}(\T_f)^{\circ}  \simeq \C^*$ if $\g = 0$ and $\infty \in \So (f)$.
\end{enumerate}
\end{proposition}

\begin{proof}
We have
\[ \Nor_{\PGL_2 ( \C (y) )  \rtimes \Aff_1}( \T_f )^{\circ} = (  \Nor_{\Jonq}( \T_f )^{\circ} \cap [  \PGL_2 ( \C (y) )  \rtimes \Aff_1 ] ) ^{\circ} \]
and ${\Nor_{\Jonq}( \T_f )^{\circ} = \T_f  \rtimes {\mathcal N}_{\C(y) ^*\rtimes \PGL_2}(\T_f)^{\circ}}$ by Proposition~\ref{proposition: connected component of the normaliser of Tf in Jonq}. Hence we get
\[ \Nor_{\PGL_2 ( \C (y) ) \rtimes \Aff_1}( \T_f )^{\circ} = \T_f  \rtimes ( {\mathcal N}_{\C(y) ^*\rtimes \PGL_2}(\T_f)^{\circ} \cap [  \C (y)  ^*\rtimes \Aff_1] )^{\circ}  \]
and the equality ${\mathcal N}_{\C(y) ^*\rtimes \Aff_1}(\T_f)^{\circ}=( {\mathcal N}_{\C(y) ^*\rtimes \PGL_2}(\T_f)^{\circ} \cap [  \C (y)  ^*\rtimes \Aff_1] )^{\circ}$ establishes \eqref{equation: neutral component of the normaliser of T_f in PGL_2 ( C(y) ) rtimes Aff_1}.

Let's now show \eqref{genus(f) is at least 1 or infinity does not belong to the odd support of f}. If $\g \ge 1$ we have ${\mathcal N}_{\C(y) ^*\rtimes \PGL_2}(\T_f)^{\circ} = \{ 1 \}$ by Lemma~\ref{lemma: description of N_{C* times PGL_2}(T_f)^{circ}}\eqref{g is at least 1} and this yields ${\mathcal N}_{\C(y) ^*\rtimes \Aff_1}(\T_f)^{\circ}  = \{ \id \}$. So, assume now that $\g =0$ and $\infty \notin \So (f)$. Then, the following short exact sequence (which is equation \eqref{equation: short exact sequence for N_{C(y)* rtimes PGL_2}(T_f)^{circ}} of Lemma~\ref{lemma: description of N_{C* times PGL_2}(T_f)^{circ}})
\[ 1 \to \{ (\pm x,y)  \} \to {\mathcal N}_{\C (y) ^*\rtimes \PGL_2}(\T_f)^{\circ} \xrightarrow{\pr_2} \Fix ( \So (f) ) \to 1 \]
yields 
\begin{multline} 1 \to \{ (\pm x,y)  \} \to {\mathcal N}_{\C (y) ^*\rtimes \PGL_2}(\T_f)^{\circ} \cap [  \C (y)  ^*\rtimes \Aff_1]   \\\xrightarrow{\pr_2}  \Fix \big(  \So (f)  \cup \{ \infty \}  \big) \to 1. \label{equation: short exact sequence related with N_{C(y)* rtimes Aff_1}(T_f)^{circ}}  \end{multline}
The hypothesis $\infty \notin \So (f)$ shows that $\So (f)  \cup \{ \infty \}$ admits $3$ elements and this proves that $ \Fix \big(  \So (f)  \cup \{ \infty \}  \big)$ is trivial. Hence \eqref{equation: short exact sequence related with N_{C(y)* rtimes Aff_1}(T_f)^{circ}} shows that
\[    {\mathcal N}_{\C (y) ^*\rtimes \PGL_2}(\T_f)^{\circ}  \cap [  \C (y)  ^*\rtimes \Aff_1]  = \{ (\pm x,y)  \} \]
and finally ${\mathcal N}_{\C(y) ^*\rtimes \Aff_1}(\T_f)^{\circ}  = \{ \id \}$.

Let's now show~\eqref{genus(f)=0 and infinity belongs to the odd support of f}. So assume that we have $\g = 0$ and $\infty \in \So (f)$. We have therefore $\Fix ( \So (f) )  \subseteq \Fix ( \infty ) = \Aff_1$ and the first exact sequence above (which was equation \eqref{equation: short exact sequence for N_{C(y)* rtimes PGL_2}(T_f)^{circ}} of Lemma~\ref{lemma: description of N_{C* times PGL_2}(T_f)^{circ}}) shows that ${\mathcal N}_{\C (y) ^*\rtimes \PGL_2}(\T_f)^{\circ} \subseteq \C (y) ^* \rtimes \Aff_1$.  Hence we have ${\mathcal N}_{\C(y) ^*\rtimes \Aff_1}(\T_f)^{\circ}  = {\mathcal N}_{\C(y) ^*\rtimes \PGL_2}(\T_f)^{\circ}$ and the isomorphism ${\mathcal N}_{\C(y) ^*\rtimes \PGL_2}(\T_f)^{\circ}  \simeq \C^*$ was already shown in Lemma~\ref{lemma: description of N_{C* times PGL_2}(T_f)^{circ}}\eqref{g=0}.
\end{proof}

The following few results, culminating in Proposition~\ref{proposition: no Borel subgroup of PGL_2( C(y) ) rtimes Aff_1 embeds into Aff_1 via pr_2}, are also in preparation for Theorem~\ref{theorem: any Borel subgroup of PGL_2 (C (y) ) rtimes Aff_1 contains a Borel subgroup of PGL_2 (C (y) )}.

\begin{lemma} \label{lemma: morphism from a closed connected subgroup of Bir(P2) to (C,+) or (C*,x)}
Let $G$ be a closed connected subgroup of $\bp$, let $H$ be a linear algebraic group isomorphic to $(\C,+)$ or $(\C^*, \times)$, and let $\varphi \colon G \to H$ be a morphism of groups such that for each variety $A$ and each morphism $\lambda \colon  A \to \bp$  (see Definition~\ref{definition: morphism to Bir(W)}) with values in $G$, the composition $\varphi \circ \lambda \colon A \to H$ is a morphism of varieties. Then, we have $\varphi (G ) = \{ 1 \}$ or $H$.
\end{lemma}

\begin{proof}
Assume that $\varphi (G) \neq \{ 1 \}$. Then, there exist two elements $g_1,g_2$ of $G$ such that $\varphi (g_1) \neq \varphi (g_2)$. By Lemma~\ref{lemma: a useful criterion of connectedness for closed subsets of Bir(Pn)}, there exists a connected (not necessarily irreducible) curve $C$ and a morphism $\lambda \colon C \to \bp$ whose image satisfies
\[ g_1,g_2 \in \Image (\lambda) \subseteq G .\]
The image $U$ of the morphism of algebraic varieties $\varphi \circ \lambda \colon C \to H$ is connected, constructible, and contains at least two points. Hence, it is a dense open subset of $H$. Moreover, $U$ is contained in the image of $\varphi$. Recall that if $V_1,V_2$ are two dense open subsets of a linear algebraic group $K$, then we have $K= V_1. V_2$ (see e.g.\ \cite[Lemma~(7.4), page 54]{Humphreys1975}). Here, we find that the image of $\varphi$ contains $U.U = H$ and hence is equal to $H$.
\end{proof}

\begin{remark}
If $\varphi \colon G \to H$ is a morphism of linear algebraic groups, it is well-known that $\varphi (G)$ is a closed subgroup of $H$. However, if $G$ is a closed subgroup of $\bp$, $H$ a linear algebraic group, and $\varphi \colon G \to H$ a morphism of groups such that for each variety $A$ and each morphism $\lambda \colon  A \to \bp$   with values in $G$, the composition $\varphi \circ \lambda \colon A \to H$ is a morphism of varieties, then it is in general false that $\varphi (G)$ is a closed subgroup of $H$. Take for $G$ the subgroup of monomial transformations
\[ G:= \{ (x^a y^b, x^c y^d) , \; \smat{a}{b}{c}{d} \in \GL_2 (\Z)   \} \subseteq \bp \]
and consider the natural morphism $\varphi \colon G \to \GL_2 (\C)$, $(x^a y^b, x^c y^d) \mapsto  \smat{a}{b}{c}{d}$. Then $G$ is a closed subgroup of $\bp$ whose image by $\varphi$ is equal to $\GL_2 (\Z)$, which is not closed in $\GL_2 (\C)$.
\end{remark}

\begin{lemma} \label{lemma: pr_2(B)}
Let $B$ be a closed connected subgroup of $\PGL_2 ( \C(y) ) \rtimes \Aff_1$. Then, $\pr_2 (B)$ is a closed connected subgroup of $\Aff_1\subset\PGL_2 ( \C(y) ) \rtimes \Aff_1$. In particular, up to conjugation, $\pr_2 (B)$ is equal to one of the following four subgroups of $\Aff_1$:
\[ \{ \id \},  \quad  \T_{0,1} = \{ (x,ay), \; a \in \C^* \}, \quad  \U_2 = \{ (x,y+c), \; c \in \C \}, \hspace{3mm} \text{or} \hspace{3mm}  \Aff_1.   \]
\end{lemma}

\begin{proof}
Set $H:= \pr_2(B)$ and consider the closed subgroup $\overline{H}$ of $\Aff_1$. Up to conjugation, $\overline{H}$ is equal to $\{ \id \}$, $\T_{0,1}$, $\U_2$, or $\Aff_1$. In the first three cases, the conclusion follows from Lemma~\ref{lemma: morphism from a closed connected subgroup of Bir(P2) to (C,+) or (C*,x)}. Hence, we may assume that $\overline{H} = \Aff_1$. This shows that $H$ is non-abelian, i.e.\ $D^1(H) \subseteq \U_2$ is nontrivial. Setting $\DDD (B) := \overline{ D(B) } \subseteq B$, this implies that $\pr_2 ( \DDD (B) ) \subseteq \U_2$ is nontrivial, and so, Lemma~\ref{lemma: morphism from a closed connected subgroup of Bir(P2) to (C,+) or (C*,x)} yields $\pr_2 ( \DDD (B) ) = \U_2$. In particular, we have $\U_2 \subseteq \pr_2 (B)$. Set $\eta \colon \Aff_1 \to \Aff_1 / \U_2 = \C^*$, $ay+b \mapsto a$. The group $\eta \circ \pr_2( B)  \subseteq \C^*$ being nontrivial, Lemma~\ref{lemma: morphism from a closed connected subgroup of Bir(P2) to (C,+) or (C*,x)} shows that $\eta \circ \pr_2( B)  = \C^*$. Since $H$ contains $\U_2$ and since $\eta (H) = \C^*$, we have $H= \Aff_1$.
\end{proof}

\begin{proposition} \label{proposition: no Borel subgroup of PGL_2( C(y) ) rtimes Aff_1 embeds into Aff_1 via pr_2}
There is no Borel subgroup $B$ of $\PGL_2(\C(y)) \rtimes \Aff_1$ such that $\restr{\pr_2}{B}\colon B\to \Aff_1$ is injective.
\end{proposition}

\begin{proof}
Suppose for contradiction that $B$ is a Borel subgroup such that \linebreak $\restr{\pr_2}{B}\colon B \to \Aff_1$ is injective. By Lemma~\ref{lemma: pr_2(B)} and up to conjugation, we may assume that $\pr_2(B)$ is equal to $\{ \id \}, \T_{0,1},\,\U_2$ or $\Aff_1$. The case $\{ \id \}$ is of course excluded. Let's begin by showing that $B$ is bounded. For each $d \ge 1$, set $B_d:= \{ f \in B, \; \deg f \le d \}$. By Lemma~\ref{lemma: Jonq is closed and pr_2 is continuous}\eqref{The image of a bounded closed subset of Jonq by pr_2 is a constructible subset of PGL_2}, $\pr_2 ( B_d)$ is a constructible subset of $\PGL_2$. Now, since the variety $\pr_2 (B)$ is the increasing union of the constructible subsets $\pr_2 ( B_d)$, $d \ge1$, there exists an integer $d$ such that $\pr_2 (B) = \pr_2 ( B_d)$, i.e.\ $B= B_d$ (see e.g. \cite[Lemma 1.3.1, page 15]{FurterKraft2018}). Hence, we have shown that $B$ is bounded, and it follows that $B$ is an algebraic subgroup of $\bp$ (see Remark~\ref{remark: an algebraic subgroup of Bir(Pn) is a bounded closed subgroup}).

If $\pr_2(B)=\U_2$, then by Lemma~\ref{lemma: characterisation of the groups conjugate to U2} and up to conjugation we may assume that  $B =\U_2$. But then $B$ would be strictly contained in the closed connected solvable subgroup $\BBB_2$ of $\PGL_2 ( \C(y) ) \rtimes \Aff_1$. A contradiction. If $\pr_2(B)=\Aff_1$, then since $\Aff_1$ normalises $\U_2$, $B$ normalises the preimage of $\U_2$ by the isomorphism $B \to \Aff_1$ induced by $\pr_2$. Up to conjugation and by Lemma~\ref{lemma: characterisation of the groups conjugate to U2} again we may (and will) assume that this latter group is $\U_2$. Hence $B$ is contained in $\Norm_{\PGL_2(\C(y)) \rtimes \Aff_1}(\U_2)=\PGL_2  \times \Aff_1$. Since $B$ is a Borel subgroup of $\PGL_2  \times \Aff_1$, this implies that, up to conjugation, we have $B  = \Aff_1 \times \Aff_1 = \{ (ax+b, cy+d), \; a,b,c,d \in \C,$    $ac \neq 0 \}$. This is again a contradiction since this group is strictly contained in $\BBB_2$. Finally, if $\pr_2(B)=\T_{0,1}$, we have $B= \T_{0,1}$ up to conjugation, by Lemma~\ref{lemma: characterisation of the groups conjugate to T_{0,1}}. This is again a contradiction, because this group is clearly not a Borel subgroup of $\PGL_2 (\C(y)) \rtimes \Aff_1$ (being strictly contained in $\BBB_2$).
\end{proof}

\begin{lemma} \label{lemma: equivalence dealing with h=(x+1,y)}
Set $h:= (x+1,y) \in \bp$.

For any $g \in \Jonq$, the following assertions are equivalent:
\begin{enumerate}
\item \label{ghg(-1) is in Aff1(C(y))}
We have $ghg^{-1} \in \Aff_1 ( \C(y) )$.
\item \label{ghg(-1) is in B2}
We have $ghg^{-1} \in \BBB_2$.
\item \label{g is in Aff1 semidirect PGL2}
We have $g \in \Aff_1 (\C (y) ) \rtimes \PGL_2$.
\end{enumerate}
\end{lemma}

\begin{proof}

\eqref{ghg(-1) is in Aff1(C(y))} $\Longrightarrow$ \eqref{ghg(-1) is in B2} This is obvious.

\eqref{ghg(-1) is in B2} $\Longrightarrow$ \eqref{g is in Aff1 semidirect PGL2} Write $g \in \Jonq = \PGL_2 (\C (y) ) \rtimes \PGL_2 $ as $g = uv $ where $u = \smat{\alpha }{\beta}{\gamma}{\delta} $  $\in \PGL_2 (\C (y) )$ and $v \in \PGL_2 $. We have
\[ ghg^{-1} = uhu^{-1} =\smat{\alpha }{\beta}{\gamma}{\delta} \smat{1 }{1}{0}{1} \smat{ \hspace{2mm} \delta }{-\beta}{-\gamma}{  \hspace{2mm} \alpha} =  \smat{\ast }{\ast}{-\gamma^2}{\ast}. \]
Hence, the condition \eqref{ghg(-1) is in B2} gives $\gamma=0$, i.e.\ $u \in \Aff_1 (\C (y) )$, and finally $ g \in \Aff_1 (\C (y) ) \rtimes \PGL_2 $.

\eqref{g is in Aff1 semidirect PGL2} $\Longrightarrow$ \eqref{ghg(-1) is in Aff1(C(y))} Write $g = uv$,  where $u \in \Aff_1 (\C(y) )$, $v \in  \PGL_2 $, and note that $ghg^{-1}  =u  h u^{-1}$ $\in \Aff_1 (\C(y) )$.
\end{proof}

\begin{theorem} \label{theorem: any Borel subgroup of PGL_2 (C (y) ) rtimes Aff_1 contains a Borel subgroup of PGL_2 (C (y) )}
Any Borel subgroup of $\PGL_2 (\C (y) ) \rtimes \Aff_1$ contains at least one Borel subgroup of $\PGL_2 (\C (y) )$.
\end{theorem}

\begin{proof}
Let $B$ be a Borel subgroup of $\PGL_2 (\C (y) ) \rtimes \Aff_1$.

\uline{Step 1.}  Let's show that $K:= B \cap \PGL_2 ( \C(y) )$ is $\C (y)$-closed in $\PGL_2 (\C (y) )$.\\
Let $\overline{K}$ be the $\C(y)$-closure of $K$ in $\PGL_2 (\C (y) )$. Note that $\overline{K}$ is normalised by $B$ (since $K$ is normalised by $B$). Hence $B':= B. \overline{K}$ is a subgroup of $\PGL_2 ( \C(y) ) \rtimes \Aff_1$. Moreover, since $B$ and $\overline{K}$ are solvable, $B'$ is also solvable (this follows from the fact that  $\overline{K}$ and $B' / \overline{K}$ are solvable; the solvability of $B' / \overline{K}$ follows from the isomorphism $B' / \overline{K} \simeq B/ (B \cap  \overline{K})$).

Let's now check that $B'$ is connected with respect to the $\bp$-topology. Let $( \overline{K}) ^{\circ}$ be the neutral connected component of $\overline{K}$ with respect to the $\C (y) $-topology on $\PGL_2 ( \C (y) )$. By Lemma~\ref{lemma: closed connected subgroup of PGL(2,C(y)) for the C(y)-topology are Bir(P2)-connected}, it is $\bp$-connected. Since there exist elements $k_1, \ldots, k_r$ of $K$ such that $\overline{K} = \bigcup_i k_i  ( \overline{K}) ^{\circ}$, it follows that $B' = B (\overline{K})^{\circ}$, and this is sufficient for showing that $B'$ is $\bp$-connected.

The Borel subgroup $B$ being contained in $B'$ which is solvable and connected, we necessarily have $B=B'$. This shows that $\overline{K}$ is contained in $B$. Hence we have $\overline{K} = K$ and we have actually proven that $K$ is $\C(y)$-closed.

Set $H:= K^{\circ}$ where the neutral connected component is taken with respect to the $\C (y) $-topology on $\PGL_2 ( \C (y) )$.

\uline{Step 2.} Let's prove that $H \neq \{ \id \}$.

Assume by contradiction that $H = \{ \id \}$. Then $K$ is a finite subgroup of $\PGL_2 ( \C(y) )$ and we have a short exact sequence
\[ 1 \to K \to B \to \pr_2 (B) \to 1 .\]

Note that $K$ is nontrivial thanks to Proposition~\ref{proposition: no Borel subgroup of PGL_2( C(y) ) rtimes Aff_1 embeds into Aff_1 via pr_2}. Since the connected group $B$ normalises the finite group $K$, it necessarily centralises it. Choosing a nontrivial element $k$ of $K$ we have $B \subseteq \Cent_{\PGL_2 ( \C (y) ) \rtimes \Aff_1 }(k)^{\circ}$. By \cite[Theorem 4.2.]{Beauville2010} and up to conjugation, one of the two following assertions holds:
\begin{enumerate}
\item \label{k= (ax,y)}
$k = \smat{a}{0}{0}{1} $ where $a \in \C \setminus \{ 0,1 \}$;
\item \label{k= iota_f}
$k = \smat{0}{f}{1}{0} $ where $f$ is a nonsquare element of $\C (y)$. \rule{0mm}{5mm} 
\end{enumerate}
In case \eqref{k= (ax,y)}, an easy computation would show that $ \Cent_{\PGL_2 ( \C (y) ) \rtimes \Aff_1}(\k)^{\circ} = \C (y)^* \rtimes \Aff_1$. In case \eqref{k= iota_f}, Proposition~\ref{proposition: normaliser of T_f in Jonq is the centraliser of iota_f} yields $\Cent_{\PGL_2 ( \C (y) ) \rtimes \Aff_1}(k)^{\circ} = \Nor_{\PGL_2 ( \C (y) ) \rtimes \Aff_1}(\T_f)^{\circ}$ and Proposition~\ref{proposition: description of N_{C* times Aff_1}(T_f)^{circ}} yields $\Nor_{\PGL_2 ( \C(y) ) \rtimes \Aff_1}(\T_f)^\circ = \T_f \rtimes  {\mathcal N}_{\C(y) ^*\rtimes \Aff_1}(\T_f)^{\circ}$. Hence, we have
\[ \Cent_{\PGL_2 ( \C (y) ) \rtimes \Aff_1}(k)^{\circ} = \T_f  \rtimes {\mathcal N}_{\C(y) ^*\rtimes \Aff_1}(\T_f)^{\circ} . \]
Recall that ${\mathcal N}_{\C(y) ^*\rtimes \Aff_1}(\T_f)^{\circ}$ is commutative (being either trivial or isomorphic to $\C^*$). Hence $\Cent_{\PGL_2 ( \C (y) ) \rtimes \Aff_1}(k)^{\circ}$ is solvable in both cases. Since it contains $B$ and is connected (by definition) it is equal to $B$. However, $\Cent_{\PGL_2 ( \C (y) ) \rtimes \Aff_1}(k)^{\circ} \cap \PGL_2 ( \C (y) )$ is infinite in both cases. A contradiction.

\uline{Step 3.} 
Since $H$ is a nontrivial closed connected solvable subgroup of $\PGL_2 (\C(y))$, we conclude by Lemma~\ref{lemma: closed connected subgroups of PGL_2(K)} that up to conjugation it is either $\T_f$ for some nonsquare element $f \in \C(y)$, or one of the following three groups: $\Aff_1 (\C(y))$, $\{ \smat{a }{0}{0}{1}, \; a \in \C(y)^*  \}$, $\{ \smat{1}{a}{0}{1}, \; a \in \C(y)  \}$.
Since $B$ normalises $K$, it also normalises $H = K^{\circ}$, i.e.\ $B$ is contained in $\Norm_{\PGL_2(\C(y))\rtimes\Aff_1}(H)$. Since $B$ is connected we even have
\[ B \subseteq \Norm_{\PGL_2(\C(y))\rtimes\Aff_1}(H)^\circ.\]
If $H = \{ \smat{1}{a}{0}{1}, \; a \in \C(y)  \}$, we get $B \subseteq \BBB_2$ by Lemma~\ref{lemma: equivalence dealing with h=(x+1,y)} -- a contradiction, because this would yield $B = \BBB_2$ and then $H = \Aff_1(\C(y))$. If $H =\{ \smat{a }{0}{0}{1}, \; a \in \C(y)^*  \}$, we start by noting that $\{1\}\rtimes\Aff_1$ normalises $H$, and it is well known that only diagonal or anti-diagonal matrices in $\PGL_2(\C(y))$ normalise $\{ \smat{a }{0}{0}{1}, \; a \in \C(y)^*  \}$. We conclude that $B$ is contained in $\{ \smat{a }{0}{0}{1}, \; a \in \C(y)^*  \}\rtimes \Aff_1$ -- a contradiction because this would yield $B= \C (y)^* \rtimes \Aff_1$,  and $B$ is strictly contained in the closed connected solvable group $\BBB_2$. Hence, we have either $H = \T_f$ or $\Aff_1 ( \C (y) )$ and this achieves the proof.
\end{proof}

In a linear algebraic group a closed connected solvable subgroup
equal to its own normaliser is not necessarily a Borel subgroup, as shown in the following example: The diagonal group $T$ in the group $B$ of upper triangular matrices of size $2$ is equal to its own normaliser, but it is not a Borel subgroup of $B$. In contrast, the following result holds:

\begin{lemma} \label{lemma: a solvable subgroup of J containing a Borel subgroup of PGL(2,C(y)) normalizes it}
If $B'$ is a solvable subgroup of $\Jonq$ containing a Borel subgroup $B$ of $\PGL_2 (\C (y))$, then $B'$ normalises $B$.
\end{lemma}

\begin{proof}
Consider the subgroup $G$ of $\Jonq$ generated by the subgroups $\varphi B \varphi^{-1}$, $\varphi \in B'$. As each $\varphi B \varphi^{-1}$ is contained in both $\PGL_2 (\C (y) )$ and $B'$, the same holds for $G$. Hence, $G$ is a solvable subgroup of $\PGL_2 (\C (y) )$ containing $B$. Since $G$ is obviously connected (being generated by connected subgroups), we have $G=B$.
\end{proof}

The following interesting result provides a description of all Borel subgroups of \linebreak $\PGL_2 ( \C(y) ) \rtimes \Aff_1$. This is the reason for including it even if it will not be used in the proof of Theorem~\ref{theorem: main theorem}.

\begin{theorem} \label{theorem: the natural bijection between the Borel subgroups of PGL_2 ( C(y) ) and the Borel subgroups of PGL_2 ( C(y) ) rtimes Aff_1}
Each Borel subgroup of $\PGL_2 (\C(y) )$ is contained in a unique Borel subgroup of $\PGL_2 ( \C(y) ) \rtimes \Aff_1$ and conversely each Borel subgroup of $\PGL_2( \C(y) ) \rtimes \Aff_1$ contains a unique Borel subgroup of $\PGL_2 (\C(y) )$. The corresponding bijection from the set of Borel subgroups of $\PGL_2 ( \C (y) )$ to the set of Borel subgroups of $\PGL_2( \C(y) ) \rtimes \Aff_1$ is the map 
\[ B \mapsto  \Nor_{\PGL_2 ( \C(y) ) \rtimes \Aff_1}(B)^{\circ} \]
and its inverse is the map
\[ B' \mapsto   [ \PGL_2 (\C (y) )  \cap B'  ]^{\circ}. \]
\end{theorem}

\begin{proof}
Let $B$ be a Borel subgroup of $\PGL_2 (\C(y) )$. If $B'$ is a Borel subgroup of \linebreak  $\PGL_2( \C(y) ) \rtimes \Aff_1$ containing $B$, then we have  $B' \subseteq \Nor_{\PGL_2 ( \C(y) ) \rtimes \Aff_1}(B)$ by \linebreak Lemma~\ref{lemma: a solvable subgroup of J containing a Borel subgroup of PGL(2,C(y)) normalizes it} and even $B' \subseteq \Nor_{\PGL_2 ( \C(y) ) \rtimes \Aff_1}(B)^{\circ}$ because $B'$ is connected. For showing that  $\Nor_{\PGL_2 ( \C(y) ) \rtimes \Aff_1}(B)^{\circ}$ is the unique Borel subgroup of $\PGL_2( \C(y) ) \rtimes \Aff_1$ containing $B$, it remains to show that this group is solvable. Actually, we will even check that $\Nor_{\PGL_2 ( \C(y) ) \rtimes \Aff_1}(B)$ is solvable. Since $D^2 ( \Nor_{\PGL_2 ( \C(y) ) \rtimes \Aff_1}(B) )$ is contained in  $\Nor_{\PGL_2 ( \C(y) )}(B)$ it is enough to check that this latter group is solvable. But up to conjugation in $\PGL_2 (\C(y) )$ we have either $B= \T_f$ for some nonsquare element $f \in \C (y)$ or $B= \Aff_1 ( \C (y) )$ (Theorem~\ref{theorem : the Borel subgroups of PGL(2,C(y))}). If $B= \T_f$, we have $[ \Nor_{\PGL_2 ( \C(y) )}(B) : B] = 2$ by Lemma~\ref{lemma: the normaliser of T_f in PGL_2(K)} and if $B= \Aff_1 ( \C (y) )$ we have $\Nor_{\PGL_2 ( \C(y) )}(B)  = B$. In both cases $\Nor_{\PGL_2 ( \C(y) )}(B)$ is actually solvable.

Moreover, any Borel subgroup of $\PGL_2 ( \C(y) ) \rtimes \Aff_1$ contains a Borel subgroup of $\PGL_2 ( \C(y) )$ (see Theorem~\ref{theorem: any Borel subgroup of PGL_2 (C (y) ) rtimes Aff_1 contains a Borel subgroup of PGL_2 (C (y) )}). Hence the map
\[ B \mapsto  \Nor_{\PGL_2 ( \C(y) ) \rtimes \Aff_1}(B)^{\circ} \]
from the set of Borel subgroups of $\PGL_2 ( \C (y) )$ to the set of Borel subgroups of \linebreak $\PGL_2( \C(y) ) \rtimes \Aff_1$ is surjective.

Let $B'$ be a Borel subgroup of $\PGL_2 ( \C(y) ) \rtimes \Aff_1$. By what has just been said, there exists at least one Borel subgroup $B$ of $\PGL_2 ( \C (y) )$ such that $B \subseteq B'$. Since $B $ is contained in the closed connected solvable subgroup $[ \PGL_2 ( \C(y) ) \cap B']^{\circ}$ of $\PGL_2 (\C (y) )$ this shows that we necessarily have $B =  [ \PGL_2 ( \C(y) ) \cap B']^{\circ}$ and this concludes the proof.
\end{proof}

\begin{lemma} \label{lemma: normaliser in Jonq of Aff_1( C(y) )}
We have $\Nor_{\Jonq}(\Aff_1(\C(y)))=\Aff_1(\C(y))\rtimes\PGL_2$.
\end{lemma}

\begin{proof}
This follows from the equality $\Nor_{\PGL_2 ( \C (y) ) }( \Aff_1 ( \C (y) ) ) = \Aff_1 ( \C (y) )$ (see \linebreak Lemma~\ref{lemma: the normaliser of Aff_1(K) in PGL_2(K)}).
\end{proof}

If $B$ is a Borel subgroup of $\PGL_2 ( \C (y) )$, we now describe more precisely the unique Borel subgroup $B'$ of $\PGL_2 ( \C(y) ) \rtimes \Aff_1$ which contains it. Recall that up to conjugation we have either $B= \Aff_1 ( \C (y) )$ or $B= \T_f$ for some nonsquare element $f \in \C (y)$(Theorem~\ref{theorem : the Borel subgroups of PGL(2,C(y))}). If $B= \Aff_1 ( \C (y) )$ it follows from Lemma~\ref{lemma: normaliser in Jonq of Aff_1( C(y) )}  that the unique Borel subgroup of $\PGL_2 ( \C(y) ) \rtimes \Aff_1$ containing $\Aff_1 ( \C (y) )$ is
\[ \Nor_{\PGL_2 ( \C(y) ) \rtimes \Aff_1} \Big( \Aff_1 ( \C (y) ) \Big) ^{\circ} =  \BBB_2. \]
If $B =\T_f$ the situation is described in Proposition~\ref{proposition: precise description of the unique Borel subgroup of PGL_2(C(y)) rtimes PGL_2 containing T_f} below which is a direct consequence of Proposition~\ref{proposition: description of N_{C* times Aff_1}(T_f)^{circ}}.

\begin{proposition} \label{proposition: precise description of the unique Borel subgroup of PGL_2(C(y)) rtimes PGL_2 containing T_f}
Let $f$ be a nonsquare element of $\C(y)$. Then the unique Borel subgroup of $\PGL_2 ( \C(y) ) \rtimes \Aff_1$ containing $\T_f$ is
\[ \Nor_{\PGL_2 ( \C(y) ) \rtimes \Aff_1}(\T_f)^\circ = \T_f \rtimes  {\mathcal N}_{\C(y) ^*\rtimes \Aff_1}(\T_f)^{\circ} . \]
\end{proposition}

\section{Some Borel subgroups of $\bp$ } \label{section: some Borel subgroups of Bir(P2)}

In this section we show that the groups listed in Theorem~\ref{theorem: main theorem} are Borel subgroups of $\bp$ (see Theorem~\ref{theorem: some Borel subgroups of Bir(P2)}). In the proof we will use the following result.

\begin{proposition} \label{proposition: a Borel subgroup of Jonq is also a Borel subgroup of Bir(P2)}
Let $B \subseteq \Jonq$ be a subgroup. Then, the following assertions are equivalent.
\begin{enumerate}
\item \label{B is a Borel subgroup of Bir(P2)}
$B$ is a Borel subgroup of $\bp$;
\item  \label{B is a Borel subgroup of Jonq}
$B$ is a Borel subgroup of $\Jonq$.
\end{enumerate}
\end{proposition}

\begin{proof}
The implication \eqref{B is a Borel subgroup of Bir(P2)} $\Longrightarrow$ \eqref{B is a Borel subgroup of Jonq} being obvious, let's prove  \eqref{B is a Borel subgroup of Jonq}  $\Longrightarrow$ \eqref{B is a Borel subgroup of Bir(P2)}. Assume that $B$ is a Borel subgroup of $\Jonq$. Up to conjugation by an element of $\Jonq$ we may assume that $B$ is a Borel subgroup of $\PGL_2 ( \C (y) ) \rtimes \Aff_1$ (see Theorem~\ref{theorem: any closed connected solvable of Jonq is conjugate to a subgroup of PGL_2(C(y)) rtimes Aff_1}). Let now $B'$ be a closed connected solvable subgroup of $\bp$ containing $B$. We want to show that $B=B'$. By Theorem~\ref{theorem: a closed connected solvable subgroup of Bir(P2) is conjugate to a subgroup of Jonq}, there exists $g \in \bp$ such that $g B' g^{-1} \subseteq \Jonq$ and by Theorem~\ref{theorem: any Borel subgroup of PGL_2 (C (y) ) rtimes Aff_1 contains a Borel subgroup of PGL_2 (C (y) )}, $B$ contains a Borel subgroup $B''$ of $\PGL_2 ( \C (y) )$. Since $B''$ is conjugate in $\PGL_2 ( \C (y) )$ to either $\Aff_1 ( \C (y) )$ or $\T_f$ for some nonsquare element $f$ of $\C (y)$ (see Theorem~\ref{theorem : the Borel subgroups of PGL(2,C(y))}), it contains a Jonqui\`eres twist. Hence the inclusion $g B'' g^{-1} \subseteq \Jonq$ and Lemma~\ref{lemma: unique rational fibration preserved and a property of some conjugants} show, in combination with \cite[Lemma~4.5]{DillerFavre2001}, 
that $g \in \Jonq$. This proves that $B'$ is actually contained in $\Jonq$ and finally the inclusion $B \subseteq B'$ implies that $B=B'$.
\end{proof}

As announced, we now describe some Borel subgroups of $\bp$. Later on, in Theorem~\ref{theorem: all Borel subgroups of Bir(P2)}, we will show that each Borel subgroup of $\bp$ is actually conjugate to one of them.

\begin{theorem} \label{theorem: some Borel subgroups of Bir(P2)}
The following subgroups are Borel subgroups of $\bp$:
\begin{enumerate}
\item \label{Borel: B2}
$\BBB_2$;
\item \label{Borel: Ty semidirect T}
$\T_y \rtimes \T_{1,2}$;
\item \label{Borel: Tf}
$\T_f$ where $f$ is a nonsquare element of $\C (y)$ such that the genus $\g$ of the curve $x^2=f(y)$ satisfies $\g \ge 1$.
\end{enumerate}
\end{theorem}

\begin{proof}
By Proposition~\ref{proposition: a Borel subgroup of Jonq is also a Borel subgroup of Bir(P2)}, it is enough to prove that all these groups are Borel subgroups of $\Jonq$. By Proposition~\ref{proposition: pre-Borel subgroups of closed subgroups of Bir(P2) are contained in Borel subgroups} any Borel subgroup $B$ of $\PGL_2 ( \C (y) )$ is contained in a Borel subgroup $B'$ of $\Jonq$ and by Lemma~\ref{lemma: a solvable subgroup of J containing a Borel subgroup of PGL(2,C(y)) normalizes it} we have $B' \subseteq \Nor_{\Jonq}(B)^\circ$.

\noindent \eqref{Borel: B2} If $B = \Aff_1 ( \C (y) )$, this yields $B \subseteq B' \subseteq B \rtimes \PGL_2$ (Lemma~\ref{lemma: normaliser in Jonq of Aff_1( C(y) )}). By Theorem~\ref{theorem: any closed connected solvable of Jonq is conjugate to a subgroup of PGL_2(C(y)) rtimes Aff_1} and up to conjugation by an element of $\{ 1 \} \rtimes \PGL_2 \subseteq \PGL_2 ( \C(y) ) \rtimes \PGL_2 = \Jonq$, we may moreover assume that we have $B' \subseteq \PGL_2 ( \C (y) ) \rtimes \Aff_1$. Using the equality
\[ (B \rtimes \PGL_2) \cap ( \PGL_2 ( \C (y) ) \rtimes \Aff_1) = B \rtimes \Aff_1\] we obtain
\[ B \subseteq B' \subseteq B \rtimes \Aff_1 .\]
Since $B \rtimes \Aff_1$ is solvable, we get $B' = B \rtimes \Aff_1 = \Aff_1 ( \C (y) ) \rtimes \Aff_1= \BBB_2$ and we have actually proven that $\BBB_2$ is a Borel subgroup of $\Jonq$. Note that this proof is different from the one given by Popov in \cite{Popov2017}.

\noindent \eqref{Borel: Ty semidirect T}
If $B= \T_y$, we have $B'  \subseteq \T_y \rtimes \T_{1,2}$ by Proposition~\ref{proposition: connected component of the normaliser of Tf in Jonq}\eqref{Nor_{Jonq}(T_y )^{circ} = T_t rtimes T_{1,2}}. Since $\T_y \rtimes \T_{1,2}$ is solvable, this gives $B' = \T_y \rtimes \T_{1,2}$.

\noindent  \eqref{Borel: Tf}
If $B:=\T_f$ with $\g \ge 1$, we have $B'  \subseteq \T_f$ by Proposition~\ref{proposition: connected component of the normaliser of Tf in Jonq}\eqref{Nor_{Jonq}(T_f )^{circ} = T_f if g  is at least 1} and this gives $B' = \T_f$.
\end{proof}

\begin{lemma} \label{lemma: estimates of rk(K semidirect H)}
Let $G,H,K$ be closed subgroups of $\bp$ such that $G = K \rtimes H$, then we have $\max ( \rk H, \rk K) \le \rk G \le \rk H + \rk K$.
\end{lemma}

\begin{proof}
Since $H,K$ are contained in $G$, the first inequality is obvious. Let now $T$ be a torus in $G$. The short exact sequence $1 \to K \to G \xrightarrow{\pi}  H \to 1$ induces the short exact sequence $1 \to T' \to T \to T'' \to 1$ where $T' := T \cap K$ and $T'' := \pi (T)$. Since $(T')^{\circ}$ and $T''$ are connected algebraic groups consisting of semisimple elements, they are tori (see \cite[Exercise 21.2, page 137]{Humphreys1975}). Hence, we get $\dim T = \dim T' + \dim T'' = \dim (T')^{\circ} + \dim T'' \le \rk K + \rk H$ and the conclusion follows.
\end{proof}

\begin{remark} \label{remark: lengths and ranks of Borel subgroups}
The three kinds of examples given in Theorem~\ref{theorem: some Borel subgroups of Bir(P2)} are non-isomorphic, hence non-conjugate, since the derived lengths of $\BBB_2$, $\T_y \rtimes \T_{1,2}, \T_f$ are respectively 4,2,1 (see Proposition~\ref{proposition: the derived length of Bn} for the derived length of $\BBB_2$).

Let's also note that the ranks of these groups are respectively 2,1,0: We know from Lemma~\ref{lemma: rk(T_f) = 0} that  $\rk (\T_f ) =0$; Lemma~\ref{lemma: estimates of rk(K semidirect H)} shows that $\rk (\T_y \rtimes \T_{1,2} ) = 1$; the equality $\rk (\BBB_2) = 2$ is obvious.
\end{remark}

\section{All Borel subgroups of $\bp$ } \label{section: all Borel subgroups of Bir(P2)}

In Theorem~\ref{theorem: some Borel subgroups of Bir(P2)} we have given some examples of Borel subgroups of $\bp$. We now prove that up to conjugation there are no others.

\begin{theorem} \label{theorem: all Borel subgroups of Bir(P2)}
Any Borel subgroup of $\bp$ is necessarily conjugate to one of the following subgroups:
\begin{enumerate}
\item \label{all Borel: B2}
$\BBB_2$;
\item \label{all Borel: Ty semidirect T}
$\T_y \rtimes \T_{1,2}$;
\item \label{all Borel: Tf}
$\T_f$ where $f$ is a nonsquare element of $\C (y)$ such that the genus $\g$ of the curve $x^2=f(y)$ satisfies $\g \ge 1$.
\end{enumerate}
\end{theorem}

\begin{proof}
Let $B'$ be a Borel subgroup of $\bp$. Up to conjugation, it is a Borel subgroup of $\PGL_2 ( \C (y) ) \rtimes \Aff_1$  (cf. Theorem~\ref{theorem: any closed connected solvable subgroup of Bir(P2) is conjugate to a subgroup of PGL_2(C(y)) rtimes Aff_1}). By Theorem~\ref{theorem: any Borel subgroup of PGL_2 (C (y) ) rtimes Aff_1 contains a Borel subgroup of PGL_2 (C (y) )} $B'$ contains a Borel subgroup $B$ of $\PGL_2 ( \C (y) )$ and by Lemma~\ref{lemma: a solvable subgroup of J containing a Borel subgroup of PGL(2,C(y)) normalizes it} we have $B' \subseteq  \Nor_{\PGL_2 ( \C (y) ) \rtimes \Aff_1}(B) ^{\circ}$.  Moreover, up to conjugation, one of the following cases occurs:
\begin{enumerate}
\item 
$B= \Aff_1 ( \C (y) )$;
\item
$B = \T_y$;
\item
$B = \T_f$ for some nonsquare element $f$ of $\C (y)$ such that the genus $\g$ of the curve $x^2=f(y)$ satisfies $\g \ge 1$.
\end{enumerate}
Then Lemma~\ref{lemma: normaliser in Jonq of Aff_1( C(y) )}, Proposition~\ref{proposition: connected component of the normaliser of Tf in Jonq}\eqref{Nor_{Jonq}(T_y )^{circ} = T_t rtimes T_{1,2}}, and Proposition~\ref{proposition: connected component of the normaliser of Tf in Jonq}\eqref{Nor_{Jonq}(T_f )^{circ} = T_f if g  is at least 1} prove that $B'$ is contained respectively in $\BBB_2$, $\T_y \rtimes \T_{1,2}$, and $\T_f$. Since these groups are solvable, this shows that $B'$ is equal respectively to $\BBB_2$, $\T_y \rtimes \T_{1,2}$, and $\T_f$.
\end{proof}

Let's now give the proof of our main theorem.

\begin{proof}[Proof of Theorem~\ref{theorem: main theorem}]
This follows from Theorems~\ref{theorem: some Borel subgroups of Bir(P2)} \&~\ref{theorem: all Borel subgroups of Bir(P2)}, Remark~\ref{remark: lengths and ranks of Borel subgroups}, and Proposition~\ref{proposition: equivalent conditions for T_f and T_g to be conjugate in Bir(P2)}.
\end{proof}

Recall that a Borel subgroup of a linear algebraic group is equal to its own normaliser
\cite[Theorem 23.1, page 143]{Humphreys1975}. We end this section by showing an analogous result in $\bp$ (see Proposition~\ref{proposition: a variant of Borel normaliser theorem for Bir(P2)} below). We begin with the following lemma.

\begin{lemma} \label{lemma: B2 is equal to its own normaliser in Bir(P2)}
The group $\BBB_2$ is equal to its own normaliser in $\bp$.
\end{lemma}

\begin{proof}
By Lemma~\ref{lemma: unique rational fibration preserved and a property of some conjugants}, we have $\Nor_{\bp} (\BBB_2) = \Nor_{\Jonq} (\BBB_2)$ (since $\BBB_2$ contains at least one Jonqui\`eres twist). By Lemma~\ref{lemma: equivalence dealing with h=(x+1,y)}, we have $\Nor_{\Jonq} (\BBB_2) \subseteq \Aff_1 (\C (y) ) \rtimes \PGL_2 $. We conclude by using the second projection $\pr_2 \colon \Jonq \to \PGL_2 $ and the equality $\Nor_{\PGL_2} ( \Aff_1) = \Aff_1$ (besides being obvious this latter equality is also Lemma~\ref{lemma: the normaliser of Aff_1(K) in PGL_2(K)} or follows from  the Borel normaliser theorem applied with the Borel subgroup $\Aff_1$ of $\PGL_2$).
\end{proof}

\begin{proposition} \label{proposition: a variant of Borel normaliser theorem for Bir(P2)}
Let $B$ be a Borel subgroup of $\bp$. Then we have
\[ B= \Nor_{\bp} (B) ^{\circ} .\]
\end{proposition}

\begin{proof}
Up to conjugation, one of the following cases occurs:
\begin{enumerate}
\item \label{Borel normaliser theorem: B2}
$B= \BBB_2$;
\item \label{Borel normaliser theorem: T_y rtimes T_{1,2}}
$B= \T_y \rtimes \T_{1,2}$;
\item \label{Borel normaliser theorem: T_f}
$B = \T_f$ where $f$ is a nonsquare element of $\C (y)$ such that the genus $\g$ of the curve $x^2=f(y)$ satisfies $\g \ge 1$.
\end{enumerate}
Note that $\Nor_{\bp} (B) =\Nor_{\Jonq} (B)$ in each of these cases by Lemma~\ref{lemma: unique rational fibration preserved and a property of some conjugants}.

In case \eqref{Borel normaliser theorem: B2}, the conclusion follows from Lemma~\ref{lemma: B2 is equal to its own normaliser in Bir(P2)}.

In case \eqref{Borel normaliser theorem: T_y rtimes T_{1,2}}, we have  $\Nor_{\Jonq} (B) \subseteq \Nor_{\Jonq} (\DDD (B) )$ because $\DDD (B) =\overline{ D(B) }$  and $D(B)$ is a characteristic subgroup of $B$. Since $\DDD  ( \T_y \rtimes \T_{1,2} ) = \T_y$ we obtain $\Nor_{\Jonq} (B)^{\circ} \subseteq \Nor_{\Jonq} ( \T_y)^{\circ} =$  $\T_y \rtimes \T_{1,2}$ (see Proposition~\ref{proposition: connected component of the normaliser of Tf in Jonq}\eqref{Nor_{Jonq}(T_y )^{circ} = T_t rtimes T_{1,2}}) and the conclusion follows.

The case~\eqref{Borel normaliser theorem: T_f} follows from Proposition~\ref{proposition: connected component of the normaliser of Tf in Jonq}\eqref{Nor_{Jonq}(T_f )^{circ} = T_f if g  is at least 1}.
\end{proof}

\begin{remark}
Note that the usual Borel normaliser theorem $B= \Nor_{G} (B)$ does not hold when $G= \bp$ and $B = \T_f$ for some nonsquare element $f$ of $\C (y)$ such that the genus $\g$ of the curve $x^2=f(y)$ satisfies $\g \ge 1$. Actually, in this case $\Nor_{\bp} (\T_f)$ is strictly larger than $\T_f$ because it contains $\T_f \rtimes \langle \smat{1}{\hspace{2mm} 0}{0}{-1} \rangle$ (see Lemma~\ref{lemma: the normaliser of T_f in PGL_2(K)}).
\end{remark}

\appendix

\section{Computation of the derived length of $\mathcal B_n$}
\label{section: computation of the derived length of Bn}

In this appendix we give the proof of Proposition~\ref{proposition: the derived length of Bn}, stated in the introduction.

\begin{proof}[Proof of Proposition \ref{proposition: the derived length of Bn}]
Set $U_k := D^{2n-2k} ( \mathcal B_n )$ for $0 \le k \le n$ and $V_k := D (U_k)$ for $1 \le k \le n$. 

We could easily check by induction that $U_k$ is contained in the group of elements $f=$  $(f_1, \ldots, f_n)$ in $\mathcal B_n$ satisfying $f_i =x _i$ for $i > k$ and that $V_k$ is contained in the group of elements $f= (f_1, \ldots, f_n)$ in $\mathcal B_n$ satisfying $f_i =x _i$ for $i > k$ and $f_k = x_k + b_k$ with $b_k \in \C (x_{k+1},\ldots, x_n)$. This implies $U_0 = \{ \id \}$. Hence the derived length of $\mathcal B_n$ is at most $2n$.

Conversely, we will now prove that this derived length is at least $2n$, i.e.\ that $V_1 \neq \{ \id \}$. For this, we introduce the following notation.
For each $k \in \{ 1, \ldots, n \}$, each nonzero element $a \in \C (x_{k+1},\ldots, x_n)$, and each element $b \in \C (x_{k+1},\ldots, x_n)$, we define the dilatation $d(k, a) \in \Bir (\p^n)$ and the elementary transformation $e(k, b) \in  \Bir (\p^n)$ by
\[ d(k, a) = (g_1, \ldots, g_n) \quad \text{and} \quad e(k, b) = (h_1, \ldots, h_n), \]
where $g_k = a x_k$, $h_k = x_k + b$ and $g_i=h_i = x_i$ for $i \neq k$.

If $i \in \{ 1, \ldots, n-1 \}$ and if $G$ is a subgroup of $\Bir (\p^n)$, the properties $(D_i)$ and $(E_i)$ for $G$ are defined in the following way:
\begin{enumerate}
\item[$(D_i)$] $\exists \, a \in \C (x_{i+1}) \setminus \C $ such that $d(i,a ) \in G$;
\item[$(E_i)$] $\forall \, b \in \C (x_{i+1} )$ we have $e(i,b) \in G$.
\end{enumerate}
For $i=n$, the  properties $(D_n)$ and $(E_n)$ are defined in the following slightly different way:
\begin{enumerate}
\item[$(D_n)$] $\exists \, a \in \C \setminus \{ 0,1 \} $ such that $d(n, a) \in G$;
\item[$(E_n)$] $\forall \, b \in \C $ we have $e(n, b) \in G$.
\end{enumerate}

We will prove  by decreasing induction on $k$ that for each $k \in \{ 1, \ldots, n \}$ the two following assertions hold:
\begin{enumerate}
\item[$(a_k)$] The group $U_k$ satisfies $(E_i)$ and $(D_i)$ for $i \in \{ 1, \ldots, k \}$.
\item[$(b_k)$]  The group $V_k$ satisfies $(E_i)$ for $i \in \{ 1, \ldots, k \}$ and $(D_i)$ for $i \in \{1 , \ldots, k-1 \}$.
\end{enumerate}

We will use two obvious identities. The first one is
\begin{equation} [  d(i,a), e(i,b)  ] = e(i, b (a-1) )  \label{commutator-[d(i), e(i)]} \end{equation}
where we have either $i \in \{ 1, \ldots, n-1 \}$, $a \in \C (x_{i+1} ) \setminus \{ 0 \}$, $b \in \C (x_{i+1} )$, 
or $i =n$, $a \in \C ^*$, $b \in \C$.

The second one is
\begin{equation} [ d(i,a), e(i+1,c) ] = d \left( i, \frac{a(x_{i+1} ) }{a( x_{i+1} - c)} \right) \label{commutator-[d(i), e(i+1)]} \end{equation}
where $i \in \{ 1, \ldots, n-1 \}$, $a \in \C (x_{i+1} ) \setminus \{ 0 \}$, $c \in \C$.

The identity \eqref{commutator-[d(i), e(i)]} implies that if a subgroup $G$ of $\Bir (\p^n)$ satisfies $(D_i)$ and $(E_i)$, then its derived subgroup $D^1(G)$ also satisfies $(E_i)$.

Before making use of the identity \eqref{commutator-[d(i), e(i+1)]}, let's check that if $a \in \C (x ) $ is nonconstant and $c \in \C$ is nonzero, then $a(x) / a(x-c)$ is nonconstant. Otherwise, we would have $a(x) = \lambda a(x-c)$ for some $\lambda \in \C^*$. This implies that the union $U(a)$ of the zeros and  poles of $a$ is invariant by translation by $c$. Since $U(a)$ is finite, it must be empty, proving that $a$ is constant. A contradiction.

It follows from \eqref{commutator-[d(i), e(i+1)]} and the last observation that if a subgroup $G$ of $\Bir (\p^n)$ satisfies $(D_i)$ and $(E_{i+1})$, then  $D^1(G)$ also satisfies $(D_i)$.

We are now ready for the induction. We begin by noting that the hypothesis $(a_n)$ is obviously satisfied. It is then enough to observe that we have $(a_k) \Longrightarrow (b_k)$ for each $k \in \{1, \ldots, n \}$ and $(b_k) \Longrightarrow (a_{k-1} )$ for each $k \in \{2, \ldots, n \}$.
\end{proof}

\section{More on Borel subgroups of $\PGL_2( \C(y) ) \rtimes \Aff_1$} \label{section: Borel subgroups of PGL_2(C(y)) rtimes PGL_2}

We have seen in Theorem~\ref{theorem: the natural bijection between the Borel subgroups of PGL_2 ( C(y) ) and the Borel subgroups of PGL_2 ( C(y) ) rtimes Aff_1} that there is a natural bijection between the Borel subgroups of $\PGL_2( \C(y) )$ and the Borel subgroups of $\PGL_2( \C(y) ) \rtimes \Aff_1$. Hence the interplay between $\PGL_2( \C(y) )$ and $\PGL_2( \C(y) ) \rtimes \Aff_1$ is very straightforward in this respect. Analogously the interplay between the set of conjugacy classes of Borel subgroups of $\Jonq$ and the set of conjugacy classes of Borel subgroups of $\bp$ is very clear: Any Borel subgroup of $\Jonq$ is a Borel subgroup of $\bp$ (Proposition~\ref{proposition: a Borel subgroup of Jonq is also a Borel subgroup of Bir(P2)}); any Borel subgroup of $\bp$ is conjugate to a Borel subgroup of $\Jonq$ (Theorem~\ref{theorem: a closed connected solvable subgroup of Bir(P2) is conjugate to a subgroup of Jonq}); and finally two Borel subgroups of $\Jonq$ are conjugate in $\Jonq$ if and only if they are conjugate in $\bp$ (Lemma~\ref{lemma: unique rational fibration preserved and a property of some conjugants} and Theorem~\ref{theorem: all Borel subgroups of Bir(P2)}). Hence, the map sending the conjugacy class of a Borel subgroup $B$ of $\Jonq$ to its conjugacy class in $\bp$ is a bijection from the set of conjugacy classes of Borel subgroups of $\Jonq$ to the set of  conjugacy classes of Borel subgroups of $\bp$.

The interplay between $\PGL_2( \C(y) ) \rtimes \Aff_1$ and $\Jonq = \PGL_2( \C(y) ) \rtimes \PGL_2$ is more intricate: We now give an example of a Borel subgroup $B$ of $\PGL_2( \C(y) ) \rtimes \Aff_1$
which is no longer a Borel subgroup of $\Jonq$. To put it differently,  $B$ is conjugate in $\Jonq$ to a subgroup $B'$ of $\PGL_2( \C(y) ) \rtimes \Aff_1$ which is no longer a Borel subgroup of $\PGL_2( \C(y) ) \rtimes \Aff_1$! This kind of situation can of course not occur for linear algebraic groups: Let $H$ be a closed subgroup of a linear algebraic group $G$ and let $B,B'$ be two subgroups of $H$ which are conjugate in $G$. Then, $B$ is a Borel subgroup of $H$ if and only if $B'$ is a Borel subgroup of $H$ (if $B,B'$ are closed, it is enough to note that they have the same dimension).

\begin{example} \label{example: a Borel subgroup of PGL_2(C(y)) rtimes PGL_2 which is not a Borel subgroup of Bir(P2)}
The subgroups $\T_{y(y-1)}$ and $ \T_y$ of $\PGL_2 ( \C (y) )$ are conjugate in $\Jonq$ since the curves $x^2=  y(y-1)$ and $x^2 = y $ both have genus 0. However, since the odd supports $\So (y (y-1) )$ and $\So (y )$ are respectively equal to $\{ 0, 1 \}$ and $ \{ 0,  \infty \}$, Proposition~\ref{proposition: precise description of the unique Borel subgroup of PGL_2(C(y)) rtimes PGL_2 containing T_f} asserts that $\T_{y(y-1)}$ is a Borel subgroup of $\PGL_2 ( \C(y) ) \rtimes \Aff_1$ but that $\T_y$ is not. In fact, the unique Borel subgroup of  $\PGL_2 ( \C(y) ) \rtimes \Aff_1$ containing $\T_y$ is $\T_y \rtimes \T_{1,2}$; see Theorem~\ref{theorem: all Borel subgroups of Bir(P2)}. This also shows that even if  $\T_{y(y-1)}$ is a Borel subgroup of $\PGL_2 ( \C(y) ) \rtimes \Aff_1$, it is no longer a Borel subgroup of $\Jonq$.
\end{example}

\textbf{Acknowledgements.} We warmly thank J\'er\'emy Blanc for interesting discussions related to the subject of this paper.


\begin{thebibliography}{22}

\bibitem{BayleBeauville2000}
\textit{L.~ Bayle} and \textit{A.~Beauville},
Birational involutions of $\p^2$,
Kodaira's issue. Asian J. Math. \textbf{4} (2000), no.\ 1, 11--17.

\bibitem{Beauville2010}
\textit{A.~Beauville},
Finite subgroups of $\PGL_2(K)$ (English summary),
in: Vector bundles and complex geometry, Contemp. Math. \textbf{522}, Amer. Math. Soc., Providence, RI (2010), 23--29.

\bibitem{BerestEshmatovEshmatov2016}
\textit{Y.~Berest}, \textit{A.~Eshmatov} and \textit{F.~Eshmatov},
Dixmier groups and Borel subgroups,
Adv. Math. \textbf{286} (2016), 387--429.

\bibitem{Bialynicki-Birula1966}
\textit{A.~S.~Bia{\l}ynicki-Birula},
Remarks on the action of an algebraic torus on $k^n$, (Russian summary)
Bull. Acad. Polon. Sci. S\'er. Sci. Math. Astronom. Phys. \textbf{14} (1966), 177--181.



\bibitem{Blanc2011}
\textit{J.~Blanc},
Elements and cyclic subgroups of finite order of the Cremona group (English summary),
Comment. Math. Helv. \textbf{86} (2011), no.\ 2, 469--497

\bibitem{BlancFurter2013}
\textit{J.~Blanc} and \textit{J.-P.~Furter},
Topologies and structures of the Cremona groups, 
Ann. of Math. \textbf{178} (2013), no.\ 3, 1173--1198.


\bibitem{BlancFurter2018}
\textit{J.~Blanc} and \textit{J.-P.~Furter},
Length in the Cremona group,
Ann. H. Lebesgue \textbf{2} (2019), 187--257.


\bibitem{Borel1991}
\textit{A.~Borel},
Linear algebraic groups. Second edition,
Graduate Texts in Mathematics, \textbf{126}. Springer-Verlag, New York, 1991.


\bibitem{Bourbaki1970}
\textit{N.~Bourbaki},
\'El\'ements de math\'ematique. Th\'eorie des ensembles,
(French) Hermann, Paris 1970.
 

\bibitem{CerveauDeserti2012}
\textit{D.~Cerveau} and \textit{J.~D\'eserti},
Centralisateurs dans le groupe de Jonqui\`eres (French) [Centralizers in the Jonqui\`eres group],
Michigan Math. J. \textbf{61} (2012), no. 4, 763--783.


\bibitem{deFernex2004}  
\textit{T.~de Fernex},
On planar Cremona maps of prime order,
Nagoya Math. J. \textbf{174} (2004), 1--28.


\bibitem{Demazure1970}  
\textit{M.~Demazure},
Sous-groupes alg\'ebriques de rang maximum du groupe de Cremona,
Ann. Sci. \'Ecole Norm. Sup. (4) \textbf{3} (1970), 507-588.


\bibitem{DillerFavre2001}
\textit{J.~Diller} and \textit{C.~Favre},
Dynamics of bimeromorphic maps of surfaces (English summary),
Amer. J. Math. \textbf{123} (2001), no. 6, 1135--1169.






\bibitem{FarkasKra1992}
\textit{H.~M.~Farkas} and \textit{I.~Kra},
Riemann surfaces. Second edition. Graduate Texts in Mathematics,  \textbf{71}. Springer-Verlag, New York, 1992.


\bibitem{FurterKraft2018}
\textit{J.-P.~Furter} and \textit{H.~Kraft},
On the geometry of automorphism groups of affine varieties,
preprint 2018, \url{https://arxiv.org/abs/1809.04175}.


\bibitem{FurterPoloni2018}
\textit{J.-P.~Furter} and \textit{P.-M.~Poloni},
On the maximality of the triangular subgroup,
Ann. Inst. Fourier (Grenoble) \textbf{68} (2018), no. 1, 393--421.


\bibitem{Hartshorne77}
\textit{R.~Hartshorne},
Algebraic geometry,
Graduate Texts in Mathematics, No. \textbf{52}. Springer-Verlag, New York-Heidelberg, 1977.


\bibitem{Humphreys1975}
\textit{J.~E.~Humphreys},
Linear algebraic groups,
Graduate Texts in Mathematics, No. \textbf{21}. Springer-Verlag, New York-Heidelberg, 1975.

\bibitem{Mumford2008}
\textit{D.~Mumford},
Abelian varieties. With appendices by C. P. Ramanujam and Yuri Manin. Corrected reprint of the second (1974) edition. Tata Institute of Fundamental Research Studies in Mathematics, \textbf{5}. Published for the Tata Institute of Fundamental Research, Bombay; by Hindustan Book Agency, New Delhi, 2008. 






\bibitem{Popov2013}
\textit{V.~L.~Popov},
Tori in the Cremona groups. (Russian. Russian summary)
Izv. Ross. Akad. Nauk Ser. Mat. 77 (2013), no. 4, 103--134; translation in Izv. Math. 77 (2013), no. 4, 742--771. 


\bibitem{Popov2017}
\textit{V.~Popov},
Borel Subgroups of Cremona groups,
Mathematical Notes, 2017, vol. \textbf{102}, No. 1, pp.\ 60--67. \copyright ~Pleiades Publishing, Ltd., 2017.
Original Russian Text \copyright ~ V.~L.~Popov, 2017, published in Matematicheskie Zametki, 2017, Vol. \textbf{102}, no. 1, pp. 72--80.



\bibitem{Serre2010}
\textit{J.-P.~Serre},
Le groupe de Cremona et ses sous-groupes finis, 
S\'eminaire Bourbaki. Volume 2008/2009. Ast\'erisque no. \textbf{332} (2010), Exp.\ No.\ 1000, vii, 75--100.

\bibitem{Umemura1982b}
\textit{H.~Umemura},
On the maximal connected algebraic subgroups of the Cremona group I,
Nagoya Math. J. \textbf{88} (1982), 213--246.

\bibitem{Urech2018}
\textit{C.~Urech},
Subgroups of elliptic elements of the Cremona group,
J. reine angew. Math. \textbf{770} (2021), 27--57.

\bibitem{Wehrfritz1973}
\textit{B.~A.~F.~Wehrfritz},
Infinite linear groups. An account of the group-theoretic properties of infinite groups of matrices,
Ergebnisse der Mathematik und ihrer Grenzgebiete, Band \textbf{76}. Springer-Verlag, New York-Heidelberg, 1973.

\end{thebibliography}
\end{document}